\documentclass[11pt]{article}

\usepackage[letterpaper,margin=1in]{geometry}
\usepackage{amsmath,amssymb,amsthm,mathtools,bm}
\usepackage{enumitem}
\usepackage{microtype}
\usepackage{xcolor}
\usepackage[colorlinks=true,linkcolor=blue,citecolor=blue,urlcolor=blue]{hyperref}

\allowdisplaybreaks
\numberwithin{equation}{section}

\newtheorem{theorem}{Theorem}[section]
\newtheorem{lemma}[theorem]{Lemma}
\newtheorem{proposition}[theorem]{Proposition}

\newtheorem{conjecture}[theorem]{Conjecture}
\newtheorem{claim}{Claim}
\theoremstyle{definition}

\newtheorem{problem}[theorem]{Problem}
\theoremstyle{remark}
\newtheorem{remark}[theorem]{Remark}

\newcommand{\up}{\uparrow}
\newcommand{\symdiff}{\mathbin{\triangle}}

\title{\bf \Large }
\date{\today }
\title{{\bf \Large 
Rigidity and stability for biased cross-intersecting families}\footnote{ Lihua Feng was supported by the NSFC (Nos. 12271527 and 12471022) and NSF of Qinghai Province (No. 2025-ZJ-902T). E-mail addresses: \url{wuyjmath@163.com} (Y. Wu), \url{fenglh@163.com} (L. Feng). }
\author{
{\small  Yongjiang Wu,\ \ Lihua Feng\footnote{Corresponding author}
}\\[2mm]
\small School of Mathematics and Statistics, HNP-LAMA, Central South University\\
 \small Changsha, Hunan, 410083, China\\ 
}}

\begin{document}
\maketitle
\begin{abstract}
Let $\mathbf p=(p_1,\ldots,p_n)$ and $\mathbf q=(q_1,\ldots,q_n)$ belong to $(0,1/2]^n$, and let $\mu_{\mathbf p}$ and $\mu_{\mathbf q}$ be the associated measures on $2^{[n]}$. Suppose that $p_1q_1=\max_{i\in[n]}p_iq_i$. We prove that every pair of cross-intersecting families $\mathcal A,\mathcal B\subseteq2^{[n]}$ satisfies the sharp inequality $\mu_{\mathbf p}(\mathcal A)\mu_{\mathbf q}(\mathcal B)\leq p_1q_1$. This  confirms a  conjecture of Suda, Tanaka and Tokushige [Math. Program. 166 (2017) 113--130]. We also determine all equality cases. When $p_1q_1<1/4$, equality is attained only when both families consist of all subsets containing the same product-maximizing coordinate. At the endpoint $p_1q_1=1/4$, we identify precisely the additional extremal pairs, which are induced by half-sized increasing families on the coordinates satisfying $p_i=q_i=1/2$.

We further resolve the remaining conjecture from the same paper by proving a dimension-free stability theorem. Assume that the first coordinate has maximum probability under both measures and that $p_1,q_1<1/2$. If $\mu_{\mathbf p}(\mathcal A)\mu_{\mathbf q}(\mathcal B)\geq(1-\varepsilon)p_1q_1$, then there exists a coordinate $j$ such that both $\mathcal A$ and $\mathcal B$ are within $c(p_1,q_1)\varepsilon$, in their respective measures, of the family of all subsets containing $j$. This improves the conjectured $O(\sqrt{\varepsilon})$ bound to a linear one. The main new ingredient in the sharp measure theorem is a log-odds interpolation combined with induction on coordinate sections, while stability follows from a semidefinite estimate and a one-coordinate approximation theorem.
\end{abstract}



\tableofcontents
\section{Introduction}
\label{sec:introduction}

 For a positive integer $n$, write $[n]=\{1,\ldots,n\}$ and  let $2^{[n]}$ be the power set of $[n]$. A family $\mathcal{A}\subseteq 2^{[n]}$ is called \textit{intersecting} if $A\cap A'\neq\emptyset$ for all $A,A'\in\mathcal{A}$. More generally, two families $\mathcal{A},\mathcal{B}\subseteq 2^{[n]}$ are called \textit{cross-intersecting} if
$
 A\cap B\neq\emptyset
$
for all $A\in\mathcal{A}$ and $B\in\mathcal{B}$.

Intersection problems constitute one of the foundational threads of extremal set theory. The classical starting point is the Erd\H{o}s--Ko--Rado theorem \cite{ErdosKoRado}, which states that if $n\geq 2k$ and $\mathcal{A}\subseteq\binom{[n]}{k}$ is intersecting, then
\begin{equation*}
 |\mathcal{A}|\leq\binom{n-1}{k-1}.
\end{equation*}
For $n>2k$, equality holds if and only if $\mathcal{A}= \left\{ A \in \binom{[n]}{k} : x \in A \right\}$ for some $x\in[n]$.
The cross-intersecting extension was initiated by Pyber \cite{Pyber} and subsequently sharpened by Matsumoto and Tokushige \cite{MatsumotoTokushige}. Specifically, if
$
 \mathcal{A}\subseteq\binom{[n]}{k},
 \mathcal{B}\subseteq\binom{[n]}{\ell}
$
are cross-intersecting and $n \geq\max\{2k,2\ell\}$, then
$$
 |\mathcal{A}|\,|\mathcal{B}|
 \leq
 \binom{n-1}{k-1}\binom{n-1}{\ell-1}.
$$
Moreover, if  $n >\max\{2k,2\ell\}$, then equality holds if and only if\ there exists $x\in[n]$ such that
 $\mathcal{A}= \left\{ A \in \binom{[n]}{k} : x \in A \right\}, \mathcal{B}= \left\{ B \in \binom{[n]}{\ell} : x \in B \right\}$.  These results initiated a broad research programme  aimed at establishing sharp bounds, extremal configurations and stability results for set systems under various restrictions; we refer to  \cite{FranklTokushigeSurvey,FranklTokushigeBook} for comprehensive surveys.
 
A natural weighted counterpart of the uniform setting is obtained by replacing cardinality with probability. For a probability vector
$
 \mathbf p=(p_1,\ldots,p_n)\in(0,1)^n,
$
the \textit{$\mathbf{p}$-biased measure} of a family $\mathcal{A}\subseteq 2^{[n]}$ is defined by
$$
 \mu_{\mathbf p}(\mathcal{A})
 =
 \sum_{A\in\mathcal{A}}
 \prod_{i\in A}p_i
 \prod_{i\in[n]\setminus A}(1-p_i).
$$
Equivalently, if $X\subseteq[n]$ is formed by including each coordinate independently, with coordinate $i$ included with probability $p_i$, then
\begin{equation*}
 \mu_{\mathbf p}(\mathcal{A})=\mathbb{P}(X\in\mathcal{A}).
\end{equation*}
The vector $\mathbf p$ thus allows different coordinates to carry different weights. 
In the special case where all coordinates have the same probability, say $p_i=p$ for every $i$, the measure reduces to the usual \textit{$p$-biased measure}
$$
\mu_p(\mathcal{A})=\sum_{A\in\mathcal{A}} p^{|A|}(1-p)^{n-|A|},
$$
which assigns to each subset a weight depending only on its size. This homogeneous setting has been extensively studied, for instance by Filmus \cite{F13}, Friedgut \cite{Friedgut} and Tokushige \cite{Tokushige} via spectral methods. The present paper, by contrast, treats the fully inhomogeneous case, in which each coordinate has its own probability. The resulting loss of permutation symmetry is precisely what makes the extremal problem difficult.

For $j\in[n]$, define the \textit{star}
$$
 \mathcal{S}_j=\{A\subseteq[n]:j\in A\},
$$
whose measure is
$
 \mu_{\mathbf p}(\mathcal{S}_j)=p_j.
$
Fishburn, Frankl, Freed, Lagarias and Odlyzko \cite{FishburnEtAl} established the basic measure analogue of the Erd\H{o}s--Ko--Rado theorem: if  $p_1\ge\cdots\ge p_n$ and $p_2\le \frac{1}{2}$, then  every intersecting family $\mathcal{A}\subseteq2^{[n]}$ satisfies
\begin{equation*}
 \mu_{\mathbf p}(\mathcal{A})\leq p_1.
\end{equation*}
Suda, Tanaka and Tokushige \cite{STT} later conjectured that the condition $p_2\le \frac{1}{2}$ can be relaxed to $p_3\le \frac{1}{2}$.  Subsequently, Tokushige \cite{T22} proved this conjecture under the additional assumptions $p_1\le\frac{1}{2}$ or $1-p_2>p_3$, using the high-dimensional Hoffman bound of Filmus, Golubev and Lifshitz \cite{F21}. Recently, Wu and Feng \cite{W26} confirmed the full conjecture using the generating set method.

For cross-intersecting families, one must work with two, possibly unrelated, probability vectors
$
 \mathbf p=(p_1,\ldots,p_n)
$
and 
$
 \mathbf q=(q_1,\ldots,q_n).
$
Related weighted cross-intersection inequalities were obtained by Tokushige \cite{Tokushige} and Borg \cite{Borg}. A systematic treatment of the present two-measure problem was initiated by Suda, Tanaka and Tokushige \cite{STT}. They represented cross-intersecting pairs as cross-independent sets in the bipartite disjointness graph and developed a semidefinite relaxation extending the classical spectral bounds for independent sets.
Their main result states that if, after relabelling, $
p_1=\max_{i\in[n]}p_i
$
and 
$q_1=\max_{i\in[n]}q_i$, with  $p_i,q_i\leq\frac{1}{2}$ for all $i\geq 2$,
then every cross-intersecting pair $\mathcal{A},\mathcal{B}\subseteq2^{[n]}$ satisfies
\begin{equation*}
 \mu_{\mathbf p}(\mathcal{A})
 \mu_{\mathbf q}(\mathcal{B})
 \leq p_1q_1.
\end{equation*}
The common-maximum assumption is strong, since it requires the same coordinate to be individually optimal for both measures. Suda, Tanaka and Tokushige \cite{STT} conjectured that it should be enough for a coordinate to maximize only the product $p_iq_i$.
We adopt the following corrected formulation of their conjecture, as stated in the survey of Frankl and Tokushige \cite[Conjecture 10.1]{FranklTokushigeSurvey}.

\begin{conjecture}[Suda--Tanaka--Tokushige \cite{STT}; corrected form in \cite{FranklTokushigeSurvey}]
\label{conj:product}
Let $\mathbf p=(p_1,\ldots,p_n), \mathbf q=(q_1,\ldots,q_n)\in(0,\frac{1}{2}]^n$ satisfy
$
 p_1q_1=\max_{i\in[n]}p_iq_i.
$
If $\mathcal{A},\mathcal{B}\subseteq2^{[n]}$ are cross-intersecting, then
$$
 \mu_{\mathbf p}(\mathcal{A})\mu_{\mathbf q}(\mathcal{B})
 \leq p_1q_1.
$$
\end{conjecture}

\begin{remark}\label{rem1}
 In the original paper  \cite{STT}, the authors stated a slightly different conjecture, where the restriction $p_i,q_i \le \frac{1}{2}$ was imposed only for coordinates $i \geq 2$, leaving the first coordinate free to be larger than $\frac{1}{2}$. That version is false. For instance, take 
$$n=2,\qquad
 \mathbf p=\left(\frac1{10},\frac15\right),\qquad
 \mathbf q=\left(\frac45,\frac38\right),\qquad \mathcal{A}=2^{[2]}\setminus\{\emptyset\},\qquad\mathcal{B}=\bigl\{\{1,2\}\bigr\}.
$$
The two families are cross-intersecting, the restrictions on the second coordinate are satisfied, and the first coordinate is the unique product-maximizer because
$
 p_1q_1=\frac{2}{25}>\frac{3}{40}=p_2q_2.
$
On the other hand,
\begin{equation*}
 \mu_{\mathbf p}(\mathcal{A})
 =1-\left(1-\frac1{10}\right)\left(1-\frac15\right)
 =\frac7{25},
 \qquad
 \mu_{\mathbf q}(\mathcal{B})
 =\frac45\cdot\frac38
 =\frac3{10}.
\end{equation*}
Hence,
\begin{equation*}
 \mu_{\mathbf p}(\mathcal{A})\mu_{\mathbf q}(\mathcal{B})
 =\frac{21}{250}
 >\frac{20}{250}
 =p_1q_1.
\end{equation*}
The failure is caused precisely by the first coordinate $q_1=\frac{4}{5}>\frac{1}{2}$. In the same paper, the authors proved a weaker version of  Conjecture~\ref{conj:product} under the stronger assumption $p_i, q_i \le \frac{1}{3}$ for all $i\geq 1$  \cite[Theorem 3]{STT}. The original conjecture with the restriction only for $i \geq 2$ is most likely a misprint.
\end{remark}

\medskip
A family $\mathcal{A}\subseteq 2^{[n]}$ is called \textit{increasing} if $A\in\mathcal{A}$ and $A\subseteq B$ imply $B\in\mathcal{A}$.
Our first main result confirms Conjecture~\ref{conj:product}  and gives a complete characterization of the equality cases.

\begin{theorem}\label{thm:main-product}
Let $\mathbf p=(p_1,\ldots,p_n), \mathbf q=(q_1,\ldots,q_n)\in(0,\frac{1}{2}]^n$,  and assume that
$
p_1q_1=\max_{i\in[n]}p_iq_i,
$
$
 I=\{i\in[n]:p_iq_i=p_1q_1\}.
$
If $\mathcal{A},\mathcal{B}\subseteq2^{[n]}$ are cross-intersecting, then
$$
 \mu_{\mathbf p}(\mathcal{A})\mu_{\mathbf q}(\mathcal{B})\leq p_1q_1.
$$
Moreover, equality is characterized as follows.
\begin{enumerate}[label=\textup{(\roman*)},leftmargin=2.2em]
 \item If $p_1q_1<\frac{1}{4}$, then equality holds if and only if
 $
  \mathcal{A}=\mathcal{B}=\mathcal{S}_j
$
 for some $j\in I$.
 \item If $p_1q_1=\frac{1}{4}$, then
$
  I=\{i\in[n]:p_i=q_i=\frac{1}{2}\}\neq \emptyset.
$
Equality holds if and only if there is an increasing family $\mathcal{H}\subseteq2^I$ with $|\mathcal{H}|=2^{|I|-1}$ such that
 $
  \mathcal{A}=\{A\subseteq[n]:A\cap I\in\mathcal{H}\}
$ 
and
$
 \mathcal{B}=\{B\subseteq[n]:B\cap I\in\mathcal{H}^*\},
$
 where
$
  \mathcal{H}^*
  =\{T\subseteq I: T\cap H\neq\emptyset\text{ for every }H\in\mathcal{H}\}.
 $
\end{enumerate}
\end{theorem}

Suda, Tanaka and Tokushige \cite{STT} also conjectured that  the following stability statement holds.

\begin{conjecture}[Suda--Tanaka--Tokushige \cite{STT}]\label{conj:stability}
Let $\mathbf p=(p_1,\ldots,p_n), \mathbf q=(q_1,\ldots,q_n)\in(0,\frac{1}{2})^n$. Assume that $p_1=\max_{j\in[n]}p_j$ and $q_1=\max_{j\in[n]}q_j$.
Then there is a constant $c=c(p_1,q_1)$ such that, whenever $\mathcal{A},\mathcal{B}\subseteq2^{[n]}$ are cross-intersecting and
\begin{equation*}
 \mu_{\mathbf p}(\mathcal{A})\mu_{\mathbf q}(\mathcal{B})>(1-\varepsilon)p_1q_1,
\end{equation*}
there exists $j\in[n]$ such that
\begin{equation*}
 \max\bigl\{
 \mu_{\mathbf p}(\mathcal{A}\symdiff\mathcal{S}_j),
 \mu_{\mathbf q}(\mathcal{B}\symdiff\mathcal{S}_j)
 \bigr\}
 < c\sqrt\varepsilon.
\end{equation*}
\end{conjecture}

Our second main result  confirms Conjecture~\ref{conj:stability} and yields the sharper linear rate in the regime $\varepsilon\leq1$.

\begin{theorem}\label{thm:main-stability}
Let $\mathbf p=(p_1,\ldots,p_n), \mathbf q=(q_1,\ldots,q_n)\in(0,\frac{1}{2})^n$. Assume that $p_1=\max_{j\in[n]}p_j$, $q_1=\max_{j\in[n]}q_j$, and $q_1\leq p_1$. Then there is a constant $c=c(p_1,q_1)$ with the following property. For any $\varepsilon$, if $\mathcal{A},\mathcal{B}\subseteq2^{[n]}$ are cross-intersecting and
\begin{equation*}
 \mu_{\mathbf p}(\mathcal{A})\mu_{\mathbf q}(\mathcal{B})\geq(1-\varepsilon)p_1q_1,
\end{equation*}
then there exists $j\in[n]$ such that
$$
 \max\bigl\{
 \mu_{\mathbf p}(\mathcal{A}\symdiff\mathcal{S}_j),
 \mu_{\mathbf q}(\mathcal{B}\symdiff\mathcal{S}_j)
 \bigr\}
 \leq c\varepsilon.
 $$
\end{theorem}

\medskip
\begin{remark}
To verify that Theorem~\ref{thm:main-stability} implies
Conjecture~\ref{conj:stability}, assume the hypotheses of the
conjecture.
First observe that the strict lower-product hypothesis in
Conjecture~\ref{conj:stability}, together with
Theorem~\ref{thm:main-product}, gives
\begin{equation*}
 (1-\varepsilon)p_1q_1
 <
 \mu_{\mathbf p}(\mathcal{A})\mu_{\mathbf q}(\mathcal{B})
 \leq p_1q_1.
\end{equation*}
Hence, $\varepsilon>0$ automatically.
When $\varepsilon<1$, the linear rate is
sharper because $\varepsilon<\sqrt{\varepsilon}$. When
$\varepsilon\geq1$, neither rate is quantitatively informative, since
every symmetric difference measure is at most one.
Since the statement is symmetric in
$(\mathcal{A},\mathbf p, p_1)$ and $(\mathcal{B},\mathbf q, q_1)$, we may assume  $q_1\leq p_1$.
Let $c_1(p_1,q_1)$ be the constant in
Theorem~\ref{thm:main-stability}.
By Theorem~\ref{thm:main-stability},  there exists $j\in[n]$ such that
$$
 \max\bigl\{
 \mu_{\mathbf p}(\mathcal{A}\symdiff\mathcal{S}_j),
 \mu_{\mathbf q}(\mathcal{B}\symdiff\mathcal{S}_j)
 \bigr\}
 \leq c_1(p_1,q_1)\varepsilon.
 $$
Since $\varepsilon>0$, we have $c_1(p_1,q_1)\geq 0$. 
Set
\begin{equation*}
c(p_1,q_1)=\max\{c_1(p_1,q_1), 1\}+1,\
 D=
 \max\bigl\{
 \mu_{\mathbf p}(\mathcal{A}\symdiff\mathcal{S}_j),
 \mu_{\mathbf q}(\mathcal{B}\symdiff\mathcal{S}_j)
 \bigr\}.
\end{equation*}
If $\varepsilon<1$, then
\begin{equation*}
 D\leq c_1(p_1,q_1)\varepsilon
 \leq c_1(p_1,q_1)\sqrt{\varepsilon}
 <c(p_1,q_1)\sqrt{\varepsilon}.
\end{equation*}
If $\varepsilon\geq1$,  then
\begin{equation*}
 D\leq1
 <c(p_1,q_1)\sqrt{\varepsilon}.
\end{equation*}
Thus, Theorem~\ref{thm:main-stability} implies the strict
$c(p_1,q_1)\sqrt{\varepsilon}$ conclusion of
Conjecture~\ref{conj:stability}.
\end{remark}

\medskip
The proof of Theorem~\ref{thm:main-product} combines a monotone
reduction, interpolation between the two product measures, and induction
on coordinate sections. A sharp two-section inequality closes the
induction, while tracking equality through the same argument gives the
complete classification both below and at the boundary.
 For Theorem~\ref{thm:main-stability}, a semidefinite estimate shows that each indicator
function is concentrated on its constant and one-coordinate components,
up to an error of order $c(p_1,q_1)\varepsilon$. We then establish an
elementary one-coordinate approximation lemma for arbitrary product
measures by a self-contained two-point concentration argument, following
the method of Jendrej, Oleszkiewicz and Wojtaszczyk
\cite{JendrejOleszkiewiczWojtaszczyk}. This reduces the problem to two
one-coordinate families, which monotonicity and cross-intersection force
to be stars with the same centre.

The remainder of this paper is organised as follows. In Section \ref{se2}, we prove Theorem  \ref{thm:main-product}.  In Section  \ref{se3}, we prove Theorem \ref{thm:main-stability}. In Section~\ref{se4}, we
present some concluding remarks and an open problem.

\section{Proof of Theorem  \ref{thm:main-product}}\label{se2}

\subsection{Monotone reduction}

Recall that a family $\mathcal{A}\subseteq2^{[n]}$ is  increasing if $A\in\mathcal{A}$ and $A\subseteq B$ imply $B\in\mathcal{A}$.  The \textit{upset} $\mathcal A^\uparrow$ of $\mathcal{A}$ is  defined by
$$
\mathcal{A}^\uparrow=\{B\subseteq[n]:\text{there exists }A\in\mathcal{A}\text{ such that }A\subseteq B\}.
$$
\begin{lemma}\label{lem:upclosure}
If $\mathcal{A}$ and $\mathcal{B}$ are cross-intersecting, then $\mathcal{A}^{\up}$ and $\mathcal{B}^{\up}$ are cross-intersecting. Moreover,
$$
 \mu_{\mathbf p}(\mathcal{A}^{\up})\geq\mu_{\mathbf p}(\mathcal{A}),
 \qquad
 \mu_{\mathbf q}(\mathcal{B}^{\up})\geq\mu_{\mathbf q}(\mathcal{B}).
$$
with equality in either inequality only when the corresponding family is already increasing.
\end{lemma}
\begin{proof}
If $A'\in\mathcal{A}^{\up}$ and $B'\in\mathcal{B}^{\up}$, then there are $A\in\mathcal{A}$ and $B\in\mathcal{B}$ such that $A\subseteq A'$ and $B\subseteq B'$. Since $A\cap B\neq\emptyset$, we  have $A'\cap B'\neq\emptyset$. Thus, the upsets remain cross-intersecting. The measure inequalities follow from the inclusions $\mathcal{A}\subseteq\mathcal{A}^{\up}$ and $\mathcal{B}\subseteq\mathcal{B}^{\up}$, and equality holds only when the containment is an equality because every singleton has positive measure.
\end{proof}

For an increasing family $\mathcal{A}\subseteq2^{[n]}$,  its \textit{transversal} is defined by
$$
 \mathcal{A}^*=\{B\subseteq[n]:B\cap A\neq\emptyset\text{ for every }A\in\mathcal{A}\}.
$$

\begin{lemma}
\label{lem:transversal}
If $\mathcal{A}\subseteq2^{[n]}$ is increasing, then
$
 \mathcal{A}^*=\{B\subseteq[n]:[n]\setminus B\notin\mathcal{A}\}.
$
Consequently, for every probability vector $\mathbf q=(q_1,\ldots,q_n)$,
$$
 \mu_{\mathbf q}(\mathcal{A}^*)
 =1-\mu_{\mathbf 1-\mathbf q}(\mathcal{A}),
$$
where $\mathbf 1-\mathbf q=(1-q_1,\ldots,1-q_n)$.
\end{lemma}

\begin{proof}
If $B\notin\mathcal{A}^*$, then $B\cap A=\emptyset$ for some $A\in\mathcal{A}$, Hence, $A\subseteq[n]\setminus B$. Since $\mathcal{A}$ is increasing, $[n]\setminus B\in\mathcal{A}$. Conversely, if $[n]\setminus B\in\mathcal{A}$, then this member of $\mathcal{A}$ is disjoint from $B$. Thus, $B\notin\mathcal{A}^*$. This proves $
 \mathcal{A}^*=\{B\subseteq[n]:[n]\setminus B\notin\mathcal{A}\}.
$
If $X$ has distribution $\mu_{\mathbf q}$, then $[n]\setminus X$ has distribution $\mu_{\mathbf 1-\mathbf q}$. Therefore, we have
$
 \mu_{\mathbf q}(\mathcal{A}^*)=\mathbb P\bigl([n]\setminus X\notin\mathcal{A}\bigr)=1-\mathbb P\bigl([n]\setminus X\in\mathcal{A}\bigr)=1-\mu_{\mathbf 1-\mathbf q}(\mathcal{A}).
$
\end{proof}

If $\mathcal{A}$ and $\mathcal{B}$ are increasing and cross-intersecting, then
$
 \mathcal{B}\subseteq\mathcal{A}^*.
$
By Lemma~\ref{lem:transversal}, we have
$$
 \mu_{\mathbf p}(\mathcal{A})\mu_{\mathbf q}(\mathcal{B})
 \leq
 \mu_{\mathbf p}(\mathcal{A})
 \bigl(1-\mu_{\mathbf 1-\mathbf q}(\mathcal{A})\bigr).
 $$
Therefore, the cross-intersecting pair problem in Theorem \ref{thm:main-product}  reduces to the separation inequality 
\begin{equation}
 \mu_{\mathbf p}(\mathcal{A})
 \bigl(1-\mu_{\mathbf 1-\mathbf q}(\mathcal{A})\bigr)
 \leq \max_{i\in[n]}p_iq_i,
 \label{eq:separation}
\end{equation}
which will be
proved
in Subsection~\ref{sec:product-proof}.
\subsection{Comparison of the  measures}

For $\mathcal{A}\subseteq2^{[n]}$ and $i\in[n]$, define the \textit{$i$-influential set} of $\mathcal{A}$ by
$$
        \mathcal{I}_i(\mathcal{A})
        =
        \{S\subseteq[n]:|\{S,S\triangle\{i\}\}\cap\mathcal{A}|=1\}.
$$
For $\mathbf p=(p_1,\ldots,p_n)\in(0,1)^n$ and $i\in[n]$, write
$\mathbf p_{-i}$ for the vector obtained by deleting $p_i$,  and define the \textit{$i$-influence} under $\mu_{\mathbf p}$ by
\begin{equation*}
 I_i^{\mathbf p}(\mathcal{A})
 =
 \mu_{\mathbf p}\bigl(\mathcal{I}_i(\mathcal{A})\bigr).
\end{equation*}
When $\mathbf p=(p,\ldots,p)$, this is the usual $p$-biased influence
$I_i^p(\mathcal{A})$.
For later use, let
\begin{equation*}
 \mathcal{A}(\bar{i})
 =
 \{S\subseteq[n]\setminus\{i\}:S\in\mathcal{A}\},
 \qquad
 \mathcal{A}(i)
 =
 \{S\subseteq[n]\setminus\{i\}:S\cup\{i\}\in\mathcal{A}\}.
\end{equation*}
If $\mathcal{A}$ is increasing, then
$ \mathcal{A}(\bar{i})\subseteq \mathcal{A}(i)$, and
\begin{equation*}
 I_i^{\mathbf p}(\mathcal{A})
 =
\mu_{\mathbf p_{-i}}
 \bigl( \mathcal{A}(i)\setminus \mathcal{A}(\bar{i})\bigr).
\end{equation*}
Indeed, for every
$S\in \mathcal{A}(i)\setminus \mathcal{A}(\bar{i})$, both $S$ and
$S\cup\{i\}$ belong to $\mathcal{I}_i(\mathcal{A})$. Conversely, let $S\in\mathcal{I}_i(\mathcal{A})$ and put
$T=S\setminus\{i\}$. Exactly one of $T$ and $T\cup\{i\}$ belongs to
$\mathcal{A}$. Since $\mathcal{A}$ is increasing, it is impossible
that $T\in\mathcal{A}$ but $T\cup\{i\}\notin\mathcal{A}$. Thus,
$T\notin\mathcal{A}$ and $T\cup\{i\}\in\mathcal{A}$, which means that
$
 T\in\mathcal{A}(i)\setminus\mathcal{A}(\bar{i}).
$
Consequently,
\begin{align*}
 I_i^{\mathbf p}(\mathcal{A})
 &=
(1- p_i)
 \mu_{\mathbf p_{-i}}
 \bigl(\mathcal{A}(i)\setminus\mathcal{A}(\bar{i})\bigr)
 +
p_i
 \mu_{\mathbf p_{-i}}
 \bigl(\mathcal{A}(i)\setminus\mathcal{A}(\bar{i})\bigr)\\
 &=
 \mu_{\mathbf p_{-i}}
 \bigl(\mathcal{A}(i)\setminus\mathcal{A}(\bar{i})\bigr).
\end{align*}

The following lemma extends the Margulis--Russo formula (see Theorem 2.2 in \cite{E24}) to product
measures with non-identical coordinate probabilities.
When $r_i(t)=t$ for every $i$, it reduces to the classical
Margulis--Russo formula.

\begin{lemma}\label{lem:margulis-russo}
Let $\mathcal{A}\subseteq2^{[n]}$ be increasing, and let
$\mathbf r(t)=(r_1(t),\ldots,r_n(t))\in(0,1)^n$. Then
\begin{equation*}
 \frac{\mathrm d}{\mathrm dt}
 \mu_{\mathbf r(t)}(\mathcal{A})
 =
 \sum_{i=1}^n
 r_i'(t)I_i^{\mathbf r(t)}(\mathcal{A}).
\end{equation*}
\end{lemma}

\begin{proof}
For $\mathbf x=(x_1,\ldots,x_n)\in(0,1)^n$, define
$
 H(\mathbf x)=\mu_{\mathbf x}(\mathcal{A}).
$
By the definition of the measure, $H$ is 
differentiable.
 Conditioning on whether coordinate $i$
is present gives
\begin{equation*}
 H(\mathbf x)
 =
 (1-x_i)
 \mu_{\mathbf x_{-i}}\bigl(\mathcal{A}(\bar{i})\bigr)
 +
 x_i
 \mu_{\mathbf x_{-i}}\bigl(\mathcal{A}(i)\bigr).
\end{equation*}
 Hence, we have
\begin{equation*}
 \frac{\partial H(\mathbf x)}{\partial x_i}
 =
 \mu_{\mathbf x_{-i}}\bigl(\mathcal{A}(i)\bigr)
 -
 \mu_{\mathbf x_{-i}}\bigl(\mathcal{A}(\bar{i})\bigr).
\end{equation*}
Since $\mathcal{A}$ is increasing,
$
 \mathcal{A}(\bar{i})\subseteq\mathcal{A}(i).
$
Therefore,
\begin{align*}
 \frac{\partial H(\mathbf x)}{\partial x_i}
=
 \mu_{\mathbf x_{-i}}
 \bigl(\mathcal{A}(i)\setminus\mathcal{A}(\bar{i})\bigr)=
 I_i^{\mathbf x}(\mathcal{A}).
\end{align*}
Finally, applying the multivariable chain rule to
$H(\mathbf r(t))$ gives
\begin{align*}
 \frac{\mathrm d}{\mathrm dt}
 \mu_{\mathbf r(t)}(\mathcal{A})
 =
 \sum_{i=1}^n
 \frac{\partial H(\mathbf r(t))}{\partial r_i(t)}r_i'(t)=
 \sum_{i=1}^n
 r_i'(t)I_i^{\mathbf r(t)}(\mathcal{A}),
\end{align*}
as required.
\end{proof}

The following conditional-variance form of the Efron--Stein inequality
will be used;  see 
\cite[Theorem~3.1]{BLM}.

\begin{lemma}\label{lem:efron-stein}
Let $\mathbf X=(X_1,\ldots,X_n)$ be a vector of independent random
variables, and let $Z=f(\mathbf X)$ be a real-valued random variable
satisfying
$
 \mathbb E[Z^2]<\infty.
$
For each $i\in[n]$, write
$
 \mathbf X_{-i}
 =
 (X_1,\ldots,X_{i-1},X_{i+1},\ldots,X_n).
$
Then
$$
 \operatorname{Var}(Z)
 \leq
 \sum_{i=1}^n
 \mathbb E\left[
  \operatorname{Var}(Z\mid\mathbf X_{-i})
 \right].
$$
\end{lemma}

\begin{lemma}\label{lem:efron-stein-influence}
Let $\mathbf p=(p_1,\ldots,p_n)\in(0,1)^n$, and let
$\mathcal{A}\subseteq2^{[n]}$ be increasing. Then
$$
 \mu_{\mathbf p}(\mathcal{A})
 \bigl(1-\mu_{\mathbf p}(\mathcal{A})\bigr)
 \leq
 \sum_{i=1}^n
 p_i(1-p_i)I_i^{\mathbf p}(\mathcal{A}).
 $$
\end{lemma}

\begin{proof}
Let $X_1,\ldots,X_n$ be independent Bernoulli random variables with
\begin{equation*}
 \mathbb P(X_i=1)=p_i,
 \qquad
 \mathbb P(X_i=0)=1-p_i,
\end{equation*}
and write
$
 \mathbf X=(X_1,\ldots,X_n).
$
We identify $\mathbf X$ with the random subset
$
 \{j\in[n]:X_j=1\},
$
which has distribution $\mu_{\mathbf p}$. Let
$
 f=\mathbf 1_{\mathcal{A}}
$
denote the indicator function. Then
$
 f(\mathbf X)^2=f(\mathbf X)\leq 1. 
$
Consequently,
\begin{align*}
 \operatorname{Var}(f(\mathbf X))
 &=
 \mathbb E[f(\mathbf X)^2]
 -
 \mathbb E[f(\mathbf X)]^2=
 \mu_{\mathbf p}(\mathcal{A})
 -
 \mu_{\mathbf p}(\mathcal{A})^2\\
 &=
 \mu_{\mathbf p}(\mathcal{A})
 \bigl(1-\mu_{\mathbf p}(\mathcal{A})\bigr).
\end{align*}
By Lemma \ref{lem:efron-stein}, we have
$$
 \mu_{\mathbf p}(\mathcal{A})
 \bigl(1-\mu_{\mathbf p}(\mathcal{A})\bigr)=\operatorname{Var}(f(\mathbf X))
 \leq
 \sum_{i=1}^n
 \mathbb E\left[
 \operatorname{Var}
 \bigl(f(\mathbf X)\mid\mathbf X_{-i}\bigr)
 \right],
$$
where $\mathbf X_{-i}$ denotes the vector obtained from $\mathbf X$ by
deleting its $i$th coordinate.

Fix
$i\in[n]$ and condition on
$
 \mathbf X_{-i}=S
$ 
and
$
 S\subseteq[n]\setminus\{i\}.
$
Since the coordinates are independent, the conditional distribution of
$X_i$ is still Bernoulli with parameter $p_i$. Hence,
\begin{equation*}
 f(\mathbf X)
 =
 \begin{cases}
  f(S),&\text{with probability }1-p_i,\\
  f(S\cup\{i\}),&\text{with probability }p_i.
 \end{cases}
\end{equation*}
It follows that
\begin{equation*}
 \mathbb E\bigl[
 f(\mathbf X)\mid\mathbf X_{-i}=S
 \bigr]
 =
 (1-p_i)f(S)+p_i f(S\cup\{i\})
\end{equation*}
and
\begin{equation*}
 \mathbb E\bigl[
 f(\mathbf X)^2\mid\mathbf X_{-i}=S
 \bigr]
 =
 (1-p_i)f(S)^2+p_i f(S\cup\{i\})^2.
\end{equation*}
Therefore,
\begin{align*}
 \operatorname{Var}
 \bigl(f(\mathbf X)\mid\mathbf X_{-i}=S\bigr)&=
 (1-p_i)f(S)^2+p_i f(S\cup\{i\})^2-
 \bigl((1-p_i)f(S)+p_i f(S\cup\{i\})\bigr)^2\\
 &=
 p_i(1-p_i)
 \bigl(
 f(S)^2+f(S\cup\{i\})^2
 -2f(S)f(S\cup\{i\})
 \bigr)\\
 &=
 p_i(1-p_i)
 \bigl(f(S\cup\{i\})-f(S)\bigr)^2.
 \label{eq:conditional-variance-calculation}
\end{align*}
Because $\mathcal{A}$ is increasing, 
$
 f(S)\leq f(S\cup\{i\}).
$
Since both values belong to $\{0,1\}$, we have
\begin{equation*}
 \bigl(f(S\cup\{i\})-f(S)\bigr)^2
 =
 \begin{cases}
  1,
  &S\in\mathcal{A}(i)\setminus\mathcal{A}(\bar{i}),\\
  0,
  &S\notin\mathcal{A}(i)\setminus\mathcal{A}(\bar{i}).
 \end{cases}
\end{equation*}
Equivalently,
\begin{equation*}
 \bigl(f(S\cup\{i\})-f(S)\bigr)^2
 =
 \mathbf 1_{\mathcal{A}(i)\setminus
 \mathcal{A}(\bar{i})}(S).
\end{equation*}
Thus,
\begin{equation*}
 \operatorname{Var}
 \bigl(f(\mathbf X)\mid\mathbf X_{-i}=S\bigr)
 =
 p_i(1-p_i)
 \mathbf 1_{\mathcal{A}(i)\setminus
 \mathcal{A}(\bar{i})}(S).
\end{equation*}
The random set $\mathbf X_{-i}$ has distribution
$\mu_{\mathbf p_{-i}}$. Taking expectation over $\mathbf X_{-i}$,
we obtain
\begin{align*}
 \mathbb E\left[
 \operatorname{Var}
 \bigl(f(\mathbf X)\mid\mathbf X_{-i}\bigr)
 \right]&=
 p_i(1-p_i)
 \sum_{S\subseteq[n]\setminus\{i\}}
 \mu_{\mathbf p_{-i}}(S)
 \mathbf 1_{\mathcal{A}(i)\setminus
 \mathcal{A}(\bar{i})}(S)\\
 &\quad=
 p_i(1-p_i)
 \mu_{\mathbf p_{-i}}
 \bigl(\mathcal{A}(i)\setminus\mathcal{A}(\bar{i})\bigr)\\
 &\quad=
 p_i(1-p_i)I_i^{\mathbf p}(\mathcal{A}).
 \label{eq:conditional-variance-influence}
\end{align*}
It follows that
$$
 \mu_{\mathbf p}(\mathcal{A})
 \bigl(1-\mu_{\mathbf p}(\mathcal{A})\bigr)
 \leq
 \sum_{i=1}^n
 p_i(1-p_i)I_i^{\mathbf p}(\mathcal{A}),
 $$
as required.
\end{proof}

Let $\mathbf p=(p_1,\ldots,p_n), \mathbf q=(q_1,\ldots,q_n)\in(0,\frac{1}{2})^n$.
For $t\in(0,1)$, define its \textit{odds} and \textit{log-odds} respectively by
\begin{equation*}
 \frac{t}{1-t}
 \qquad\text{and}\qquad
 \log\frac{t}{1-t},
\end{equation*}
where  $\log$ denotes the natural logarithm.
The advantage of the log-odds coordinate is that the
passage from $p_i$ to $1-q_i$ becomes additive.
For each $i\in[n]$, set
\begin{equation*}
 d_i=p_iq_i,
 \qquad
 K_i=\frac{(1-p_i)(1-q_i)}{p_iq_i},
 \qquad
 K=\min_{i\in[n]}K_i.
\end{equation*}
Since $p_i+q_i<1$, we have $K_i>1$ for every $i\in[n]$.
Combining the generalized Margulis--Russo formula with the preceding
variance estimate yields the following comparison between the measures
of an increasing family under $\mu_{\mathbf p}$ and
$\mu_{\mathbf{1-\mathbf q}}$.

\begin{lemma}\label{lem:logodds}
Let $\mathbf p=(p_1,\ldots,p_n), \mathbf q=(q_1,\ldots,q_n)\in(0,\frac{1}{2})^n$.
Let $\mathcal{A}\subseteq2^{[n]}$ be increasing, and put
\begin{equation*}
 \alpha=\mu_{\mathbf p}(\mathcal{A}),
 \qquad
 \beta=\mu_{\mathbf 1-\mathbf q}(\mathcal{A}).
\end{equation*}
If $0<\alpha<1$, then
\begin{equation*}
 \frac{\beta}{1-\beta}
 \geq
 K\frac{\alpha}{1-\alpha}.
\end{equation*}
Equivalently,
\begin{equation*}
 1-\beta\leq h_K(\alpha),
 \qquad
 h_K(t)=\frac{1-t}{1+(K-1)t}.
\end{equation*}
The latter inequality also holds when $\alpha\in\{0,1\}$.
\end{lemma}

\begin{proof}
If $\alpha=0$, then 
$\mathcal{A}=\emptyset$, and hence $\beta=0$. If $\alpha=1$, then
$\mathcal{A}=2^{[n]}$, and hence $\beta=1$. Thus,
$
 1-\beta\leq h_K(\alpha)
$ holds in both boundary cases. We may therefore
assume that $0<\alpha<1$. Then
$\mathcal{A}\neq\emptyset$ and 
$\mathcal{A}\neq 2^{[n]}$.
Put
$
 \lambda_i=\log K_i
$
and define $\mathbf s(t)=(s_1(t),\ldots,s_n(t))$ by
\begin{equation*}
 \frac{s_i(t)}{1-s_i(t)}
 =
 \frac{p_i}{1-p_i}K_i^t,
 \qquad 0\leq t\leq1.
\end{equation*}
Set
\begin{equation*}
 R_i(t)=\frac{p_i}{1-p_i}K_i^t>0.
\end{equation*}
Solving $s_i(t)/(1-s_i(t))=R_i(t)$ gives
\begin{equation*}
 s_i(t)=\frac{R_i(t)}{1+R_i(t)}.
\end{equation*}
In particular, we have $s_i(0)=p_i$ and $s_i(1)=1-q_i$. 
Since $R_i(t)>0$, it follows  that
$0<s_i(t)<1$. Hence, $\mathbf s(t)\in(0,1)^n$ and
$\mu_{\mathbf s(t)}$ is well defined.
Let
$
 g(t)=\mu_{\mathbf s(t)}(\mathcal{A}).
$
Since $\mathcal{A}\neq\emptyset$ and 
$\mathcal{A}\neq 2^{[n]}$, we have $0<g(t)<1$.

Taking logarithms in the defining equation for $s_i(t)$ gives
\begin{equation*}
 \log\frac{s_i(t)}{1-s_i(t)}
 =
 \log\frac{p_i}{1-p_i}+t\log K_i.
\end{equation*}
Differentiating both sides, we obtain
\begin{align*}
 \frac{s_i'(t)}{s_i(t)}
 +
 \frac{s_i'(t)}{1-s_i(t)}
 &=
 \log K_i.
\end{align*}
Hence,
\begin{equation*}
 \frac{s_i'(t)}
 {s_i(t)(1-s_i(t))}
 =
 \log K_i.
\end{equation*}
Writing $\lambda_i=\log K_i$, we have
$
 s_i'(t)
 =
 \lambda_i s_i(t)(1-s_i(t)).
$
By Lemma \ref{lem:margulis-russo},  we obtain
\begin{align*}
 g'(t)
 =
 \sum_{i=1}^n
 s_i'(t)I_i^{\mathbf s(t)}(\mathcal{A})=
 \sum_{i=1}^n
 \lambda_i s_i(t)(1-s_i(t))
 I_i^{\mathbf s(t)}(\mathcal{A}).
\end{align*}
On the other hand, applying Lemma \ref{lem:efron-stein-influence}  yields
\begin{align*}
 g(t)(1-g(t))\leq
 \sum_{i=1}^n
 s_i(t)(1-s_i(t))
 I_i^{\mathbf s(t)}(\mathcal{A}).
\end{align*}
Since $K=\min_{i\in[n]}K_i$ and $K_i>1$ for every $i$, we have
$
 \lambda_i=\log K_i\geq\log K>0.
$
All the influences are nonnegative, so
\begin{align*}
 g'(t)
 &\geq
 (\log K)
 \sum_{i=1}^n
 s_i(t)(1-s_i(t))
 I_i^{\mathbf s(t)}(\mathcal{A})\\
 &\geq
 (\log K)g(t)(1-g(t)).
\end{align*}
Because $0<g(t)<1$, division by $g(t)(1-g(t))$ is valid. Thus,
\begin{equation*}
 \frac{g'(t)}{g(t)(1-g(t))}
 \geq
 \log K.
\end{equation*}
Note that
\begin{align*}
 \frac{\mathrm d}{\mathrm dt}
 \log\frac{g(t)}{1-g(t)}
 &=
 \frac{g'(t)}{g(t)}
 +
 \frac{g'(t)}{1-g(t)}=
 \frac{g'(t)}{g(t)(1-g(t))}.
\end{align*}
Integrating over $t\in[0,1]$  gives
\begin{equation*}
 \log\frac{g(1)}{1-g(1)}
 -
 \log\frac{g(0)}{1-g(0)}
 \geq
 \log K.
\end{equation*}
Since $\mathbf s(0)=\mathbf p$ and
$\mathbf s(1)=\mathbf 1-\mathbf q$, we have
\begin{equation*}
 g(0)=\alpha,
 \qquad
 g(1)=\beta.
\end{equation*}
Consequently,
\begin{equation*}
 \log\frac{\beta}{1-\beta}
 -
 \log\frac{\alpha}{1-\alpha}
 \geq
 \log K.
\end{equation*}
Exponentiating both sides yields
\begin{equation*}
 \frac{\beta}{1-\beta}
 \geq
 K\frac{\alpha}{1-\alpha}.
\end{equation*}

Finally, rearranging this inequality gives
$
 \beta
 \geq
 \frac{K\alpha}{1+(K-1)\alpha}.
$
Hence,
\begin{equation*}
 1-\beta
 \leq
 \frac{1-\alpha}{1+(K-1)\alpha}
 =
 h_K(\alpha).
\end{equation*}
Together with the two boundary cases considered at the beginning, this
completes the proof.
\end{proof}

Before proceeding, we record the following relation between $d_1$ and $K$.

\begin{lemma}\label{lem:range}
Let $\mathbf p=(p_1,\ldots,p_n), \mathbf q=(q_1,\ldots,q_n)\in(0,\frac{1}{2})^n$ satisfy
$
 d_1=\max_{i\in[n]}d_i,
$
where $d_i=p_iq_i$.
Then
$$
 \frac{1}{2(K+1)}
 <
d_1
 \leq
 \frac{1}{(1+\sqrt K)^2}.
$$
\end{lemma}

\begin{proof}
For every $i\in[n]$, the definition of $K_i$ gives
\begin{align*}
 K_id_i=
 (1-p_i)(1-q_i)=
 1-p_i-q_i+d_i.
\end{align*}
Consequently,
$$
 p_i+q_i=1-(K_i-1)d_i.
 $$

We first prove the lower bound. Choose $j\in[n]$ such that
$K_j=K$. Since $p_j,q_j<1/2$, we have
$
 (1-2p_j)(1-2q_j)>0.
$
Expanding the left-hand side, we obtain
\begin{align*}
 0
 &<
 (1-2p_j)(1-2q_j)=
 1-2p_j-2q_j+4p_jq_j\\
 &=
 1-2(p_j+q_j)+4d_j=
 1-2\bigl(1-(K_j-1)d_j\bigr)+4d_j\\
 &=
 -1+2(K_j-1)d_j+4d_j=
 -1+2(K_j+1)d_j.
\end{align*}
It follows that
\begin{equation*}
 d_j>\frac{1}{2(K_j+1)}
     =\frac{1}{2(K+1)}.
\end{equation*}
Since $d_1=\max_{i\in[n]}d_i$, we have $$d_1\geq d_j>\frac{1}{2(K+1)}.$$

We next prove the upper bound. By the arithmetic--geometric mean
inequality,
\begin{equation*}
 p_1+q_1\geq2\sqrt{p_1q_1}=2\sqrt{d_1}.
\end{equation*}
It follows that
\begin{equation*}
 2\sqrt{d_1}
 \leq
 1-(K_1-1)d_1.
\end{equation*}
Set $x=\sqrt{d_1}$. Then
$
 1-2x-(K_1-1)x^2\geq0.
$
The left-hand side factors as
\begin{equation*}
 1-2x-(K_1-1)x^2
 =
 \bigl(1-(1+\sqrt{K_1})x\bigr)
 \bigl(1+(\sqrt{K_1}-1)x\bigr).
\end{equation*}
Since $p_1,q_1<1/2$, we have $K_1>1$. Hence,
$
 1+(\sqrt{K_1}-1)x>0.
$
It follows that
\begin{equation*}
 1-(1+\sqrt{K_1})x\geq0.
\end{equation*}
Thus, we have
\begin{equation*}
 \sqrt{d_1}=x\leq\frac{1}{1+\sqrt{K_1}}.
\end{equation*}
Then
\begin{equation*}
 d_1\leq\frac{1}{(1+\sqrt{K_1})^2} \leq
 \frac{1}{(1+\sqrt K)^2}.
\end{equation*}
This proves both bounds.
\end{proof}

\subsection{A sharp two-section inequality}\label{sec:bellman}

Let $\mathbf p=(p_1,\ldots,p_n), \mathbf q=(q_1,\ldots,q_n)\in(0,\frac{1}{2})^n$ satisfy
$
 d_1=\max_{i\in[n]}d_i,
$
where $d_i=p_iq_i$. Recall that
\begin{equation*}
 K_i=\frac{(1-p_i)(1-q_i)}{p_iq_i},
 \qquad
 K=\min_{i\in[n]}K_i>1,\qquad  h_K(t)=\frac{1-t}{1+(K-1)t}.
\end{equation*}
We shall use the following numerical inequality.

\begin{lemma}\label{lem:bellman}
Let $K>1$ and suppose that
\begin{equation*}
 \frac{1}{2(K+1)}
 \leq d_1
 \leq
 \frac{1}{(1+\sqrt K)^2}.
\end{equation*}
Let $0\leq x\leq y\leq1$ and $0\leq v\leq u\leq1$ satisfy
\begin{equation*}
 xu\leq d_1,
 \qquad
 yv\leq d_1,
 \qquad
 u\leq h_K(x),
 \qquad
 v\leq h_K(y).
\end{equation*}
Then, for $0<p,q\leq\frac{1}{2}$,
\begin{equation}
 \bigl((1-p)x+py\bigr)
 \bigl(qu+(1-q)v\bigr)
 \leq
 \max\{d_1,pq\}.
 \label{eq:bellman}
\end{equation}
\end{lemma}

\begin{proof}
Define
\begin{equation*}
 F(t)=t h_K(t)
     =\frac{t(1-t)}{1+(K-1)t},
 \qquad 0\leq t\leq1.
\end{equation*}
Differentiating gives
\begin{equation*}
 F'(t)
 =
 \frac{1-2t-(K-1)t^2}
      {(1+(K-1)t)^2}.
\end{equation*}
The numerator has a unique zero in $(0,1)$, namely
\begin{equation*}
 t_0=\frac{1}{1+\sqrt K}.
\end{equation*}
Consequently, $F$ is increasing on $[0,t_0]$ and decreasing on
$[t_0,1]$. Moreover,
\begin{equation*}
 F(t_0)=\frac{1}{(1+\sqrt K)^2},
 \qquad
 F\left(\frac12\right)=\frac{1}{2(K+1)}.
\end{equation*}
The assumed range of $d_1$ therefore implies that the equation
$
 t h_K(t)=d_1
$
has two roots, counted with multiplicity,
\begin{equation*}
 0<a\leq t_0\leq b\leq\frac12.
\end{equation*}

The equation $t h_K(t)=d_1$ can be rewritten as
$
 t^2-\bigl(1-(K-1)d_1\bigr)t+d_1=0.
$
Hence,
\begin{equation*}
 ab=d_1,
 \qquad
 a+b=1-(K-1)d_1.
\end{equation*}
Set
$
 c:=1-a-b=(K-1)d_1>0.
$
Then
\begin{equation*}
 K=\frac{(1-a)(1-b)}{ab},\qquad
 h_K(t)
 =
 \frac{d_1(1-t)}{d_1+ct}.
\end{equation*}

For $t>0$, define
\begin{equation*}
 \phi(t)=\min\left\{h_K(t),\frac{d_1}{t}\right\},
\end{equation*}
and set $\phi(0)=1$. Since
\begin{equation*}
 h_K(t)-\frac{d_1}{t}
 =
 -\frac{d_1(t-a)(t-b)}
        {t(d_1+ct)},
\end{equation*}
we have
\begin{equation}
 \phi(t)=
 \begin{cases}
  h_K(t),&0\leq t\leq a,\\[2pt]
  d_1/t,&a\leq t\leq b,\\[2pt]
  h_K(t),&b\leq t\leq1.
 \end{cases}
 \label{eq:phi-pieces}
\end{equation}
A direct calculation gives
\begin{equation*}
 h_K(h_K(t))=t.
\end{equation*}
Furthermore,
\begin{equation*}
 h_K(a)=b,
 \qquad
 h_K(b)=a.
\end{equation*}
Since $h_K$ is decreasing and
\begin{equation*}
 h_K(0)=1,
 \qquad
 h_K(a)=b,
 \qquad
 h_K(b)=a,
 \qquad
 h_K(1)=0,
\end{equation*}
we have
\begin{equation*}
 h_K([0,a])=[b,1],
 \qquad
 h_K([b,1])=[0,a].
\end{equation*}
Moreover, the map $t\mapsto d_1/t$ maps $[a,b]$ onto itself. Both maps are decreasing
and involutive on the corresponding intervals. It follows that
$\phi$ is decreasing and satisfies
\begin{equation*}
 \phi(\phi(t))=t
 \qquad\text{for every }t\in[0,1].
\end{equation*}

The inequalities $xu\leq d_1$ and $u\leq h_K(x)$ imply
$
 u\leq\phi(x).
$
Indeed, this is immediate when $x=0$, while for $x>0$ we have both
$u\leq h_K(x)$ and $u\leq d_1/x$. Similarly,
$
 v\leq\phi(y).
$
Since $\phi$ is decreasing and involutive, the latter inequality gives
\begin{equation*}
 \phi(v)\geq\phi(\phi(y))=y.
\end{equation*}
Moreover, $x\leq y$ implies
\begin{equation*}
 v\leq\phi(y)\leq\phi(x).
\end{equation*}
The left-hand side of \eqref{eq:bellman} is nondecreasing in $u$ and
$y$. We may therefore replace $u$ by $\phi(x)$ and $y$ by $\phi(v)$.
It remains to prove
\begin{equation*}
 G_{p,q}(x,v)\leq\max\{d_1,pq\},
\end{equation*}
where
\begin{equation*}
 G_{p,q}(x,v)
 =
 \bigl((1-p)x+p\phi(v)\bigr)
 \bigl((1-q)v+q\phi(x)\bigr),\qquad v\leq\phi(x).
\end{equation*}
Since $\phi$ is a decreasing involution, the last inequality is
equivalent to $x\leq\phi(v)$.

We first reduce the proof to the case $pq=d_1$. Observe that $d_1<1/4$. The inequalities
$\phi(v)\geq x$ and $\phi(x)\geq v$ give
\begin{align*}
 \frac{\partial G_{p,q}}{\partial p}
 &=
 \bigl(\phi(v)-x\bigr)
 \bigl((1-q)v+q\phi(x)\bigr)
 \geq0,\\
 \frac{\partial G_{p,q}}{\partial q}
 &=
 \bigl(\phi(x)-v\bigr)
 \bigl((1-p)x+p\phi(v)\bigr)
 \geq0.
\end{align*}
Thus, $G_{p,q}$ is nondecreasing in each of $p$ and $q$. On the other
hand,
\begin{equation}
 \frac{G_{p,q}(x,v)}{pq}
 =
 \bigl(\phi(v)+(p^{-1}-1)x\bigr)
 \bigl(\phi(x)+(q^{-1}-1)v\bigr).\label{eq:normalized-G}
\end{equation}
For fixed $x$ and $v$, the right-hand side of
\eqref{eq:normalized-G} is nonincreasing separately in $p$ and $q$.

Suppose first that $pq\leq d_1$. For $0\leq t\leq1$, define
\begin{equation*}
 p(t)=p+t\left(\frac12-p\right),
 \qquad
 q(t)=q+t\left(\frac12-q\right).
\end{equation*}
Then $p(t)$ and $q(t)$ are nondecreasing functions of $t$, and
\begin{equation*}
 p(0)q(0)=pq\leq d_1,
 \qquad
 p(1)q(1)=\frac14>d_1.
\end{equation*}
Since $p(t)q(t)$ is continuous, there exists $t_0\in[0,1)$ such that
$
 p(t_0)q(t_0)=d_1.
$
Set
$
 p'=p(t_0)
$
and 
$
 q'=q(t_0).
$
Then
\begin{equation*}
 p'\geq p,
 \qquad
 q'\geq q,
 \qquad
 p'q'=d_1.
\end{equation*}
Since $G_{p,q}$ is nondecreasing separately in $p$ and $q$, we
have
$
 G_{p,q}(x,v)
 \leq
 G_{p',q'}(x,v).
$
Thus, it suffices to prove
\begin{equation*}
 G_{p',q'}(x,v)\leq d_1
\end{equation*}
Therefore, the case $pq\leq d_1$ is reduced to the case $pq=d_1$.

Suppose next that $pq\geq d_1$. Keep $p$ fixed and define
$
 q'=\frac{d_1}{p}.
$
Then $q'>0$ and
$
 q'\leq q\leq\frac12.
$
By construction,
$
 pq'=d_1.
$
For fixed $p,x$ and $v$, equation \eqref{eq:normalized-G} shows that
\begin{equation*}
 \frac{G_{p,s}(x,v)}{ps}
\end{equation*}
is nonincreasing as a function of $s$. Since $q'\leq q$, it follows
that
\begin{equation*}
 \frac{G_{p,q}(x,v)}{pq}
 \leq
 \frac{G_{p,q'}(x,v)}{pq'}.
\end{equation*}
Once the estimate has been proved for pairs whose product is $d_1$, we
may apply it to $(p,q')$ to obtain
$
 G_{p,q'}(x,v)\leq d_1=pq'.
$
Thus,
\begin{equation*}
 \frac{G_{p,q}(x,v)}{pq}
 \leq
 \frac{G_{p,q'}(x,v)}{pq'}
 \leq1,
\end{equation*}
and hence
\begin{equation*}
 G_{p,q}(x,v)
 \leq pq
 =
 \max\{d_1,pq\}.
\end{equation*}
Therefore, the case $pq\geq d_1$ is also reduced to the case $pq=d_1$.

\textbf{It remains to prove
$
 G_{p,q}(x,v)\leq d_1
$
under the assumption
$
 pq=d_1.
$}
In this case, we have
$
 \frac{d_1}{p}=q\leq\frac12,
$
which is equivalent to
$
 p\geq2d_1.
$
Together with $p\leq1/2$, this gives
\begin{equation*}
 2d_1\leq p\leq\frac12.
\end{equation*}
By symmetry, $p\leq1/2$ and $pq=d_1$ also give
\begin{equation*}
 2d_1\leq q\leq\frac12.
\end{equation*}
Consequently, it suffices to prove $
 G_{p,q}(x,v)\leq d_1
$ under
\begin{equation*}
 pq=d_1,
 \qquad
 2d_1\leq p,q\leq\frac12.
\end{equation*}

Recall that $0<a\leq b\leq\frac12$.
Let
$
 I_0=[0,a],
 I_1=[a,b]
$
and
$
 I_2=[b,1].
$
By \eqref{eq:phi-pieces} and the involution property of $\phi$, we have
\begin{equation*}
 \phi(I_0)=I_2,
 \qquad
 \phi(I_1)=I_1,
 \qquad
 \phi(I_2)=I_0.
\end{equation*}
We now determine which pairs of intervals can contain $(x,v)$ under
the constraint $v\leq\phi(x)$.

If $x\in I_0$, then $\phi(x)\in I_2$, so $v$ may lie in $I_0$,
$I_1$, or $I_2$. If $x\in I_1$, then $\phi(x)\in I_1$ and hence
$v\leq b$, so $v$ can lie only in $I_0$ or $I_1$. Finally, if
$x\in I_2$, then $\phi(x)\in I_0$ and hence $v\leq a$, so
$v$ must lie in $I_0$.
Consequently, every pair $(x,v)$ satisfying $v\leq\phi(x)$ belongs to
one of the following six regions:
\begin{equation*}
 I_0\times I_0,\quad
 I_0\times I_1,\quad
 I_0\times I_2,\quad
 I_1\times I_0,\quad
 I_1\times I_1,\quad
 I_2\times I_0.
\end{equation*}
Within each of these regions, the constraint $v\leq\phi(x)$ remains in
force.

\textbf{Suppose first that $(x,v)\in I_1\times I_1$.} Then
$
 \phi(x)=\frac{d_1}{x}
$
and 
$
 \phi(v)=\frac{d_1}{v}.
$
Put $z=xv$. Since $x,v\geq a>0$ and $v\leq \phi(x)=d_1/x$, we have
\begin{equation*}
 a^2\leq z\leq d_1.
\end{equation*}
Expanding $G_{p,q}$ and using $pq=d_1$ gives
\begin{equation*}
 G_{p,q}(x,v)
 =
 (1-p)(1-q)z
 +
 \bigl((1-p)q+p(1-q)\bigr)d_1
 +
 \frac{d_1^3}{z}.
\end{equation*}
Denote the right-hand side by
$
 H(z).
$
Since $z>0$, we have
\begin{equation*}
 H''(z)=\frac{2d_1^3}{z^3}>0.
\end{equation*}
Thus, $H$ is convex on $[a^2,d_1]$. For any
$z\in[a^2,d_1]$, write
$
 z=(1-\theta)a^2+\theta d_1
$
for some $\theta\in[0,1]$. Convexity gives
\begin{equation*}
 H(z)
 \leq
 (1-\theta)H(a^2)+\theta H(d_1)
 \leq
 \max\{H(a^2),H(d_1)\}.
\end{equation*}
Thus, it suffices to verify the desired inequality at
$z=a^2$ and $z=d_1$. 
At $z=d_1$, we have $xv=d_1$. Since $x,v\in I_1$,
\begin{equation*}
 \phi(x)=\frac{d_1}{x}=v,
 \qquad
 \phi(v)=\frac{d_1}{v}=x.
\end{equation*}
Therefore,
\begin{align*}
 G_{p,q}(x,v)
 =
 \bigl((1-p)x+p\phi(v)\bigr)
 \bigl((1-q)v+q\phi(x)\bigr)=xv=d_1.
\end{align*} At $z=a^2$, the inequalities $x,v\geq a$ and the identity
$xv=a^2$ imply that $x=v=a$. Since $d_1=ab$, we have
\begin{equation*}
 \phi(a)=\frac{d_1}{a}=b.
\end{equation*}
Therefore,
\begin{align*}
 G_{p,q}(a,a)
 =
 \bigl((1-p)a+p\phi(a)\bigr)
 \bigl((1-q)a+q\phi(a)\bigr)=
 \bigl(a+p(b-a)\bigr)
 \bigl(a+q(b-a)\bigr).
\end{align*}
Since
$
 (1-2p)(1-2q)\geq0,
$
we have
\begin{equation*}
 p+q\leq\frac12+2pq=\frac12+2ab.
\end{equation*}
Consequently,
\begin{align*}
 d_1-
 \bigl(a+p(b-a)\bigr)
 \bigl(a+q(b-a)\bigr)&=
 a(b-a)\bigl(1-p-q-b(b-a)\bigr)\\
 &\geq
 a(b-a)\left(\frac12-b(a+b)\right)\geq0.
\end{align*}
where the last inequality follows from $b\leq1/2$ and $a+b\leq1$.
Therefore, $G_{p,q}(x,v)\leq d_1$ on $I_1\times I_1$.

\textbf{We next consider $(x,v)\in I_0\times I_1$.} Here,
$
 \phi(x)=h_K(x)
$
and
$
 \phi(v)=\frac{d_1}{v}.
$
Recall that $G_{p,q}(x,v)
 =
 \bigl((1-p)x+p\phi(v)\bigr)
 \bigl((1-q)v+q\phi(x)\bigr).$
Since $x\leq a\leq1/2<1$, we have $h_K(x)>0$.
For fixed $x$, differentiation with respect to $v$ gives
\begin{equation*}
 \frac{\partial^2G_{p,q}}{\partial v^2}
 =
 \frac{2d_1^2h_K(x)}{v^3}>0.
\end{equation*}
Thus, for each fixed $x\in I_0$, the function
$v\mapsto G_{p,q}(x,v)$ is convex on $[a,b]$. As in the preceding
convexity argument, this implies
\begin{equation*}
 G_{p,q}(x,v)
 \leq
 \max\bigl\{G_{p,q}(x,a),G_{p,q}(x,b)\bigr\}
 \qquad\text{for every }v\in[a,b].
\end{equation*}
Therefore, it suffices to verify the desired inequality at
$v=a$ and $v=b$.

Put
$
 D_x=d_1+cx.
$
Substituting $q=d_1/p$ and
$h_K(x)=d_1(1-x)/D_x$, and then collecting powers of $x$, gives
\begin{align}
 d_1-G_{p,q}(x,a)
 &=
 \frac{a(b-x)Q_a(x)}{pD_x},
 \label{eq:Qa-factor}\\
 d_1-G_{p,q}(x,b)
 &=
 \frac{b(a-x)Q_b(x)}{pD_x},
 \label{eq:Qb-factor}
\end{align}
where $Q_a$ and $Q_b$ are the following affine functions:
\begin{align*}
 Q_a(x)
 &=
 \left(1-\frac{x}{a}\right)
 abp(1-b-p+ab)
 +
 \frac{x}{a}\,a(a-1)R_a(p),\\
 Q_b(x)
 &=
 \left(1-\frac{x}{a}\right)
 abp(1-a-p+ab)
 +
 \frac{x}{a}\,aR_b(p),
\end{align*}
with
\begin{align*}
 R_a(p)
 &=
 p^2+p(b^2-ab-1)+ab,\\
 R_b(p)
 &=
 (1-a)p(1-p)-ab(1-b).
\end{align*}

Since $pq=d_1=ab$ and $q\leq1/2$, we have
$p=ab/q\geq2ab$.
Since $0<a,b,p\leq1/2$, we have
\begin{equation*}
 Q_a(0)=abp(1-b-p+ab)\geq0,\qquad
 Q_b(0)=abp(1-a-p+ab)\geq0.
\end{equation*}
The polynomial $R_a(p)$ is convex, and
\begin{align*}
 R_a(2ab)
 =
 ab\bigl(2b(a+b)-1\bigr)\leq0,\qquad
 R_a\left(\frac12\right)
 =
 \frac{2b(a+b)-1}{4}\leq0.
\end{align*}
As in the preceding
convexity argument, $R_a(p)\leq0$ for $2ab\leq p\leq1/2$. Since
$a(a-1)<0$, it follows that
\begin{equation*}
 Q_a(a)=a(a-1)R_a(p)\geq0.
\end{equation*}
As $Q_a$ is affine, $Q_a(x)\geq0$ for every $x\in[0,a]$.
Similarly, $R_b(p)$ is concave, and
\begin{align*}
 R_b(2ab)
 =
 ab(2a-1)(2ab-b-1)\geq0,\qquad
 R_b\left(\frac12\right)
 =
 \frac{1-2a+a(1-2b)^2}{4}\geq0.
\end{align*}
Hence, $R_b(p)\geq0$ for $2ab\leq p\leq1/2$, and therefore
\begin{equation*}
 Q_b(a)=aR_b(p)\geq0.
\end{equation*}
Because $Q_b$ is affine, $Q_b(x)\geq0$ on $[0,a]$.
Equations \eqref{eq:Qa-factor} and \eqref{eq:Qb-factor} now show that
\begin{equation*}
 G_{p,q}(x,v)\leq d_1
 \qquad\text{on }I_0\times I_1.
\end{equation*}
Since $pq=d_1$ and $p\leq1/2$, we also have
$q=d_1/p\geq2d_1$. Hence, $2d_1\leq q\leq1/2$.
In  view of  $G_{p,q}(x,v)=G_{q,p}(v,x)$
and $2d_1\leq q\leq1/2$, \textbf{the same argument proves the result on
$I_1\times I_0$}. Indeed, the constraint $v\leq\phi(x)$ is equivalent to
$x\leq\phi(v)$. Thus,
after interchanging $(x,p)$ with $(v,q)$, a feasible point in
$I_1\times I_0$ becomes a feasible point in $I_0\times I_1$.

\textbf{It remains to treat the three regions on which both occurrences of
$\phi$ are equal to $h_K$.} Write
\begin{equation*}
 A=(1-p)x+ph_K(v),
 \qquad
 B=(1-q)v+qh_K(x),
 \qquad
 G=AB,\qquad  D_t=d_1+ct.
\end{equation*}
Observe that $G= G_{p,q}(x,v)$, and
\begin{equation*}
 h_K'(t)
 =
 -\frac{d_1(d_1+c)}{D_t^2},
 \qquad
 h_K''(t)
 =
 \frac{2cd_1(d_1+c)}{D_t^3}.
\end{equation*}
Consequently,
\begin{align*}
 G_x=(1-p)B+qh_K'(x)A,\qquad G_v=ph_K'(v)B+(1-q)A,
\end{align*}
and
\begin{align*}
 G_{xx}
 &=
 \frac{2qd_1(d_1+c)}{D_x^3}
 \bigl(cp h_K(v)-d_1(1-p)\bigr),\\
 G_{vv}
 &=
 \frac{2pd_1(d_1+c)}{D_v^3}
 \bigl(cq h_K(x)-d_1(1-q)\bigr).
\end{align*}
Here and below, subscripts on $G$ denote partial derivatives with
respect to the indicated variables.
Suppose that $G$ attains a local maximum at an
interior point $(x,v)$ of one of these regions. Then
\begin{equation*}
 G_x=G_v=0,\qquad
 G_{xx}\leq0,
 \qquad
 G_{vv}\leq0,
 \qquad
 G_{xx}G_{vv}-G_{xv}^2\geq0.
\end{equation*}
At an interior point, $x,v<1$, and hence
$h_K(x),h_K(v)>0$. Since $p,q>0$, we have
\begin{equation*}
 A\geq ph_K(v)>0,
 \qquad
 B\geq qh_K(x)>0.
\end{equation*}
The equations $G_x=G_v=0$ imply
\begin{equation*}
 pqh_K'(x)h_K'(v)
 =
 (1-p)(1-q)=:L.
\end{equation*}
Moreover,
\begin{equation*}
 G_{xv}
 =
 (1-p)(1-q)+pqh_K'(x)h_K'(v)
 =
 2L>0.
\end{equation*}
Define
\begin{equation*}
 \Delta_p=d_1(1-p)-cp h_K(v),
 \qquad
 \Delta_q=d_1(1-q)-cq h_K(x).
\end{equation*}
Since $ G_{xx}\leq0$ and $
 G_{vv}\leq0$, we have $\Delta_p,\Delta_q\geq0$. If $\Delta_p=0$, then $G_{xx}=0$, while if $\Delta_q=0$, then
$G_{vv}=0$. In either case, since $G_{xv}>0$, we have
\begin{equation*}
 G_{xx}G_{vv}-G_{xv}^2
 =
 -G_{xv}^2<0,
\end{equation*}
a contradiction. Therefore, we have
$
 \Delta_p, \Delta_q>0.
$
Since $pq=d_1$, $G_{xv}=2L$ and
\begin{equation*}
 L=pqh_K'(x)h_K'(v)
  =\frac{d_1^3(d_1+c)^2}{D_x^2D_v^2},
\end{equation*}
the formulas for $G_{xx}$ and $G_{vv}$ give
\begin{align*}
 \frac{G_{xx}G_{vv}}{G_{xv}^2}
 =
 \frac{\Delta_p\Delta_q}{LD_xD_v}<
 \frac{d_1^2(1-p)(1-q)}{LD_xD_v}=
 \frac{d_1^2}{D_xD_v}
 \leq1.
\end{align*}
Thus,
\begin{equation*}
 G_{xx}G_{vv}-G_{xv}^2<0,
\end{equation*}
again a contradiction. Therefore, none of
these three regions has an interior local maximum.

It remains to check their boundaries. We first consider the boundary $v=0$. Since $h_K(0)=1$, we have
\begin{equation*}
 G(x,0)
 =
 \bigl(p+(1-p)x\bigr)qh_K(x).
\end{equation*}
Using
\begin{equation*}
 q=\frac{d_1}{p},
 \qquad
 h_K(x)=\frac{d_1(1-x)}{D_x},
\end{equation*}
we obtain
\begin{equation*}
 G(x,0)
 =
 \frac{d_1^2(1-x)\bigl(p+(1-p)x\bigr)}
 {pD_x}.
\end{equation*}
Consequently,
\begin{align*}
 d_1-G(x,0)
 &=
 \frac{d_1}{pD_x}
 \left[
 pD_x-d_1(1-x)\bigl(p+(1-p)x\bigr)
 \right]\\
&=\frac{d_1x}{pD_x}
 \bigl(pc-d_1(1-2p)+d_1(1-p)x\bigr).
\end{align*}
Recall that $d_1=pq$ and
$
 c=1-a-b.
$
Since $q\leq1/2$, we have
\begin{align*}
 q(1-2p)
 =q-2pq=q-2d_1\leq\frac12-2d_1.
\end{align*}
Moreover,
\begin{equation*}
 c-\left(\frac12-2d_1\right)
 =
 2\left(\frac12-a\right)
  \left(\frac12-b\right)
 \geq0.
\end{equation*}
Therefore,
$
 q(1-2p)\leq c.
$
Multiplying by $p>0$ gives
$
 d_1(1-2p)=pq(1-2p)\leq pc.
$
Thus,
\begin{equation*}
 pc-d_1(1-2p)+d_1(1-p)x\geq0.
\end{equation*}
It follows that
\begin{equation*}
 d_1-G(x,0)\geq0.
\end{equation*}

On $x=0$, direct simplification gives
\begin{equation*}
 d_1-G(0,v)
 =
 \frac{d_1v
 \bigl(c+2d_1-p+v(p-d_1)\bigr)}
 {d_1+cv}.
\end{equation*}
Observe that
$
 p-d_1=p(1-q)>0
$
and
\begin{align*}
 c+2d_1-p
 \geq
 c+2d_1-\frac12=
 2\left(\frac12-a\right)
  \left(\frac12-b\right)
 \geq0.
\end{align*}
Therefore, $G(0,v)\leq d_1$. 
The edge $v=a$ was treated in \eqref{eq:Qa-factor}, and the edge
$x=a$ follows from the symmetric case $I_1\times I_0$. Thus all four
boundary edges of $I_0\times I_0$, namely $x=0$, $v=0$, $x=a$, and
$v=a$, have been treated. Since $G$ has no local maximum in the
interior, it follows that $G\leq d_1$ throughout
$I_0\times I_0$.

Finally, consider the feasible part of $I_0\times I_2$. Since
$x\in[0,a]$, $v\in[b,1]$, and $v\leq\phi(x)=h_K(x)$, it is described
by
\begin{equation*}
 0\leq x\leq a,
 \qquad
 b\leq v\leq h_K(x).
\end{equation*}
Recall that $h_K(0)=1$ and $h_K(a)=b$. Hence, the boundary of the
feasible region consists of
\begin{equation*}
 \{(0,v):b\leq v\leq1\},
 \qquad
 \{(x,b):0\leq x\leq a\},
 \qquad
 \{(x,h_K(x)):0\leq x\leq a\}.
\end{equation*}
When $x=a$, the inequalities
$
 b\leq v\leq h_K(a)=b
$
force $v=b$. Thus, $x=a$ contributes only the point $(a,b)$, which
already belongs to both the second and the third boundary parts.
Similarly, $v=1$ forces $x=0$, giving the point $(0,1)$, which already belongs to both the first and the third boundary parts.
The boundary $x=0$ was treated above, while the boundary $v=b$ was
treated in \eqref{eq:Qb-factor}. On the curved boundary
$v=h_K(x)=\phi(x)$, the involution property of $\phi$ gives
$
 \phi(v)=\phi(\phi(x))=x.
$
Therefore,
\begin{align*}
 G_{p,q}(x,v)
 =
 \bigl((1-p)x+p\phi(v)\bigr)
 \bigl((1-q)v+q\phi(x)\bigr)=xv=xh_K(x).
\end{align*}
Since $0\leq x\leq a\leq b$, we have
\begin{equation*}
 d_1-xh_K(x)
 =
 \frac{d_1(a-x)(b-x)}{d_1+cx}
 \geq0.
\end{equation*}
Thus, $G_{p,q}\leq d_1$ on every boundary part. The preceding
second-derivative argument shows that $G_{p,q}$ has no local maximum
in the interior. Hence, $G_{p,q}\leq d_1$ throughout the feasible part
of $I_0\times I_2$. As before, interchanging $(x,p)$ with $(v,q)$ maps the feasible part
of $I_2\times I_0$ to that of $I_0\times I_2$. Since
$
 G_{p,q}(x,v)=G_{q,p}(v,x),
$
the same argument applies to $I_2\times I_0$.
This completes the proof.
\end{proof}

\subsection{Induction and equality in the open range}
\label{sec:product-proof}

We now state and prove the separation inequality required in
\eqref{eq:separation}.

\begin{proposition}\label{prop:separation}
Let $\mathbf p=(p_1,\ldots,p_n), \mathbf q=(q_1,\ldots,q_n)\in(0,\frac{1}{2})^n$. Set $d_i=p_iq_i$ for
$i\in[n]$, and assume that
$
 d_1=\max_{i\in[n]}d_i.
$
Then, for every increasing family
$\mathcal{A}\subseteq2^{[n]}$,
\begin{equation*}
 \mu_{\mathbf p}(\mathcal{A})
 \bigl(1-\mu_{\mathbf 1-\mathbf q}(\mathcal{A})\bigr)
 \leq d_1.
\end{equation*}
\end{proposition}
\begin{proof}
We proceed by induction on $n$. For $n=1$, the only increasing families are
$\emptyset$, $\{\{1\}\}$, and $2^{\{1\}}$.  The corresponding values of
$
\mu_{\mathbf p}(\mathcal A)
\bigl(1-\mu_{\mathbf 1-\mathbf q}(\mathcal A)\bigr)
$
are respectively $0$, $p_1q_1=d_1$, and $0$
Moreover, equality holds if and only if $\mathcal A=\{\{1\}\}$.
Thus, the conclusion follows.

Now let $n\geq2$ and assume that the proposition holds in dimension
$n-1$.  Set
$
 p=p_n$
and
$
q=q_n.
$
Recall that $\mathcal{A}(\bar n)$ and $\mathcal{A}(n)$ are the sections
of $\mathcal{A}$ according as coordinate $n$ is absent or present,
respectively, while $\mathbf p_{-n}$ and $\mathbf q_{-n}$ are obtained
from $\mathbf p$ and $\mathbf q$ by deleting their $n$th coordinates.
Since $\mathcal{A}$ is increasing, both sections are increasing families
on $[n-1]$, and
$
 \mathcal{A}(\bar n)\subseteq\mathcal{A}(n).
$
Define
\begin{align*}
 x&=\mu_{\mathbf p_{-n}}\bigl(\mathcal{A}(\bar n)\bigr),
&
 y&=\mu_{\mathbf p_{-n}}\bigl(\mathcal{A}(n)\bigr),\\
 u&=1-\mu_{\mathbf 1-\mathbf q_{-n}}
       \bigl(\mathcal{A}(\bar n)\bigr),
&
 v&=1-\mu_{\mathbf 1-\mathbf q_{-n}}
       \bigl(\mathcal{A}(n)\bigr).
\end{align*}
It follows from $\mathcal{A}(\bar n)\subseteq\mathcal{A}(n)$ that
\[
 0\leq x\leq y\leq1,
 \qquad
 0\leq v\leq u\leq1.
\]
Since
$
 d_1=\max_{i\in[n]}d_i,
$
 we
have
$
 d_1=\max_{i\in[n-1]}d_i.
$
The induction hypothesis applied separately to
$\mathcal{A}(\bar n)$ and $\mathcal{A}(n)$ therefore gives
\[
 xu\leq d_1,
 \qquad
 yv\leq d_1.
\]

Recall that
\[
 K_i=\frac{(1-p_i)(1-q_i)}{p_iq_i},
 \qquad i\in[n].
\]
For the truncated probability vectors $\mathbf p_{-n}$ and
$\mathbf q_{-n}$, put
\[
 K_{-n}=\min_{i\in[n-1]}K_i,
 \qquad
 h_{K_{-n}}(t)
 =
 \frac{1-t}{1+(K_{-n}-1)t}.
\]
Applying Lemma~\ref{lem:logodds} to the two sections yields
\[
 u\leq h_{K_{-n}}(x),
 \qquad
 v\leq h_{K_{-n}}(y).
\]
Moreover, Lemma~\ref{lem:range}, applied to
$\mathbf p_{-n}$ and $\mathbf q_{-n}$, gives
\[
 \frac{1}{2(K_{-n}+1)}
 <
 d_1
 \leq
 \frac{1}{(1+\sqrt{K_{-n}})^2}.
\]
Thus, all the hypotheses of Lemma~\ref{lem:bellman} are satisfied with
$K=K_{-n}$ and with the parameter $d_1$.

Finally, conditioning on coordinate $n$ gives
\begin{align*}
 \mu_{\mathbf p}(\mathcal{A})
 &=(1-p)x+py,\\
 1-\mu_{\mathbf 1-\mathbf q}(\mathcal{A})
 &=qu+(1-q)v.
\end{align*}
Hence, Lemma~\ref{lem:bellman} implies
\begin{align*}
 \mu_{\mathbf p}(\mathcal{A})
 \bigl(1-\mu_{\mathbf 1-\mathbf q}(\mathcal{A})\bigr)
 &\leq \max\{d_1,pq\}=d_1,
\end{align*}
where the last equality follows from
$pq=p_nq_n=d_n\leq d_1$.  This completes the induction.
\end{proof}

In the following, we complete the proof of the inequality in
Theorem~\ref{thm:main-product} under the present assumption
$
 \mathbf p,\mathbf q\in(0,1/2)^n.
$
By the reduction established above, it suffices to consider increasing
cross-intersecting families $\mathcal{A}$ and $\mathcal{B}$.  In this
case $\mathcal{B}\subseteq\mathcal{A}^*$, and
Lemma~\ref{lem:transversal}, followed by
Proposition~\ref{prop:separation}, gives
\begin{align*}
 \mu_{\mathbf p}(\mathcal{A})\mu_{\mathbf q}(\mathcal{B})
 &\leq
 \mu_{\mathbf p}(\mathcal{A})
 \bigl(1-\mu_{\mathbf 1-\mathbf q}(\mathcal{A})\bigr)\leq
 \max_{i\in[n]}p_iq_i.
\end{align*}

We now determine the equality cases under the same assumption.  Equality
in Theorem~\ref{thm:main-product} requires equality in both inequalities
above.  

 First suppose that $pq=d_1$ in
Lemma~\ref{lem:bellman}. Let $a,b$ be the two roots introduced in its
proof. We claim that the six-region analysis is strict except when
\begin{equation}
 (x,v)=(0,0),
 \qquad\text{or}\qquad
 x,v\in[a,b]\quad\text{and}\quad xv=d_1.
 \label{eq:bellman-equality-points}
\end{equation}

To establish this characterization, consider first the region
$I_1\times I_1$. The expression considered as a function
of $z=xv$ is strictly convex. Its value at $z=a^2$ is strictly smaller
than $d_1$ unless $a=b$, in which case $a^2=d_1$. Its value at
$z=d_1$ is exactly $d_1$.

Since $K>1$, the upper bound in Lemma~\ref{lem:range} gives
$d_1<1/4$.  Hence,
$
 a\leq\sqrt{ab}=\sqrt{d_1}<\frac12.
$
Moreover, using $ab=d_1$ and
$a+b=1-(K-1)d_1$, we obtain
\[
 \left(\frac12-a\right)\left(\frac12-b\right)
 =
 \frac{2(K+1)d_1-1}{4}>0,
\]
where the strict inequality follows from the strict lower bound in
Lemma~\ref{lem:range}.  
Since $a<1/2$, this implies $b<1/2$. Together with $a+b<1$, this gives
\begin{equation*}
 2b(a+b)-1<0.
\end{equation*}

Since $2b(a+b)-1<0$, the endpoint values of $Q_a$ and $Q_b$ used in the proof
of Lemma~\ref{lem:bellman} are positive. Since $Q_a$ and $Q_b$ are affine functions of $x$, positivity at
$x=0$ and $x=a$ implies
\[
 Q_a(x)>0,
 \qquad
 Q_b(x)>0
 \qquad (0\leq x\leq a).
\]
The factorizations \eqref{eq:Qa-factor} and
\eqref{eq:Qb-factor} therefore give
\begin{align*}
 d_1-G_{p,q}(x,a)
 &=
 \frac{a(b-x)Q_a(x)}{pD_x}\geq0,\\
 d_1-G_{p,q}(x,b)
 &=
 \frac{b(a-x)Q_b(x)}{pD_x}\geq0.
\end{align*}
Since $pD_x>0$, equality in either expression is possible only when
$b=x$ in the first or $a=x$ in the second.  For the endpoint $v=a$, equality requires $x=b$.  Since
$x\leq a\leq b$, this forces $x=a=b$, and hence $xv=ab=d_1$.
At the endpoint $v=b$, equality requires
$x=a$, because all the other factors are positive.  Consequently,
$xv=ab=d_1$.
For each fixed $x\in I_0$, the function
$v\mapsto G_{p,q}(x,v)$ is convex on $I_1=[a,b]$, so its maximum is
attained at $v=a$ or $v=b$.  Hence, equality on
$I_0\times I_1$ can occur only at an endpoint with $xv=d_1$.
By the symmetry
$
 G_{p,q}(x,v)=G_{q,p}(v,x),
$
the same conclusion holds on $I_1\times I_0$.

On the three remaining regions
$I_0\times I_0$, $I_0\times I_2$, and $I_2\times I_0$, the boundary
identities obtained in the proof of Lemma~\ref{lem:bellman} give
\begin{align*}
 d_1-G_{p,q}(x,0)
 &=
 \frac{d_1x}{pD_x}
 \left[
  p\bigl(c-q(1-2p)\bigr)+d_1(1-p)x
 \right],\\
 d_1-G_{p,q}(0,v)
 &=
 \frac{d_1v}{D_v}
 \left[
  c+2d_1-p+v(p-d_1)
 \right].
\end{align*}
Here,  since $a, b<1/2$, we have
\[
 q(1-2p)\leq\frac12-2d_1<c
\]
and
\[
 c+2d_1-p>\frac12-p\geq0.
\]
Thus, the first difference is positive for $x>0$, and the second is
positive for $v>0$.  Since $G_{p,q}(0,0)=pq=d_1$, equality on these
boundary edges occurs only at $(0,0)$.
On the curved boundary, we have
\begin{equation*}
 d_1-xh_K(x)
 =
 \frac{d_1(a-x)(b-x)}{d_1+cx},
\end{equation*}
whose only zero for $0\leq x\leq a$ is $x=a$. The second-derivative
argument excludes an equality maximum in the interior. This proves
\eqref{eq:bellman-equality-points}.

The constraints give
$
 u\leq\phi(x)
$
and 
$
 y\leq\phi(v).
$
Replacing $u$ by $\phi(x)$ and $y$ by $\phi(v)$ can only increase
the two factors in \eqref{eq:bellman}.  In an equality case their
product is $d_1>0$, and hence both factors are positive.  Any strict
increase in either $u$ or $y$ would therefore make the product exceed
$d_1$, contradicting Lemma~\ref{lem:bellman}.  Consequently,
$$
 u=\phi(x),
 \qquad
 y=\phi(v).
$$
Combining \eqref{eq:bellman-equality-points} with
$u=\phi(x)$ and $y=\phi(v)$ gives two possibilities.  If
$(x,v)=(0,0)$, then $\phi(0)=1$, so $y=u=1$.  Otherwise,
$x,v\in I_1$ and $xv=d_1$.  Since $\phi(t)=d_1/t$ on $I_1$,
\[
 u=\phi(x)=\frac{d_1}{x}=v,
 \qquad
 y=\phi(v)=\frac{d_1}{v}=x.
\]
Consequently, the equality configurations are
\begin{enumerate}[label=\textup{(\roman*)},leftmargin=2.2em]
 \item $x=v=0$ and $y=u=1$;
 \item $x=y$, $u=v$, and $xu=d_1$.
\end{enumerate}

Since $pq=d_n\leq d_1$, alternative \textup{(i)} can occur only when
$d_n=d_1$; otherwise its value in \eqref{eq:bellman} would be
$pq=d_n<d_1$.  Moreover, we have
$$
 \mathcal{A}(\bar n)=\emptyset,
 \qquad
 \mathcal{A}(n)=2^{[n-1]}.
$$
Hence, $\mathcal{A}=\mathcal{S}_n$.
In alternative \textup{(ii)}, the inclusion
$\mathcal{A}(\bar n)\subseteq\mathcal{A}(n)$ and the equality $x=y$
imply
$$
 \mathcal{A}(\bar n)=\mathcal{A}(n)=:\mathcal{C}.
$$
Moreover, $xu=d_1$ means that $\mathcal{C}$ attains equality in
dimension $n-1$.  By induction,
$\mathcal{C}=\{C\subseteq[n-1]:j\in C\}$ for some $j<n$ with $d_j=d_1$, and hence
$\mathcal{A}=\mathcal{S}_j$.  Thus, equality in
Proposition \ref{prop:separation} occurs precisely for
$$
 \mathcal{A}=\mathcal{S}_j,
 \qquad
 p_jq_j=d_1.
$$
Returning to the cross-intersecting pair, equality forces
$\mathcal{B}=\mathcal{A}^*$ and equality in
Proposition~\ref{prop:separation}.  Therefore,
$\mathcal{A}=\mathcal{S}_j$ for some $j$ with $p_jq_j=d_1$.
Since $\mathcal{S}_j^*=\mathcal{S}_j$, it follows that
$$
 \mathcal{A}=\mathcal{B}=\mathcal{S}_j.
$$
This proves the equality statement of
Theorem~\ref{thm:main-product} for
$\mathbf p,\mathbf q\in(0,1/2)^n$.

Suppose that $pq<d_1$ and that equality holds in
\eqref{eq:bellman}.  Write
\[
 L_{p,q}
 =
 ((1-p)x+py)(qu+(1-q)v).
\]
The constraints used in the proof of Lemma~\ref{lem:bellman} give
$
 y\leq\phi(v),
 u\leq\phi(x).
$
Consequently,
\[
 L_{p,q}
 \leq
 G_{p,q}(x,v)
 =
 \bigl((1-p)x+p\phi(v)\bigr)
 \bigl(q\phi(x)+(1-q)v\bigr)
 \leq d_1.
\]
Since $L_{p,q}=d_1$, equality must hold throughout.  Moreover, both
factors of $L_{p,q}$ are positive because their product is $d_1>0$.
Thus, either strict inequality $y<\phi(v)$ or
$u<\phi(x)$ would strictly increase the corresponding factor, which is
impossible.  Hence, we have
$$
 y=\phi(v),
 \qquad
 u=\phi(x),
 \qquad
 G_{p,q}(x,v)=d_1.
$$
Choose $\widetilde p\geq p$ and $\widetilde q\geq q$, as in the proof
of Lemma~\ref{lem:bellman}, such that
$
 \widetilde p\widetilde q=d_1.
$
Since $x\leq\phi(v)$ and $v\leq\phi(x)$,
$G_{p,q}(x,v)$ is nondecreasing in both $p$ and $q$.  Therefore
\[
 d_1
 =
 G_{p,q}(x,v)
 \leq
 G_{\widetilde p,\widetilde q}(x,v)
 \leq d_1.
\]
Thus, equality also holds for the saturated parameters
$\widetilde p,\widetilde q$, to which
\eqref{eq:bellman-equality-points} applies.
If the saturated configuration were of type~\textup{(i)}, then
$x=v=0$.  Since $\phi(0)=1$, the identities above would give
$y=u=1$, and hence
\[
 L_{p,q}
 =
 ((1-p)0+p)(q+(1-q)0)
 =
 pq<d_1,
\]
a contradiction.  Thus, the saturated configuration must be of type~\textup{(ii)}.
 Hence,
\[
 x=y,\qquad u=v,\qquad xu=d_1.
\]
The same section argument then shows that
$\mathcal{A}=\mathcal{B}=\mathcal{S}_j$ for some $j<n$ with $d_j=d_1$.

\subsection{Boundary equality cases}
\label{sec:endpoint}

The preceding subsection proves Theorem~\ref{thm:main-product} when all
coordinate probabilities are strictly less than $1/2$.  We now extend
the inequality to $(0,1/2]^n$ and determine all equality cases on the
boundary.

\begin{proof}[Completion of the proof of Theorem~\ref{thm:main-product}]
Recall that
\begin{equation*}
 d_i=p_iq_i,
 \qquad
 d_1=\max_{i\in[n]}d_i,
 \qquad
 I=\{i\in[n]:d_i=d_1\}.
\end{equation*}
For $0<r<1$, the vectors $r\mathbf p$ and $r\mathbf q$ belong to
$(0,1/2)^n$, and
$
 \max_{i\in[n]}(rp_i)(rq_i)=r^2d_1.
$
The result in the open range therefore gives
\begin{equation*}
 \mu_{r\mathbf p}(\mathcal{A})
 \mu_{r\mathbf q}(\mathcal{B})
 \leq r^2d_1.
\end{equation*}
Both measures are continuous functions of $r$.  Taking the limit as
$r\to1^{-}$ yields
\begin{equation*}
 \mu_{\mathbf p}(\mathcal{A})
 \mu_{\mathbf q}(\mathcal{B})
 \leq d_1.
\end{equation*}

Suppose that the equality holds.  By
Lemma~\ref{lem:upclosure}, the families
$\mathcal{A}^{\up}$ and $\mathcal{B}^{\up}$ remain cross-intersecting.
Consequently,
$\mathcal{A}$ and $\mathcal{B}$ are increasing.
Set
$
 J=\{i\in[n]:p_i=q_i=1/2\}.
$

\textbf{First suppose that $J=\emptyset$}.  Then
$p_i+q_i<1$ for every $i\in[n]$, and hence
$
 K_i=\frac{(1-p_i)(1-q_i)}{p_iq_i}>1.
$
Although Proposition~\ref{prop:separation} is stated for the open
range, its inductive proof extends to the present setting.  We explain
why each ingredient remains applicable.
The generalized Margulis--Russo formula and the variance estimate were
proved for arbitrary probability vectors in $(0,1)^n$, so they require
no modification.  Moreover, the proof of
Lemma~\ref{lem:logodds} uses the assumptions
$p_i,q_i<1/2$ only to ensure that $p_i+q_i<1$ and hence $K_i>1$.
These properties still hold when $J=\emptyset$.  In particular, the
quantities $s_i(t)$ defined there satisfy
\begin{equation*}
 0<s_i(t)<1,
 \qquad  0\leq t\leq1.
\end{equation*}
Thus,  Lemma~\ref{lem:logodds} continues to hold.
The only change in Lemma~\ref{lem:range} is that its lower bound may
now be attained.  To see this, choose $j\in[n]$ such that $K_j=K$.
Since $p_j,q_j\leq1/2$,
\begin{align*}
 0\leq(1-2p_j)(1-2q_j)=-1+2(K_j+1)d_j.
\end{align*}
Therefore,
\begin{equation*}
 d_j\geq\frac{1}{2(K+1)},
\end{equation*}
and hence
\begin{equation}
 \frac{1}{2(K+1)}
 \leq d_1
 \leq
 \frac{1}{(1+\sqrt K)^2}.
 \label{eq:closed-range}
\end{equation}
The proof of the upper bound is unchanged.  Finally,
Lemma~\ref{lem:bellman} is already stated under the weak range
\eqref{eq:closed-range} and permits its scalar parameters to satisfy
$0<p,q\leq1/2$.  It therefore applies without modification.

For completeness, consider the inductive step in
Proposition~\ref{prop:separation}.  Expose coordinate $n$ and set
$p=p_n$, $q=q_n$.  With the section notation introduced earlier, put
\begin{align*}
 x&=\mu_{\mathbf p_{-n}}\bigl(\mathcal{A}(\bar n)\bigr),
&
 y&=\mu_{\mathbf p_{-n}}\bigl(\mathcal{A}(n)\bigr),\\
 u&=1-\mu_{\mathbf 1-\mathbf q_{-n}}
       \bigl(\mathcal{A}(\bar n)\bigr),
&
 v&=1-\mu_{\mathbf 1-\mathbf q_{-n}}
       \bigl(\mathcal{A}(n)\bigr).
\end{align*}
Since $\mathcal{A}$ is increasing,
$
 \mathcal{A}(\bar n)\subseteq\mathcal{A}(n),
$
and therefore
\begin{equation*}
 0\leq x\leq y\leq1,
 \qquad
 0\leq v\leq u\leq1.
\end{equation*}
Also,
$
 d_1=\max_{i\in[n-1]}d_i.
$
The induction hypothesis gives
$
 xu\leq d_1,
 yv\leq d_1.
$
For the truncated vectors, put
\begin{equation*}
 K_{-n}=\min_{i\in[n-1]}K_i.
\end{equation*}
The endpoint version of Lemma~\ref{lem:logodds} gives
$
 u\leq h_{K_{-n}}(x),
 v\leq h_{K_{-n}}(y),
$
while the preceding extension of Lemma~\ref{lem:range} gives
\begin{equation*}
 \frac{1}{2(K_{-n}+1)}
 \leq d_1
 \leq
 \frac{1}{(1+\sqrt{K_{-n}})^2}.
\end{equation*}
All the hypotheses of Lemma~\ref{lem:bellman} are therefore satisfied.
Since $pq=d_n\leq d_1$, it follows that
\begin{align*}
 \mu_{\mathbf p}(\mathcal{A})
 \bigl(1-\mu_{\mathbf 1-\mathbf q}(\mathcal{A})\bigr)
 &=
 ((1-p)x+py)(qu+(1-q)v)\leq\max\{d_1,pq\}
 =d_1.
\end{align*}
Thus the separation inequality remains valid when $J=\emptyset$.

It remains to verify that equality at the new lower endpoint in
\eqref{eq:closed-range} creates no additional extremizer.  Let $a\leq b$
be the roots introduced in the proof of
Lemma~\ref{lem:bellman}, so that
\begin{equation*}
 ab=d_1,
 \qquad
 a+b=1-(K-1)d_1.
\end{equation*}
If
$
 d_1=\frac{1}{2(K+1)},
$
then
$
 b=\frac12
$ and
$
 a=2d_1<\frac12.
$
Indeed, since $
 F\left(\frac12\right)
 =
 \frac{1}{2(K+1)},
$
the equality $ d_1=\frac{1}{2(K+1)}$ implies that $1/2$ is the larger
root. Hence, $b=1/2$. Using $ab=d_1$, we  obtain
$
 a=\frac{d_1}{b}=2d_1<\frac12,
$
where the last inequality follows from $K>1$.
 If the lower
inequality is strict, then $b<1/2$.  Hence, in every case
\begin{equation*}
 0<a\leq b\leq\frac12,
 \qquad
 a<\frac12,
 \qquad
 a+b<1.
\end{equation*}
In particular,
\begin{equation}
 2b(a+b)-1<0.
 \label{eq:boundary-sign}
\end{equation}

We next check the signs used in the equality analysis of
Lemma~\ref{lem:bellman}.  Under the critical relation
$pq=d_1=ab$, one has
$
 2d_1\leq p\leq\frac12.
$
Recall that
$
 R_a(p)
 =
 p^2+p(b^2-ab-1)+ab
$
and
$  
R_b(p)
 =
 (1-a)p(1-p)-ab(1-b).
$
Then
\begin{align*}
 R_a(2d_1)
 =
 d_1\bigl(2b(a+b)-1\bigr)<0,\qquad
 R_a\left(\frac12\right)=
 \frac{2b(a+b)-1}{4}<0.
\end{align*}
Since $R_a$ is convex, it is negative throughout
$[2d_1,1/2]$.  Similarly,
\begin{align*}
 R_b(2d_1)
 =
 d_1(2a-1)(2d_1-b-1)>0,\qquad
 R_b\left(\frac12\right)
 =
 \frac{1-2a+a(1-2b)^2}{4}
 \geq
 \frac{1-2a}{4}>0.
\end{align*}
Since $R_b$ is concave, it is positive throughout the same interval.
Together with the endpoint formulas for $Q_a$ and $Q_b$, this yields
\begin{equation*}
 Q_a(x)>0,
 \qquad
 Q_b(x)>0,
 \qquad 0\leq x\leq a.
\end{equation*}
Thus, the equality analysis on
$I_0\times I_1$ and $I_1\times I_0$ is unchanged.

On $I_1\times I_1$, the endpoint calculation from the proof of
Lemma~\ref{lem:bellman} gives
\begin{align*}
 d_1-
 \bigl(a+p(b-a)\bigr)
 \bigl(a+q(b-a)\bigr)
 &\geq
 a(b-a)\left(\frac12-b(a+b)\right).
\end{align*}
Since $b\leq1/2$ and $a+b<1$, the right-hand side is positive when
$a<b$.  If $a=b$, then $a^2=d_1$.  Hence, equality on
$I_1\times I_1$ still requires $xv=d_1$.

The possible loss of strictness on the coordinate axes also produces
no new equality.  Under $pq=d_1$, the boundary identities from
Lemma~\ref{lem:bellman} are
\begin{align*}
 d_1-G_{p,q}(x,0)
 &=
 \frac{d_1x}{pD_x}
 \left[
  p\bigl(c-q(1-2p)\bigr)+d_1(1-p)x
 \right],\\
 d_1-G_{p,q}(0,v)
 &=
 \frac{d_1v}{D_v}
 \left[
  c+2d_1-p+v(p-d_1)
 \right].
\end{align*}
Since
$
 c-\left(\frac12-2d_1\right)
 =
 2\left(\frac12-a\right)
  \left(\frac12-b\right)\geq0,
$
we have
$
 q(1-2p)
 =
 q-2d_1
 \leq\frac12-2d_1
 \leq c
$
and
$
 c+2d_1-p
 \geq\frac12-p
 \geq0.
$
If $x>0$, the first bracket contains the strictly positive term
$d_1(1-p)x$.  If $v>0$, the second contains
$v(p-d_1)>0$, because $p-d_1=p(1-q)>0$.  Thus, equality on these
edges occurs only at $(x,v)=(0,0)$. 
On the curved boundary, we have
\begin{equation*}
 d_1-xh_K(x)
 =
 \frac{d_1(a-x)(b-x)}{d_1+cx},
\end{equation*}
whose only zero for $0\leq x\leq a$ is $x=a$.
 Moreover, the interior second-derivative argument is unchanged.  Consequently, the only
critical equality configurations remain
\begin{equation*}
 (x,v)=(0,0),
 \qquad\text{or}\qquad
 x,v\in[a,b]\ \text{and}\ xv=d_1.
\end{equation*}
As in the open case, equality in \eqref{eq:bellman} forces
$
 u=\phi(x)
$
and
$
 y=\phi(v)
$
because replacing $u$ and $y$ by these larger values cannot strictly
increase either positive factor.  Hence, the two possible configurations
are
\begin{enumerate}[label=\textup{(\roman*)},leftmargin=2.2em]
 \item $x=v=0$ and $y=u=1$;
 \item $x=y$, $u=v$, and $xu=d_1$.
\end{enumerate}
The parameter-reduction argument for $pq<d_1$ is also unchanged, since
the monotonicity used there was proved on the full range
$0<p,q\leq1/2$.  It rules out type~\textup{(i)} when $pq<d_1$.

In type~\textup{(i)}, we have $
 \mathcal{A}(\bar n)=\emptyset$ and $
 \mathcal{A}(n)=2^{[n-1]}$.
Thus, $\mathcal{A}=\mathcal{S}_n$ and necessarily $d_n=d_1$.  In
type~\textup{(ii)}, we have
$
 \mathcal{A}(\bar n)=\mathcal{A}(n),
$
and the common section attains equality in dimension $n-1$.
Induction therefore shows that equality in the endpoint extension of
Proposition~\ref{prop:separation} occurs precisely when
$
 \mathcal{A}=\mathcal{S}_j
$
for some $j\in I$.
Returning to the cross-intersecting pair, equality in the transversal
reduction forces $\mathcal{B}=\mathcal{A}^*$.  Since
$\mathcal{S}_j^*=\mathcal{S}_j$, it follows that
$
 \mathcal{A}=\mathcal{B}=\mathcal{S}_j.
$
This proves part~\textup{(i)} of Theorem~\ref{thm:main-product}.

\textbf{Now suppose that $J\neq\emptyset$}.  Then
$
 d_1=\frac14
$
and
$
 I=J$
because $p_iq_i=1/4$ with $p_i,q_i\leq1/2$ holds if and only if
$p_i=q_i=1/2$.
Put
\begin{equation*}
 \alpha=\mu_{\mathbf p}(\mathcal{A}),
 \qquad
 \beta=\mu_{\mathbf 1-\mathbf q}(\mathcal{A}).
\end{equation*}
\textbf{For an increasing family, its product measure is nondecreasing in each
coordinate probability.}  To see this, let
$
 \mathbf r=(r_1,\ldots,r_n)\in(0,1)^n
$
and fix $i\in[n]$.  For $s_i\in[r_i,1)$, define
$
 \mathbf r^{(i)}
 =
 (r_1,\ldots,r_{i-1},s_i,r_{i+1},\ldots,r_n).
$
Recall that $\mathbf r_{-i}$ is the vector obtained from $\mathbf r$ by
deleting its $i$th coordinate.  Conditioning on coordinate $i$ gives
\begin{align*}
 \mu_{\mathbf r^{(i)}}(\mathcal{A})
 -
 \mu_{\mathbf r}(\mathcal{A})
 &=
 (s_i-r_i)
 \left(
  \mu_{\mathbf r_{-i}}\bigl(\mathcal{A}(i)\bigr)
  -
  \mu_{\mathbf r_{-i}}\bigl(\mathcal{A}(\bar i)\bigr)
 \right)\geq0,
\end{align*}
because $\mathcal{A}(\bar i)\subseteq\mathcal{A}(i)$.
Thus, $\mu_{\mathbf r}(\mathcal{A})$ is nondecreasing in every
coordinate of $\mathbf r$.
Since $\mathbf p\leq\mathbf 1-\mathbf q$ coordinatewise and
$\mathcal{A}$ is increasing, we have $\alpha\leq\beta$. 
Moreover,
$\mathcal{B}\subseteq\mathcal{A}^*$.  Applying Lemma~\ref{lem:transversal} yields
\begin{equation*}
 \frac14
 =
 \alpha\mu_{\mathbf q}(\mathcal{B})
 \leq
 \alpha\mu_{\mathbf q}(\mathcal{A}^*)
 =
 \alpha(1-\beta)
 \leq
 \alpha(1-\alpha)
 \leq
 \frac14.
\end{equation*}
 Consequently,
$$
 \alpha=\beta=\frac12,
 \qquad
 \mathcal{B}=\mathcal{A}^*.
 $$

It remains to determine the coordinates on which $\mathcal{A}$ depends.
Fix $i\notin J$, and let $\mathbf p^{(i)}$ be obtained from
$\mathbf p$ by replacing $p_i$ with $1-q_i$.  Then
$
 \mathbf p\leq\mathbf p^{(i)}
 \leq\mathbf 1-\mathbf q.
$
Monotonicity and $\alpha=\beta$ imply
\begin{equation*}
 \mu_{\mathbf p^{(i)}}(\mathcal{A})
 =
 \mu_{\mathbf p}(\mathcal{A}).
\end{equation*}
Conditioning on coordinate $i$ gives
\begin{align*}
 0
 =
 \mu_{\mathbf p^{(i)}}(\mathcal{A})
 -
 \mu_{\mathbf p}(\mathcal{A})=
 (1-q_i-p_i)
 \left(
 \mu_{\mathbf p_{-i}}\bigl(\mathcal{A}(i)\bigr)
 -
 \mu_{\mathbf p_{-i}}\bigl(\mathcal{A}(\bar i)\bigr)
 \right).
\end{align*}
Since $i\notin J$, at least one of $p_i,q_i$ is strictly less than
$1/2$, and hence
$
 1-q_i-p_i>0.
$
It follows from 
$\mathcal{A}(\bar i)\subseteq\mathcal{A}(i)$ that
\begin{equation*}
 \mathcal{A}(\bar i)=\mathcal{A}(i).
\end{equation*}
Thus, membership in $\mathcal{A}$ is independent of every coordinate
outside $J$.  Consequently, there is an increasing family
$\mathcal{H}\subseteq2^J$ such that
\begin{equation*}
 \mathcal{A}
 =
 \{A\subseteq[n]:A\cap J\in\mathcal{H}\}.
\end{equation*}
All coordinates in $J$ have probability $1/2$. Then
$
 \frac12
 =
 \mu_{\mathbf p}(\mathcal{A})
 =
 \frac{|\mathcal{H}|}{2^{|J|}}.
$
Therefore,
\begin{equation*}
 |\mathcal{H}|=2^{|J|-1}.
\end{equation*}
Let
$
 \mathcal{H}^*
 =
 \{T\subseteq J:T\cap H\neq\emptyset
   \text{ for every }H\in\mathcal{H}\}.
$
The equality $\mathcal{B}=\mathcal{A}^*$ now gives
\begin{equation*}
 \mathcal{B}
 =
 \{B\subseteq[n]:B\cap J\in\mathcal{H}^*\}.
\end{equation*}
This proves the asserted equality structure.

Conversely, let $\mathcal{H}\subseteq2^J$ be increasing with
$|\mathcal{H}|=2^{|J|-1}$, and define
\begin{equation*}
 \mathcal{A}
 =
 \{A\subseteq[n]:A\cap J\in\mathcal{H}\},
 \qquad
 \mathcal{B}
 =
 \{B\subseteq[n]:B\cap J\in\mathcal{H}^*\}.
\end{equation*}
These families are cross-intersecting by the definition of
$\mathcal{H}^*$.  Applying Lemma~\ref{lem:transversal} on the ground
set $J$ gives
\begin{equation*}
 \mathcal{H}^*
 =
 \{T\subseteq J:J\setminus T\notin\mathcal{H}\}.
\end{equation*}
Since $|\mathcal{H}|=2^{|J|-1}$, we have
\begin{equation*}
 |\mathcal{H}^*|
 =
 2^{|J|-1}.
\end{equation*}
Let $X$ have distribution $\mu_{\mathbf p}$.  Since $p_i=1/2$ for
every $i\in J$, the random set $X\cap J$ is uniformly distributed over
$2^J$.  
Hence,
\[
 \mu_{\mathbf p}(\mathcal{A})
 =
 \mathbb{P}(X\cap J\in\mathcal{H})
 =
 \frac{|\mathcal{H}|}{2^{|J|}}
 =
 \frac12.
\]
Similarly, we have
$
 \mu_{\mathbf q}(\mathcal{B})
 =
 \frac{|\mathcal{H}^*|}{2^{|J|}}
 =
 \frac12.
$
Their product is $1/4=d_1$, proving the converse and completing the
proof.
\end{proof}

\begin{remark}
Additional non-star equality cases first occur when $|J|\geq3$. Indeed, if $|J|=1$, the only increasing family of size
$2^{|J|-1}=1$ is the one-coordinate star.  If $|J|=2$, write
$J=\{i,j\}$.  An increasing family
$\mathcal{H}\subseteq2^J$ of size two cannot contain
$\emptyset$ and must contain $J$.  Hence,
\[
 \mathcal{H}
 =
 \bigl\{\{i\},J\bigr\}
 \qquad\text{or}\qquad
 \mathcal{H}
 =
 \bigl\{\{j\},J\bigr\}.
\]
In either case $\mathcal{H}$ is a star and
$\mathcal{H}^*=\mathcal{H}$.  Thus, no non-star equality case occurs
when $|J|\leq2$. For
 $|J|\geq3$, if $\{1,2,3\}\subseteq J$, let
\begin{equation*}
 \mathcal{H}
 =
 \{H\subseteq J:|H\cap\{1,2,3\}|\geq2\}.
\end{equation*}
Then $\mathcal{H}=\mathcal{H}^*$ and
$|\mathcal{H}|=2^{|J|-1}$.  Consequently, the two families
\begin{equation*}
 \mathcal{A}
 =
 \{A\subseteq[n]:A\cap J\in\mathcal{H}\},
 \qquad
 \mathcal{B}
 =
 \{B\subseteq[n]:B\cap J\in\mathcal{H}^*\}
\end{equation*}
coincide, have measure $1/2$, and attain $d_1=1/4$, although they are
not stars.
\end{remark}

\section{Proof of Theorem~\ref{thm:main-stability}}\label{se3}

Throughout this section, put
$$
 p=p_1,\qquad q=q_1,\qquad
 \mu_1=\mu_{\mathbf p},\qquad
 \mu_2=\mu_{\mathbf q}.
$$
Thus,
$$
 0<q\leq p<\frac12,\qquad
 p_\ell\leq p,\qquad q_\ell\leq q
 \quad\text{for every }\ell\in[n].
$$
With this notation, the hypothesis of the theorem becomes
$
 \mu_1(\mathcal A)\mu_2(\mathcal B)
 \geq(1-\varepsilon)pq.
$

\subsection{A semidefinite estimate}
\label{sec:semidefinite-estimate}

The first step is to show that the preceding near-extremal inequality
forces the indicator functions of
$\mathcal A$ and $\mathcal B$ to have little Fourier weight above
degree one.  This will follow from an explicit semidefinite
estimate.
To treat the two measures simultaneously, set
$$
 \mathbf r_1=\mathbf p,\qquad
 \mathbf r_2=\mathbf q.
$$
For $i\in\{1,2\}$ and $\ell\in[n]$, let
$r_i^{(\ell)}$ denote the $\ell$-th coordinate of $\mathbf r_i$,
and put
$$
 \overline r_i^{(\ell)}=1-r_i^{(\ell)},
 \qquad
 c_i^{(\ell)}
 =
 \sqrt{\frac{r_i^{(\ell)}}{\overline r_i^{(\ell)}}}.
$$
For one coordinate $\ell$, define
$$
 A_{ij}^{(\ell)}
 =
 \begin{pmatrix}
  1-\dfrac{r_j^{(\ell)}}{\overline r_i^{(\ell)}}&
  \dfrac{r_j^{(\ell)}}{\overline r_i^{(\ell)}}\\[5pt]
  1&0
 \end{pmatrix},
 \qquad
 D_{ij}^{(\ell)}
 =
 \begin{pmatrix}
  1&0\\
  0&-c_i^{(\ell)}c_j^{(\ell)}
 \end{pmatrix},
$$
and
$$
 V_i^{(\ell)}
 =
 \begin{pmatrix}
  1&c_i^{(\ell)}\\[2pt]
  1&-1/c_i^{(\ell)}
 \end{pmatrix},
 \qquad
 \Delta_i^{(\ell)}
 =
 \begin{pmatrix}
  \overline r_i^{(\ell)}&0\\
  0&r_i^{(\ell)}
 \end{pmatrix}.
$$
The rows and columns in these matrices are indexed by
$\emptyset$ and $\{\ell\}$.   Let
$A_{ij},D_{ij},V_i,\Delta_i$ be the tensor products, over all
$\ell\in[n]$, of the corresponding one-coordinate matrices.
We identify each set $S\subseteq[n]$ with its indicator vector
$(\mathbf 1_{\{1\in S\}},\ldots,\mathbf 1_{\{n\in S\}})\in\{0,1\}^n$
and use the lexicographic order on these vectors. Thus, if
$M=M^{(1)}\otimes\cdots\otimes M^{(n)}$, then its rows and columns
are indexed by subsets $S,T\subseteq[n]$, and
$$
 M_{S,T}
 =
 \prod_{\ell=1}^n
 M^{(\ell)}_{\mathbf 1_{\{\ell\in S\}},
                 \mathbf 1_{\{\ell\in T\}}}.
$$
In particular, $\Delta_i$ is diagonal and
$$
 (\Delta_i)_{S,S}
 =
 \prod_{\ell\in S}r_i^{(\ell)}
 \prod_{\ell\notin S}\bigl(1-r_i^{(\ell)}\bigr)
 =
 \mu_i(\{S\}).
$$
Moreover, $D_{ij}$ is diagonal, with
$$
 (D_{ij})_{T,T}
 =
 \prod_{\ell\in T}
 \bigl(-c_i^{(\ell)}c_j^{(\ell)}\bigr).
$$
The entries of $A_{ij}$ are given by
$$
 (A_{ij})_{S,T}
 =
 \begin{cases}
 \displaystyle
 \prod_{\ell\in T}
 \frac{r_j^{(\ell)}}{1-r_i^{(\ell)}}
 \prod_{\ell\notin S\cup T}
 \frac{1-r_i^{(\ell)}-r_j^{(\ell)}}
      {1-r_i^{(\ell)}},
 & S\cap T=\emptyset,\\[12pt]
 0,
 & S\cap T\neq\emptyset.
 \end{cases}
$$
Consequently,
$$
 (\Delta_iA_{ij})_{S,T}
 =
 \begin{cases}
 \displaystyle
 \prod_{\ell\in S}r_i^{(\ell)}
 \prod_{\ell\in T}r_j^{(\ell)}
 \prod_{\ell\notin S\cup T}
 \bigl(1-r_i^{(\ell)}-r_j^{(\ell)}\bigr),
 & S\cap T=\emptyset,\\[10pt]
 0,
 & S\cap T\neq\emptyset.
 \end{cases}
$$

\begin{lemma}\label{lem:semidefinite-identities}
For $i,j\in\{1,2\}$,
$$
 A_{ij}V_j=V_iD_{ij},
 \qquad
 V_i^{\mathsf T}\Delta_iV_i=I,
 \qquad
 (\Delta_iA_{ij})^{\mathsf T}=\Delta_jA_{ji}.
$$
If $ J$ denotes the $2^n\times2^n$ all-one matrix, then
$$
 V_i^{\mathsf T}\Delta_i J\Delta_jV_j
 =
 E_{\emptyset,\emptyset},
$$
where $E_{\emptyset,\emptyset}$ has a single nonzero entry,
equal to one, in position $(\emptyset,\emptyset)$.
\end{lemma}

\begin{proof}
For each $\ell\in[n]$, we first verify the three identities for the
corresponding one-coordinate matrices. Since
$
 \bigl(c_k^{(\ell)}\bigr)^2
 =
 \frac{r_k^{(\ell)}}{\overline r_k^{(\ell)}},
$
we have
$
 r_j^{(\ell)}
 \left(
  c_j^{(\ell)}+\frac{1}{c_j^{(\ell)}}
 \right)
 =
 c_j^{(\ell)}.
$
Therefore, direct multiplication gives
$$
\begin{aligned}
 A_{ij}^{(\ell)}V_j^{(\ell)}
 &=
 \begin{pmatrix}
  1-\dfrac{r_j^{(\ell)}}{\overline r_i^{(\ell)}}&
  \dfrac{r_j^{(\ell)}}{\overline r_i^{(\ell)}}\\[5pt]
  1&0
 \end{pmatrix}
 \begin{pmatrix}
  1&c_j^{(\ell)}\\[2pt]
  1&-1/c_j^{(\ell)}
 \end{pmatrix}\\
 &=
 \begin{pmatrix}
  1&-\bigl(c_i^{(\ell)}\bigr)^2c_j^{(\ell)}\\
  1&c_j^{(\ell)}
 \end{pmatrix}
 =
 V_i^{(\ell)}D_{ij}^{(\ell)}.
\end{aligned}
$$
Similarly,
$$
\begin{aligned}
 \bigl(V_i^{(\ell)}\bigr)^{\mathsf T}
 \Delta_i^{(\ell)}V_i^{(\ell)}
 &=
 \begin{pmatrix}
  \overline r_i^{(\ell)}+r_i^{(\ell)}&
  \overline r_i^{(\ell)}c_i^{(\ell)}
  -r_i^{(\ell)}/c_i^{(\ell)}\\[2pt]
  \overline r_i^{(\ell)}c_i^{(\ell)}
  -r_i^{(\ell)}/c_i^{(\ell)}&
  \overline r_i^{(\ell)}\bigl(c_i^{(\ell)}\bigr)^2
  +r_i^{(\ell)}/\bigl(c_i^{(\ell)}\bigr)^2
 \end{pmatrix}.
\end{aligned}
$$
The two off-diagonal entries are zero, since
$
 \overline r_i^{(\ell)}c_i^{(\ell)}
 -\frac{r_i^{(\ell)}}{c_i^{(\ell)}}
 =
 \frac{
  \overline r_i^{(\ell)}\bigl(c_i^{(\ell)}\bigr)^2
  -r_i^{(\ell)}
 }{c_i^{(\ell)}}
 =0.
$
The two diagonal entries are equal to one, because
$
 \overline r_i^{(\ell)}+r_i^{(\ell)}=1
$
and
$
 \overline r_i^{(\ell)}\bigl(c_i^{(\ell)}\bigr)^2
 +\frac{r_i^{(\ell)}}{\bigl(c_i^{(\ell)}\bigr)^2}
 =
 r_i^{(\ell)}+\overline r_i^{(\ell)}
 =1.
$
Hence,
$$
 \bigl(V_i^{(\ell)}\bigr)^{\mathsf T}
 \Delta_i^{(\ell)}V_i^{(\ell)}
 =I.
$$
Furthermore,
$$
\begin{aligned}
 \Delta_i^{(\ell)}A_{ij}^{(\ell)}
 =
 \begin{pmatrix}
  \overline r_i^{(\ell)}&0\\
  0&r_i^{(\ell)}
 \end{pmatrix}
 \begin{pmatrix}
  1-\dfrac{r_j^{(\ell)}}{\overline r_i^{(\ell)}}&
  \dfrac{r_j^{(\ell)}}{\overline r_i^{(\ell)}}\\[5pt]
  1&0
 \end{pmatrix}=
 \begin{pmatrix}
  1-r_i^{(\ell)}-r_j^{(\ell)}&r_j^{(\ell)}\\
  r_i^{(\ell)}&0
 \end{pmatrix}.
\end{aligned}
$$
Consequently,
$$
 \bigl(\Delta_i^{(\ell)}A_{ij}^{(\ell)}\bigr)^{\mathsf T}
 =
 \Delta_j^{(\ell)}A_{ji}^{(\ell)}.
$$
Taking tensor products over all $\ell\in[n]$ now gives
$$
 A_{ij}V_j=V_iD_{ij},
 \qquad
 V_i^{\mathsf T}\Delta_iV_i=I,
 \qquad
 (\Delta_iA_{ij})^{\mathsf T}=\Delta_jA_{ji}.
$$
Finally, let
$
 \mathbf 1^{(\ell)}=\begin{pmatrix}1\\1\end{pmatrix}$
and
$
 \mathbf 1=\bigotimes_{\ell=1}^n\mathbf 1^{(\ell)}.
$
For every $\ell\in[n]$, we have
$
 \bigl(V_i^{(\ell)}\bigr)^{\mathsf T}
 \Delta_i^{(\ell)}\mathbf 1^{(\ell)}
 =
 \begin{pmatrix}1\\0\end{pmatrix}.
$
It follows that
$$
 V_i^{\mathsf T}\Delta_i\mathbf 1
 =
 \bigotimes_{\ell=1}^n
 \begin{pmatrix}1\\0\end{pmatrix}
 =
 e_{\emptyset}.
$$
where $e_{\emptyset}$ is the standard basis vector indexed by $\emptyset$.
Since $ J=\mathbf 1\mathbf 1^{\mathsf T}$, we conclude that
$$
\begin{aligned}
 V_i^{\mathsf T}\Delta_i\mathsf J\Delta_jV_j
 =
 \bigl(V_i^{\mathsf T}\Delta_i\mathbf 1\bigr)
 \bigl(\mathbf 1^{\mathsf T}\Delta_jV_j\bigr)=
 e_{\emptyset}e_{\emptyset}^{\mathsf T}
 =
 E_{\emptyset,\emptyset},
\end{aligned}
$$
as required.
\end{proof}

Set
$
 s=\frac{\sqrt{pq}}2.
$
Choose a positive number $\tau$, whose size will be fixed below, and
define
$$
 \xi_2=\tau,\qquad
 \xi_1
 =
 \frac{q}{p}\,\tau
 +\frac{(p-q)q}{2\sqrt{pq}},
 \qquad
 \eta
 =
 \frac q{\sqrt{pq}}\,\tau+\frac{1-q}{2}.
$$
Consider the  block matrix
$$
  S
 =
 \begin{pmatrix}
  s\Delta_1-\xi_1\Delta_1A_{11}&
  \eta\Delta_1A_{12}
  -\frac12\Delta_1\mathsf J\Delta_2\\[2pt]
  \eta\Delta_2A_{21}
  -\frac12\Delta_2\mathsf J\Delta_1&
  s\Delta_2-\xi_2\Delta_2A_{22}
 \end{pmatrix}.
$$
By Lemma~\ref{lem:semidefinite-identities}, $\mathsf S$ is symmetric.
Let
$$
 Z
 =
 \begin{pmatrix}
  \xi_1\Delta_1A_{11}&0\\
  0&\xi_2\Delta_2A_{22}
 \end{pmatrix}.
$$
Since $\xi_1,\xi_2>0$ and $r_i^{(\ell)}<1/2$, every entry of
$A_{ii}^{(\ell)}$, and hence every entry of $Z$, is
nonnegative.

For $T\subseteq[n]$, define
$$
 \lambda_{ij}(T)
 =
 \prod_{\ell\in T}
 \bigl(-c_i^{(\ell)}c_j^{(\ell)}\bigr),
$$
where the empty product is one.  By Lemma~\ref{lem:semidefinite-identities},
\begin{align}
\begin{pmatrix}V_1&0\\0&V_2\end{pmatrix}^{\mathsf T}
 S
\begin{pmatrix}V_1&0\\0&V_2\end{pmatrix}=
\begin{pmatrix}
 sI-\xi_1D_{11}&
 \eta D_{12}-\frac12E_{\emptyset,\emptyset}\\
 \eta D_{21}-\frac12E_{\emptyset,\emptyset}&
 sI-\xi_2D_{22}
\end{pmatrix}.\label{ff1}
\end{align}
Each matrix appearing on the right-hand side is diagonal in the basis
$\{e_T:T\subseteq[n]\}$, where  $e_{T}$ is the standard basis vector indexed by $T$. Hence, each two-dimensional subspace
$$
 \operatorname{span}\{(e_T,0),(0,e_T)\}
$$
is invariant. After ordering the basis as
$(e_T,0),(0,e_T)$ for each $T\subseteq[n]$, the transformed matrix is
the direct sum of one $2\times2$ block for each $T\subseteq[n]$.  Since
$V_i^{\mathsf T}\Delta_iV_i=I$, each $V_i$ is invertible.
For a real symmetric matrix $ M$, we write
$ M\succeq0$ if $ M$ is positive semidefinite.
Since each $V_i$ is invertible, the preceding block decomposition
shows that $ S\succeq0$ if and only if every resulting
$2\times2$ block is positive semidefinite.  The block corresponding to
$T=\emptyset$ is
$$
 M_{\emptyset}
 =
 \begin{pmatrix}
  s-\xi_1&\eta-1/2\\
  \eta-1/2&s-\xi_2
 \end{pmatrix},
$$
whereas, for $T\neq\emptyset$, it is
$$
 M_T
 =
 \begin{pmatrix}
  s-\xi_1\lambda_{11}(T)&
  \eta\lambda_{12}(T)\\
  \eta\lambda_{12}(T)&
  s-\xi_2\lambda_{22}(T)
 \end{pmatrix}.
$$

\begin{lemma}\label{lem:semidefinite-gap}
For $\tau>0$ sufficiently small, depending only on $p$ and $q$,
the matrix $ S$ is positive semidefinite.  Moreover, there is
a constant $\gamma=\gamma(p,q)>0$ such that
$$
 M_T\succeq\gamma I_2
 \qquad\text{whenever }|T|\geq2.
$$
\end{lemma}

\begin{proof}
Since
$
 \frac{(p-q)q}{2\sqrt{pq}}
 =
 \left(1-\frac qp\right)s,
 \frac q{\sqrt{pq}}=\sqrt{\frac qp}
$
and
$
 \frac q2=\sqrt{\frac qp}\,s,
$
the definitions of $\xi_1,\xi_2$, and $\eta$ give
$$
 s-\xi_1=\frac qp(s-\tau),
 \qquad
 \eta-\frac12=-\sqrt{\frac qp}(s-\tau),
 \qquad
 s-\xi_2=s-\tau.
$$
Consequently, if $0<\tau\leq s$, then
$$
 M_{\emptyset}
 =
 (s-\tau)
 \begin{pmatrix}
  q/p&-\sqrt{q/p}\\
  -\sqrt{q/p}&1
 \end{pmatrix}
 \succeq0.
$$

For $T\neq\emptyset$, set
$$
 \rho_i(T)
 =
 \prod_{\ell\in T}
 \frac{r_i^{(\ell)}}{1-r_i^{(\ell)}}
 \qquad i=1,2,
$$
and write
$$
 R_1=\frac p{1-p},
 \qquad
 R_2=\frac q{1-q}.
$$
Since the function $t\mapsto t/(1-t)$ is increasing on $(0,1)$,
the inequalities
$
 p_\ell\leq p<\frac12,
 q_\ell\leq q<\frac12
$
give
$$
 0<
 \frac{r_i^{(\ell)}}{1-r_i^{(\ell)}}
 \leq R_i<1
 \qquad(i=1,2).
$$
Since $T\neq\emptyset$, multiplying these inequalities over
$\ell\in T$ yields
$$
 0<\rho_i(T)\leq R_i^{|T|}<1.
$$
Moreover,
$$
 \lambda_{ii}(T)=(-1)^{|T|}\rho_i(T),
 \qquad
 \lambda_{12}(T)^2=\rho_1(T)\rho_2(T).
$$

First set $\tau=0$.  If $|T|$ is odd, a direct expansion of the
determinant of the displayed matrix $M_T$ gives
$$
 4\det M_T
 =
 pq+(p-q)q\,\rho_1(T)
 -(1-q)^2\rho_1(T)\rho_2(T).
$$
If $|T|$ is even, the same calculation gives
$$
 4\det M_T
 =
 pq-(p-q)q\,\rho_1(T)
 -(1-q)^2\rho_1(T)\rho_2(T).
$$
Suppose that $|T|\geq2$ is even.  Using
$0<\rho_i(T)\leq R_i^{|T|}<1$, we obtain
$$
\begin{aligned}
 (p-q)q\,\rho_1(T)
 +(1-q)^2\rho_1(T)\rho_2(T)&\leq
 R_1^2\bigl((p-q)q+(1-q)^2R_2^2\bigr)\\
 &=
 R_1^2\bigl((p-q)q+q^2\bigr)
 =pqR_1^2.
\end{aligned}
$$
Thus, if $|T|\geq2$ is even,  then
$$
 \det M_T
 \geq
 \delta_{\mathrm e}
 :=
 \frac{pq(1-R_1^2)}4>0.
$$
Now suppose that $|T|\geq3$ is odd.  Discarding the nonnegative
middle term in the odd-case determinant formula and again using
$\rho_i(T)\leq R_i^{|T|}$, we obtain
$$
 4\det M_T
 \geq
 pq-(1-q)^2R_1^3R_2^3.
$$
The last quantity is positive because
$$
\begin{aligned}
 \frac{(1-q)^2R_1^3R_2^3}{pq}
=
 \frac{p^2q^2}{(1-p)^3(1-q)}\leq
 \frac{p^4}{(1-p)^4}
 <1,
\end{aligned}
$$
where the first inequality uses $q\leq p$.  Hence, if $|T|\geq3$ is odd,  then
$$
 \det M_T
 \geq
 \delta_{\mathrm o}
 :=
 \frac{pq-(1-q)^2R_1^3R_2^3}{4}>0.
$$

At $\tau=0$, the lower-right entry of each nonempty block $M_T$ is
$s>0$.  The determinant bounds above therefore show that every
block with $|T|\geq2$ is positive definite.  At $\tau=0$, we have
$
 \xi_2=0
$
and
$
 \xi_1=\frac{(p-q)q}{2\sqrt{pq}}.
$
Hence, for $|T|\geq2$,
$$
\begin{aligned}
 \operatorname{tr}M_T
 =
 2s-\xi_1\lambda_{11}(T)\leq
 2s+\xi_1|\lambda_{11}(T)|\leq
 2s+\frac{(p-q)q}{2\sqrt{pq}}R_1^2
 =:L_0,
\end{aligned}
$$
where we used
$|\lambda_{11}(T)|=\rho_1(T)\leq R_1^{|T|}\leq R_1^2$.
For a positive definite $2\times2$ matrix, its smaller eigenvalue is
at least its determinant divided by its trace.  Thus, at $\tau=0$,
all blocks with $|T|\geq2$ have smaller eigenvalue at least
$$
 \gamma_0
 :=
 \frac{\min\{\delta_{\mathrm e},\delta_{\mathrm o}\}}{L_0}>0.
$$

Assume that $|T|\geq2$. From the definitions of
$\xi_1,\xi_2$, and $\eta$, we have
$$
 M_T(\tau)-M_T(0)
 =
 \tau
 \begin{pmatrix}
  -\dfrac qp\lambda_{11}(T)&
  \sqrt{\dfrac qp}\lambda_{12}(T)\\[5pt]
  \sqrt{\dfrac qp}\lambda_{12}(T)&
  -\lambda_{22}(T)
 \end{pmatrix}.
$$
Since $q\leq p$ and
$
 |\lambda_{11}(T)|,\ 
 |\lambda_{12}(T)|,\ 
 |\lambda_{22}(T)|
 \leq1,
$
every entry of $M_T(\tau)-M_T(0)$ has absolute value at most $\tau$.
Consequently, for every $(x,y)\in\mathbb R^2$,
$$
\begin{aligned}
\left|
 \begin{pmatrix}x&y\end{pmatrix}
 \bigl(M_T(\tau)-M_T(0)\bigr)
 \begin{pmatrix}x\\y\end{pmatrix}
\right|
\leq
\tau\bigl(x^2+2|xy|+y^2\bigr)\leq
2\tau(x^2+y^2).
\end{aligned}
$$
Choose
$
 0<\tau\leq\min\left\{s,\frac{\gamma_0}{4}\right\}.
$
Since $M_T(0)\succeq\gamma_0I_2$ for $|T|\geq2$, it follows that
$$
\begin{aligned}
 \begin{pmatrix}x&y\end{pmatrix}
 M_T(\tau)
 \begin{pmatrix}x\\y\end{pmatrix}
 \geq
 \gamma_0(x^2+y^2)-2\tau(x^2+y^2)\geq
 \frac{\gamma_0}{2}(x^2+y^2).
\end{aligned}
$$
Hence,
$$
 M_T(\tau)\succeq\frac{\gamma_0}{2}I_2.
$$

It remains to treat the singleton blocks.  Assume that $|T|=1$, and put
$
 u=\rho_1(T)\in[0,R_1],
$
$
 v=\rho_2(T)\in[0,R_2].
$
The determinant of $M_T$ is
$$
 D(u,v)
 =
 (s+\xi_1u)(s+\tau v)-\eta^2uv.
$$
This is affine in each of $u$ and $v$.  Define
$$
 H=\frac{\sqrt{pq}(1-q)}2+q\tau>0.
$$
Direct substitution, using
$
 s+\xi_1R_1=\frac{H}{1-p},
 s+\tau R_2=\frac{H}{1-q},
 \eta=\frac{H}{\sqrt{pq}},
$
gives
$$
\begin{aligned}
 D(0,0)=s^2,\qquad
 D(R_1,0)=\frac{sH}{1-p},\qquad
 D(0,R_2)=\frac{sH}{1-q},\qquad
 D(R_1,R_2)=0.
\end{aligned}
$$
The function $D(u,v)$ is affine in each variable separately.
For fixed $v$, its minimum over $0\leq u\leq R_1$ is attained at
$u=0$ or $u=R_1$. For either choice of $u$, its minimum over
$0\leq v\leq R_2$ is attained at $v=0$ or $v=R_2$. Hence the minimum
of $D$ on $[0,R_1]\times[0,R_2]$ is attained at one of the four
corners. Since all four corner values are nonnegative,
$$
 D(u,v)\geq0
 \qquad
 \text{for all }(u,v)\in[0,R_1]\times[0,R_2].
$$
Since $p\geq q$ and $\tau>0$, we have
$
 \xi_1
 =
 \frac qp\tau+\frac{(p-q)q}{2\sqrt{pq}}
 >0.
$
Hence, the two diagonal entries of $M_T$ satisfy
$$
 s+\xi_1u\geq s>0,
 \qquad
 s+\tau v\geq s>0.
$$
This, together with $\det M_T=D(u,v)\geq0$,  implies that 
$M_T$ is positive semidefinite.
 The block
$M_{\emptyset}$ is positive semidefinite because $\tau\leq s$, and
the preceding estimate gives
$$
 M_T(\tau)\succeq\frac{\gamma_0}{2}I_2,\qquad
 |T|\geq2.
$$
Therefore, every block is positive semidefinite. Then $\mathsf S$ is
positive semidefinite. The asserted uniform bound holds with
$\gamma=\gamma_0/2$, completing the proof.
\end{proof}

Recall that $\Delta_i$ is diagonal and
$
 (\Delta_i)_{S,S}
 =
 \mu_i(\{S\}).
$
Identify a function $f:2^{[n]}\to\mathbb R$ with the column vector
$(f(S))_{S\subseteq[n]}$, and equip this space with the
\textit{$\mu_i$-weighted inner product}
$$
 \langle f,g\rangle_{\mu_i}
 =
 f^{\mathsf T}\Delta_i g
 =
 \sum_{S\subseteq[n]}\mu_i(S)f(S)g(S).
$$
Recall that
$
 V_i=\bigotimes_{\ell=1}^nV_i^{(\ell)}.
$
For each $T\subseteq[n]$, let $v_i(T)$ denote the column of $V_i$
indexed by $T$. By Lemma~\ref{lem:semidefinite-identities}, we have
$
 V_i^{\mathsf T}\Delta_iV_i=I,
$
which shows that
$$
 \langle v_i(T),v_i(U)\rangle_{\mu_i}
 =
 \begin{cases}
  1,&T=U,\\
  0,&T\neq U.
 \end{cases}
$$
Thus, the columns of $V_i$ form an orthonormal basis with respect to
$\langle\cdot,\cdot\rangle_{\mu_i}$. Consequently, every function
$f:2^{[n]}\to\mathbb R$ has the unique \textit{Fourier expansion}
$$
 f
 =
 \sum_{T\subseteq[n]}\widehat f_i(T)v_i(T),
 \qquad
 \widehat f_i(T)
 =
 \langle f,v_i(T)\rangle_{\mu_i}
 =
 v_i(T)^{\mathsf T}\Delta_i f.
$$
The \textit{degree} of the basis function $v_i(T)$ is defined to be $|T|$.
We define the \textit{Fourier weight of $f$ above degree one} by
$$
 W_{>1}^{\mu_i}(f)
 =
 \sum_{T\subseteq[n], |T|\geq2}
 \widehat f_i(T)^2.
$$

\begin{proposition}\label{prop:higher-degree-estimate}
There is a constant $C_{\mathrm{sd}}=C_{\mathrm{sd}}(p,q)>0$ such
that, whenever $0\leq\varepsilon<1$ and
$\mathcal A,\mathcal B\subseteq 2^{[n]}$ are cross-intersecting and satisfy
$$
 \mu_1(\mathcal A)\mu_2(\mathcal B)
 \geq(1-\varepsilon)pq,
$$
one has
\begin{equation}
 W_{>1}^{\mu_1}(\mathbf 1_{\mathcal A})
 +
 W_{>1}^{\mu_2}(\mathbf 1_{\mathcal B})
 \leq C_{\mathrm{sd}}\varepsilon.
 \label{eq:stability-high-degree}
\end{equation}
\end{proposition}

\begin{proof}
Let $x_1,x_2$ be the characteristic column vectors of
$\mathcal A,\mathcal B$, respectively. Set
$$
 a_1=\mu_1(\mathcal A)\leq 1,
 \qquad
 a_2=\mu_2(\mathcal B)\leq 1.
$$
Since
$\varepsilon<1$, the assumed lower bound gives
$a_1a_2>0$.  Hence the vectors
$$
 u_1=\frac{x_1}{\sqrt{a_1}},
 \qquad
 u_2=\frac{x_2}{\sqrt{a_2}},
 \qquad
 u=\binom{u_1}{u_2}
$$
are well defined.  Since  $x_i$ is a $0$-$1$ vector, we have
$$
 x_i^{\mathsf T}\Delta_i x_i
 =
 \sum_{S\subseteq[n]}\mu_i(S)x_i(S)^2
 =
 \sum_{S\subseteq[n]}\mu_i(S)x_i(S)
 =
 a_i.
$$
Consequently,
$$
 u_i^{\mathsf T}\Delta_i u_i
 =
 \frac{x_i^{\mathsf T}\Delta_i x_i}{a_i}
 =1,
 \qquad i=1,2.
$$
Recall that the matrix $\Delta_1A_{12}$ is supported on pairs
$(A,B)$ with $A\cap B=\emptyset$.  Then cross-intersection gives
$$
 x_1^{\mathsf T}\Delta_1A_{12}x_2
 =
 \sum_{S,T\subseteq[n]}
 x_1(S)(\Delta_1A_{12})_{S,T}x_2(T)=0.
$$
Thus, $u_1^{\mathsf T}\Delta_1A_{12}u_2=0$.
Moreover, writing $ J=\mathbf 1\mathbf 1^{\mathsf T}$, we have
$$
\begin{aligned}
 x_1^{\mathsf T}\Delta_1 J\Delta_2x_2
 =
 \bigl(x_1^{\mathsf T}\Delta_1\mathbf 1\bigr)
 \bigl(\mathbf 1^{\mathsf T}\Delta_2x_2\bigr)=
 \mu_1(\mathcal A)\mu_2(\mathcal B)
 =
 a_1a_2.
\end{aligned}
$$
Consequently,
$
 u_1^{\mathsf T}\Delta_1 J\Delta_2u_2
 =\sqrt{a_1a_2}.
$

Substituting the definitions of $S$ and $Z$, and using
the symmetry of $S$, we obtain
$$
\begin{aligned}
 u^{\mathsf T} S u+u^{\mathsf T}Z u
 =&
 u_1^{\mathsf T}
 \bigl(s\Delta_1-\xi_1\Delta_1A_{11}\bigr)u_1+
 u_2^{\mathsf T}
 \bigl(s\Delta_2-\xi_2\Delta_2A_{22}\bigr)u_2\\
 &+
 2u_1^{\mathsf T}
 \left(
  \eta\Delta_1A_{12}
  -\frac12\Delta_1 J\Delta_2
 \right)u_2+
 \xi_1u_1^{\mathsf T}\Delta_1A_{11}u_1
 +
 \xi_2u_2^{\mathsf T}\Delta_2A_{22}u_2\\
 =&
 s\left(
  u_1^{\mathsf T}\Delta_1u_1
  +u_2^{\mathsf T}\Delta_2u_2
 \right)
 +
 2\eta\,u_1^{\mathsf T}\Delta_1A_{12}u_2-
 u_1^{\mathsf T}\Delta_1 J\Delta_2u_2\\
 =&
 2s-\sqrt{a_1a_2}=
 \sqrt{pq}-\sqrt{a_1a_2}.
\end{aligned}
$$
Although $S$ and $ Z$ depend on $\tau$, the identity
above holds for every $\tau>0$. We now fix $\tau>0$, depending only on $p$ and
$q$, sufficiently small that Lemma~\ref{lem:semidefinite-gap}
applies. Then $ S$ is positive semidefinite, and hence
$u^{\mathsf T} S u\geq0$.  
Since $u$ and
$ Z$ are entrywise nonnegative,
$u^{\mathsf T} Z u\geq0$ as well.  It follows from
the assumed lower bound and the preceding identity that
$$
\begin{aligned}
 u^{\mathsf T} S u
 \leq
 \sqrt{pq}-\sqrt{a_1a_2}\leq
 \sqrt{pq}\bigl(1-\sqrt{1-\varepsilon}\bigr)
 \leq
 \sqrt{pq}\,\varepsilon.
\end{aligned}
$$

Write
$
 u_i=\sum_{T\subseteq[n]}\theta_i(T)v_i(T).
$
Let
$
 \boldsymbol\theta_i
 =
 \bigl(\theta_i(T)\bigr)_{T\subseteq[n]}.
$
Since the vectors $v_i(T)$ are the columns of $V_i$, this is
equivalent to
$
 u_i=V_i\boldsymbol\theta_i.
$
Thus, with
$$
 W=
 \begin{pmatrix}V_1&0\\0&V_2\end{pmatrix},
 \qquad
 \boldsymbol\theta=
 \binom{\boldsymbol\theta_1}{\boldsymbol\theta_2},
$$
we have $u=W\boldsymbol\theta$, and hence
$$
 u^{\mathsf T} S u
 =
 \boldsymbol\theta^{\mathsf T}
 W^{\mathsf T}\mathsf S W
 \boldsymbol\theta.
$$
As established \eqref{ff1}, after reordering the basis so that $(e_T,0)$ and $(0,e_T)$ are
adjacent for every $T\subseteq[n]$, there exists a permutation matrix
$P$ such that
$$
 P^{\mathsf T}W^{\mathsf T} S WP
 =
 \bigoplus_{T\subseteq[n]}M_T.
$$
Set
$
 \widetilde{\boldsymbol\theta}
 =
 P^{\mathsf T}\boldsymbol\theta.
$
By the definition of $P$, we have
$
 \widetilde{\boldsymbol\theta}
 =
 (z_T)_{T\subseteq[n]},
$
where
$
 z_T
 =
 \binom{\theta_1(T)}{\theta_2(T)}.
$
Since $P$ is a permutation matrix, $PP^{\mathsf T}=I$ and hence
$\boldsymbol\theta=P\widetilde{\boldsymbol\theta}$. Recalling that
$u=W\boldsymbol\theta$, we obtain
$$
\begin{aligned}
 u^{\mathsf T} S u
 &=
 \boldsymbol\theta^{\mathsf T}
 W^{\mathsf T} S W
 \boldsymbol\theta=
 \widetilde{\boldsymbol\theta}^{\mathsf T}
 P^{\mathsf T}W^{\mathsf T} S WP
 \widetilde{\boldsymbol\theta}=
 \widetilde{\boldsymbol\theta}^{\mathsf T}
 \left(
  \bigoplus_{T\subseteq[n]}M_T
 \right)
 \widetilde{\boldsymbol\theta}=
 \sum_{T\subseteq[n]}z_T^{\mathsf T}M_Tz_T.
\end{aligned}
$$
Every $M_T$ is positive semidefinite. Moreover,
Lemma~\ref{lem:semidefinite-gap} gives
$
 M_T\succeq\gamma I_2
$
whenever $|T|\geq2$. Therefore,
$$
 z_T^{\mathsf T}M_Tz_T
 \geq
 \gamma z_T^{\mathsf T}z_T
 =
 \gamma\left(\theta_1(T)^2+\theta_2(T)^2\right)
$$
for every $T$ with $|T|\geq2$, whereas
$
z_T^{\mathsf T}M_Tz_T\geq0
$
for $|T|\leq1$. Summing these inequalities gives
$$
 u^{\mathsf T} S u
 \geq
 \gamma
 \sum_{\substack{T\subseteq[n]\\|T|\geq2}}
 \left(\theta_1(T)^2+\theta_2(T)^2\right).
$$
Combining this with
$
u^{\mathsf T} S u\leq\sqrt{pq}\,\varepsilon
$
yields
$$
 \sum_{\substack{T\subseteq[n]\\|T|\geq2}}
 \left(\theta_1(T)^2+\theta_2(T)^2\right)
 \leq
 \frac{\sqrt{pq}}{\gamma}\,\varepsilon.
$$

Recall that $\langle v_i(T),v_i(U)\rangle_{\mu_i} =1$ if and only if $U=T$.
Since $x_i=\sqrt{a_i}\,u_i$ and  $
 u_i=\sum_{T\subseteq[n]}\theta_i(T)v_i(T)$,
we have
$$
\begin{aligned}
 \widehat{x_i}_i(T)
 =
 v_i(T)^{\mathsf T}\Delta_i x_i=
 \sqrt{a_i}\,
 v_i(T)^{\mathsf T}\Delta_i u_i=
 \sqrt{a_i}\,\theta_i(T).
\end{aligned}
$$
Consequently,
$$
\begin{aligned}
 W_{>1}^{\mu_1}(\mathbf 1_{\mathcal A})
 +W_{>1}^{\mu_2}(\mathbf 1_{\mathcal B})
 &=
 \sum_{T\subseteq[n], |T|\geq2}
 \left(
  a_1\theta_1(T)^2+a_2\theta_2(T)^2
 \right)\\
 &\leq
 \sum_{T\subseteq[n], |T|\geq2}
 \left(
  \theta_1(T)^2+\theta_2(T)^2
 \right)\leq
 \frac{\sqrt{pq}}{\gamma}\,\varepsilon,
\end{aligned}
$$
where we used $0<a_1,a_2\leq1$. Thus, the proposition holds with
$
 C_{\mathrm{sd}}=\frac{\sqrt{pq}}{\gamma}.
$
\end{proof}

\subsection{One-coordinate approximation}

We first prove an elementary estimate for a sum of independent random
variables.  The argument is a self-contained version of the two-point
concentration method of Jendrej, Oleszkiewicz and Wojtaszczyk
\cite{JendrejOleszkiewiczWojtaszczyk}.

\begin{lemma}\label{lem:one-large-summand}
Let $Y_1,\ldots,Y_n$ be independent  real-valued random
variables  such
that
$
 \mathbb E[Y_i^2]<\infty$ for every $i\in[n]$.
Put $L=\sum_{i=1}^nY_i$.  Suppose that
$$
 \mathbb E[(|L|-1)^2]\leq \eta\leq 1
 \qquad\text{and}\qquad
 \operatorname{Var}(L)>0.
$$
Then there is an index $k\in[n]$ such that
$$
 \operatorname{Var}\Bigl(\sum_{i\neq k}Y_i\Bigr)
 \leq \frac{1600\eta}{\operatorname{Var}(L)}.
$$
\end{lemma}
\begin{proof}\setcounter{claim}{0}
We begin with the following two-summand estimate.  Let $X$ and $Y$ be
independent, set $S=X+Y$, and suppose that
$$
 \mathbb E[(|S|-1)^2]\leq \rho^2,\qquad
 0\leq\rho\leq1,\qquad
 \sigma^2=\operatorname{Var}(S)>0.
$$
\begin{claim}\label{cl1}
$
\min\{\operatorname{Var}(X),\operatorname{Var}(Y)\}
 \leq \frac{800\rho^2}{\sigma^2}.
$
\end{claim}
\begin{proof}[Proof of Claim]
Let $(X',Y')$ be an independent copy of $(X,Y)$. Put
$
E=\{-1,1\}
$
and
$
D=E-E=\{-2,0,2\}.
$
Define $\phi:\mathbb R\to D$ by
$$
\phi(t)=
\begin{cases}
 2,  & t>1,\\
 0,  & -1\leq t\leq1,\\
 -2, & t<-1.
\end{cases}
$$
Thus, $\phi(t)$ is a point of $D$ nearest to $t$, and hence
$$
|t-\phi(t)|=\operatorname{dist}(t,D)
:=\min_{d\in D}|t-d|.
$$
Notice also that, for every $z\in\mathbb R$,
$
\operatorname{dist}(z,E)
=\min\{|z-1|,|z+1|\}
=\bigl||z|-1\bigr|.
$
For a square-integrable random variable $Z$, i.e., $
 \mathbb E[Z^2]<\infty$, we use the notation
$$
\|Z\|_2:=\left(\mathbb E[|Z|^2]\right)^{1/2}
$$
for its $L^2$-norm. The assumption therefore gives
$$
\begin{aligned}
\bigl\|\operatorname{dist}(X+Y,E)\bigr\|_2
&=\left(
\mathbb E\left[
\operatorname{dist}(X+Y,E)^2
\right]\right)^{1/2}\\
&=\left(
\mathbb E\bigl[(|X+Y|-1)^2\bigr]
\right)^{1/2}\leq\rho.
\end{aligned}
$$
Moreover, each of
$
X'+Y, X+Y', X'+Y'
$
has the same distribution as $X+Y$, so the same $L^2$ bound holds
with $X+Y$ replaced by any of these three random variables.

Since $D=E-E$, for any $u,v\in\mathbb R$, we have
$$
\begin{aligned}
\operatorname{dist}(u-v,D)
&=\min_{s,s'\in E}|(u-s)-(v-s')|\leq
\min_{s\in E}|u-s|+\min_{s'\in E}|v-s'|\\
&=\operatorname{dist}(u,E)+\operatorname{dist}(v,E).
\end{aligned}
$$
Applying the pointwise inequality with
$
u=X+Y
$
and
$
v=X'+Y,
$
we obtain
$$
\operatorname{dist}(X-X',D)\leq U+V,
$$
where
$
U=\operatorname{dist}(X+Y,E)$ and
$V=\operatorname{dist}(X'+Y,E).$
Since all the random variables in the preceding inequality are
nonnegative,
$$
\bigl\|\operatorname{dist}(X-X',D)\bigr\|_2
\leq\|U+V\|_2.
$$
By the Cauchy--Schwarz inequality,
$$
\mathbb E[UV]\leq
\bigl(\mathbb E[U^2]\bigr)^{1/2}
\bigl(\mathbb E[V^2]\bigr)^{1/2}
=\|U\|_2\|V\|_2.
$$
Consequently,
$$
\begin{aligned}
\|U+V\|_2
&=\left(\mathbb E[(U+V)^2]\right)^{1/2}=\left(
\|U\|_2^2+2\mathbb E[UV]+\|V\|_2^2
\right)^{1/2}\\
&\leq
\left(
\|U\|_2^2+2\|U\|_2\|V\|_2+\|V\|_2^2
\right)^{1/2}\\
&=\|U\|_2+\|V\|_2.
\end{aligned}
$$
This is the triangle inequality in $L^2$.
Since both $X+Y$ and $X'+Y$ have the same distribution and satisfy
the preceding $L^2$ bound, we have $\|U\|_2,\|V\|_2\leq\rho$.
Therefore,
$$
\bigl\|\operatorname{dist}(X-X',D)\bigr\|_2
\leq\|U\|_2+\|V\|_2
\leq2\rho.
$$
Since $|t-\phi(t)|=\operatorname{dist}(t,D)$, 
this is equivalent to
$$
\bigl\|(X-X')-\phi(X-X')\bigr\|_2\leq2\rho.
$$
Applying the same argument with
$
u=X+Y$ and $v=X+Y'
$
gives
$$
\bigl\|(Y-Y')-\phi(Y-Y')\bigr\|_2\leq2\rho.
$$
Finally, taking
$
u=X+Y$ and $ v=X'+Y'$
similarly  gives
$$
\begin{aligned}
\bigl\|\operatorname{dist}(X+Y-X'-Y',D)\bigr\|_2
&\leq
\bigl\|\operatorname{dist}(X+Y,E)\bigr\|_2
+\bigl\|\operatorname{dist}(X'+Y',E)\bigr\|_2\leq2\rho.
\end{aligned}
$$

For any $a,b\in\mathbb R$, the triangle inequality gives
$
\operatorname{dist}(a,D)
\leq |a-b|+\operatorname{dist}(b,D).
$
Indeed, for every $d\in D$,
$
|a-d|\leq|a-b|+|b-d|,
$
and taking the infimum over $d\in D$ gives the assertion.
Apply this pointwise with
$
a=\phi(X-X')+\phi(Y-Y')
$
and
$
b=(X-X')+(Y-Y')=X+Y-X'-Y'.
$
We obtain
$$
\begin{aligned}
\operatorname{dist}\bigl(
\phi(X-X')+\phi(Y-Y'),D
\bigr)
\leq&
\left|
\phi(X-X')+\phi(Y-Y')
-(X+Y-X'-Y')
\right|\\
&+
\operatorname{dist}(X+Y-X'-Y',D)\\
\leq&
\bigl|(X-X')-\phi(X-X')\bigr|
+\bigl|(Y-Y')-\phi(Y-Y')\bigr|\\
&+
\operatorname{dist}(X+Y-X'-Y',D).
\end{aligned}
$$
Taking $L^2$-norms and using the triangle inequality in $L^2$ now
gives
$$
\begin{aligned}
\bigl\|\operatorname{dist}\bigl(
\phi(X-X')+\phi(Y-Y'),D
\bigr)\bigr\|_2
\leq&
\bigl\|(X-X')-\phi(X-X')\bigr\|_2
+\bigl\|(Y-Y')-\phi(Y-Y')\bigr\|_2\\
&+
\bigl\|\operatorname{dist}
(X+Y-X'-Y',D)\bigr\|_2\\
\leq&2\rho+2\rho+2\rho
=6\rho.
\end{aligned}
$$

Set
$
\Delta_X=X-X'$
and
$
\Delta_Y=Y-Y'.
$
Since $X$ and $X'$ are independent and identically distributed,
$\Delta_X$ and $-\Delta_X=X'-X$ have the same distribution. Hence,
$$
\mathbb P(\Delta_X>1)
=\mathbb P(-\Delta_X>1)
=\mathbb P(\Delta_X<-1).
$$
The same observation applies to $\Delta_Y$. We may therefore define
$$
\begin{aligned}
\alpha
=\mathbb P(\Delta_X>1)
 =\mathbb P(\Delta_X<-1),\qquad
\beta=\mathbb P(\Delta_Y>1)
 =\mathbb P(\Delta_Y<-1).
\end{aligned}
$$
Moreover, the pair $(X,X')$ is independent of the pair $(Y,Y')$.
Hence, the random variables $\Delta_X$ and $\Delta_Y$
are independent. Consequently,
$$
\mathbb P(\Delta_X>1,\Delta_Y>1)
=
\mathbb P(\Delta_X>1)\mathbb P(\Delta_Y>1)
=\alpha\beta.
$$
Now define the nonnegative random variable
$$
R=\operatorname{dist}\bigl(
\phi(\Delta_X)+\phi(\Delta_Y),D
\bigr).
$$
On the event
$
\mathcal E=\{\Delta_X>1,\Delta_Y>1\},
$
the definition of $\phi$ gives
$
\phi(\Delta_X)=\phi(\Delta_Y)=2.
$
Therefore, on $\mathcal E$,
$$
R=\operatorname{dist}(4,D)=2.
$$
It follows that
$
R^2=4
$
on $\mathcal E$.
Hence,
$$
4\alpha\beta=4\mathbb P(\mathcal E)
=\mathbb E\bigl[R^2 1_{\mathcal E}\bigr]
\leq\mathbb E[R^2]=\|R\|_2^2.
$$
The preceding $L^2$ estimate gives $\|R\|_2\leq6\rho$. Thus,
$
4\alpha\beta
\leq\|R\|_2^2
\leq(6\rho)^2
=36\rho^2.
$
Then
$$
\alpha\beta\leq9\rho^2.
$$

Since $X$ and $Y$ are independent, we have
$
\sigma^2=\operatorname{Var}(S)
=\operatorname{Var}(X)+\operatorname{Var}(Y).
$
After interchanging $X$ and $Y$ if necessary, we may assume that
$
\operatorname{Var}(X)\geq\frac{\sigma^2}{2}.
$
It follows that
$
\operatorname{Var}(Y)\leq\frac{\sigma^2}{2}.
$
Assume that $\sigma<6\sqrt{\rho}$. Then
$\sigma^4<1296\rho^2$, and hence
$$
\operatorname{Var}(Y)
\leq\frac{\sigma^2}{2}
<\frac{648\rho^2}{\sigma^2}
<\frac{800\rho^2}{\sigma^2},
$$
as required.
Next assume that $\sigma\geq6\sqrt{\rho}$. Since
$0\leq\rho\leq1$, we have $\sqrt{\rho}\geq\rho$, and consequently
$$
\sigma\geq6\rho.
$$
Because $X$ and $X'$ are independent copies,
$$
\|X-X'\|_2^2
=\mathbb E[(X-X')^2]
=2\operatorname{Var}(X).
$$
Moreover, $\phi(X-X')$ equals $2$ on $\{X-X'>1\}$, equals $-2$ on
$\{X-X'<-1\}$, and vanishes otherwise. Since both events have probability $\alpha$, it follows that
$$
\begin{aligned}
\|\phi(X-X')\|_2^2
&=\mathbb E\bigl[\phi(X-X')^2\bigr]=4\mathbb P(X-X'>1)
 +4\mathbb P(X-X'<-1)\\
&=4\alpha+4\alpha=8\alpha.
\end{aligned}
$$
Consequently,
$$
\|\phi(X-X')\|_2=2\sqrt{2\alpha}.
$$
Since $
\|X-X'\|_2^2
=2\operatorname{Var}(X),
$ 
$
\|(X-X')-\phi(X-X')\|_2\leq2\rho,
$
$\sigma\geq6\rho$
and $\operatorname{Var}(X)\geq\frac{\sigma^2}{2},$
it follows from the triangle inequality that
$$
\begin{aligned}
2\sqrt{2\alpha}
&=\|\phi(X-X')\|_2\geq\|X-X'\|_2
-\|(X-X')-\phi(X-X')\|_2\\
&\geq\sqrt{2\operatorname{Var}(X)}-2\rho\geq\sigma-2\rho\geq\frac{2\sigma}{3}.
\end{aligned}
$$
It follows that
$$
\alpha\geq\frac{\sigma^2}{18}.
$$
Combining this with $\alpha\beta\leq9\rho^2$ gives
$$
\beta\leq\frac{9\rho^2}{\alpha}
\leq\frac{162\rho^2}{\sigma^2}.
$$
The same $L^2$ estimate for $Y-Y'$ gives
$$
\begin{aligned}
\sqrt{2\operatorname{Var}(Y)}
&=\|Y-Y'\|_2\leq
\|(Y-Y')-\phi(Y-Y')\|_2
+\|\phi(Y-Y')\|_2\\
&\leq2\rho+2\sqrt{2\beta}\leq2\rho+\frac{36\rho}{\sigma}.
\end{aligned}
$$
Recall that $\mathbb E[(|S|-1)^2]\leq \rho^2$.
Then
$$
\sigma
=\sqrt{\operatorname{Var}(S)}
\leq\|S\|_2=\||S|\|_2\leq\|1\|_2+\||S|-1\|_2\leq1+\rho\leq2.
$$
Thus,
$
2\rho=\frac{2\sigma\rho}{\sigma}
\leq\frac{4\rho}{\sigma}.
$
Consequently,
$
\sqrt{2\operatorname{Var}(Y)}
\leq\frac{40\rho}{\sigma}.
$
Squaring this inequality yields
$$
\operatorname{Var}(Y)
\leq\frac{800\rho^2}{\sigma^2}.
$$
This proves the  claim.
\end{proof}

We now apply the two-summand estimate to the original sum
$L=\sum_{i=1}^nY_i$. Put
$
V=\operatorname{Var}(L)
$
and
$
K=\frac{800\eta}{V}.
$

Suppose first that $V\leq K$. Since $Y_1,\ldots,Y_n$ are independent and each
variance is nonnegative, for every $k\in[n]$, we have
$$
\sum_{i\neq k}\operatorname{Var}(Y_i)
\leq
\sum_{i=1}^n\operatorname{Var}(Y_i)
=V\leq K\leq \frac{1600\eta}{V},
$$
as required.

Assume now that $V>K$. Since
$
\operatorname{Var}\Bigl(\sum_{i\in[n]}Y_i\Bigr)=V>K,
$
there exists a nonempty set $I\subseteq[n]$ that is
inclusion-minimal satisfying
$$
\operatorname{Var}\Bigl(\sum_{i\in I}Y_i\Bigr)>K.
$$
Fix $k\in I$. By the minimality of $I$, we have
$$
\operatorname{Var}\Bigl(
\sum_{i\in I\setminus\{k\}}Y_i
\Bigr)\leq K.
$$
Apply the two-summand estimate to
$
X=\sum_{i\in I}Y_i
$
and
$
Y=\sum_{i\notin I}Y_i.
$
These two random variables are independent and satisfy $X+Y=L$.
Then Claim \ref{cl1} gives 
$$
\min\{\operatorname{Var}(X),\operatorname{Var}(Y)\}\leq \frac{800\eta}{V}=K.
$$
By the choice of $I$, we have $\operatorname{Var}(X)>K$, so necessarily
$$
\operatorname{Var}(Y)
=
\operatorname{Var}\Bigl(\sum_{i\notin I}Y_i\Bigr)
\leq K.
$$
Consequently,
$$
\begin{aligned}
\operatorname{Var}\Bigl(\sum_{i\neq k}Y_i\Bigr)
=
\operatorname{Var}\Bigl(
\sum_{i\in I\setminus\{k\}}Y_i
\Bigr)
+
\operatorname{Var}\Bigl(
\sum_{i\notin I}Y_i
\Bigr)\leq2K
=\frac{1600\eta}{V}.
\end{aligned}
$$
This completes the proof.
\end{proof}

For $\mathbf r=(r_1,\ldots,r_n)\in(0,1)^n$ and a function
$g:2^{[n]}\to\mathbb R$, recall that
$$
\mathbb E_{\mu_{\mathbf r}}[g]
=
\sum_{A\subseteq[n]}\mu_{\mathbf r}(\{A\})g(A),\qquad
\operatorname{Var}_{\mu_{\mathbf r}}(g)
=
\mathbb E_{\mu_{\mathbf r}}\left[
\bigl(g-\mathbb E_{\mu_{\mathbf r}}[g]\bigr)^2
\right].
$$
We also write
$$
\|g\|_{2,\mu_{\mathbf r}}
=
\left(\mathbb E_{\mu_{\mathbf r}}[g^2]\right)^{1/2},\qquad
W_{>1}^{\mu_{\mathbf r}}(g)
=
\sum_{T\subseteq[n], |T|\geq2}
\bigl(\widehat g^{\,\mu_{\mathbf r}}(T)\bigr)^2,
$$
where $\widehat g^{\,\mu_{\mathbf r}}(T)$ is the coefficient indexed
by $T$ in the orthonormal Fourier expansion associated with
$\mu_{\mathbf r}$. Thus, $W_{>1}^{\mu_{\mathbf r}}(g)$ is the total
Fourier weight of $g$ on levels at least two.

The preceding estimate yields the following one-coordinate approximation.

\begin{lemma}\label{lem:one-coordinate-approximation}
Let $\mathbf r=(r_1,\ldots,r_n)\in(0,1)^n$, and let
$\mu_{\mathbf r}$ be the $\mathbf r$-biased measure on
$2^{[n]}$. Suppose that $\mathcal A\subseteq2^{[n]}$ satisfies
$
0<\mu_{\mathbf r}(\mathcal A)<1,
$
and put $f=1_{\mathcal A}$. Then there exists a family
$$
\mathcal H\in
\{\emptyset,2^{[n]}\}
\cup
\{\mathcal S_j,\overline{\mathcal S}_j:j\in[n]\}
$$
such that
$$
\mu_{\mathbf r}(\mathcal A\triangle\mathcal H)
\leq
\frac{3201}
{\operatorname{Var}_{\mu_{\mathbf r}}(f)}
W_{>1}^{\mu_{\mathbf r}}(f),
$$
where
$
\mathcal S_j=\{A\subseteq[n]:j\in A\}
$
and
$
\overline{\mathcal S}_j
=\{A\subseteq[n]:j\notin A\}.
$
\end{lemma}

\begin{proof}
Put
$
F=2f-1.
$
Then $F$ takes values in $\{-1,1\}$. Since
$
\{v_{\mathbf r}(T):T\subseteq[n]\}
$
is an orthonormal basis with respect to
$\langle\cdot,\cdot\rangle_{\mu_{\mathbf r}}$, the function $F$ has
the unique Fourier expansion
$$
F
=
\sum_{T\subseteq[n]}
\widehat F_{\mathbf r}(T)v_{\mathbf r}(T),
$$
where
$
\widehat F_{\mathbf r}(T)
=
\langle F,v_{\mathbf r}(T)\rangle_{\mu_{\mathbf r}}.
$
Since $v_{\mathbf r}(\emptyset)=1$, we have
$
\widehat F_{\mathbf r}(\emptyset)
=
\langle F,1\rangle_{\mu_{\mathbf r}}
=
\mathbb E_{\mu_{\mathbf r}}[F].
$
Separating the terms according to their degrees, define
$$
m
=
\widehat F_{\mathbf r}(\emptyset)
=
\mathbb E_{\mu_{\mathbf r}}[F],
\qquad
X_i
=
\widehat F_{\mathbf r}(\{i\})
v_{\mathbf r}(\{i\}),
\qquad 
R
=
\sum_{T\subseteq[n], |T|\geq2}
\widehat F_{\mathbf r}(T)v_{\mathbf r}(T).
$$
Then $m$ is finite and
$$
F=m+\sum_{i=1}^nX_i+R.
$$
The basis function $v_{\mathbf r}(\{i\})$ depends only on whether
$i$ belongs to the random set. Hence, $X_i$ depends only on coordinate
$i$. Since the
basis is orthonormal and 
$v_{\mathbf r}(\emptyset)=1$, we have
$
\left\langle
v_{\mathbf r}(\{i\}),1
\right\rangle_{\mu_{\mathbf r}}
=0.
$
Thus,
$$
\mathbb E_{\mu_{\mathbf r}}[X_i]
=
\widehat F_{\mathbf r}(\{i\})
\mathbb E_{\mu_{\mathbf r}}
[v_{\mathbf r}(\{i\})]
=0.
$$
Since the coordinates are independent under
$\mu_{\mathbf r}$,
$X_1,\ldots,X_n$ are independent and have mean zero.

Set
$$
w=W_{>1}^{\mu_{\mathbf r}}(f),
\qquad
s^2=\operatorname{Var}_{\mu_{\mathbf r}}(f).
$$
Since $f=1_{\mathcal A}$, we have $f^2=f$ and
$
\mathbb E_{\mu_{\mathbf r}}[f]
=
\mu_{\mathbf r}(\mathcal A).
$
Thus,
$$
\begin{aligned}
s^2=\operatorname{Var}_{\mu_{\mathbf r}}(f)=\mathbb E_{\mu_{\mathbf r}}[f^2]
-\bigl(\mathbb E_{\mu_{\mathbf r}}[f]\bigr)^2=\mu_{\mathbf r}(\mathcal A)
-\mu_{\mathbf r}(\mathcal A)^2=\mu_{\mathbf r}(\mathcal A)
\bigl(1-\mu_{\mathbf r}(\mathcal A)\bigr).
\end{aligned}
$$
It follows that
$$
0<s^2\leq\frac14.
$$
For every nonempty $T\subseteq[n]$, orthogonality to the constant
function gives
$
\left\langle
1,v_{\mathbf r}(T)
\right\rangle_{\mu_{\mathbf r}}=0.
$
Hence,
$$
\begin{aligned}
\widehat F_{\mathbf r}(T)
=
\left\langle
2f-1,v_{\mathbf r}(T)
\right\rangle_{\mu_{\mathbf r}}=
2\widehat f_{\mathbf r}(T).
\end{aligned}
$$
Recall that
$
\|R\|_{2,\mu_{\mathbf r}}^2
=
\mathbb E_{\mu_{\mathbf r}}[R^2]
=
\langle R,R\rangle_{\mu_{\mathbf r}}.
$
Since
$
R=
\sum_{T\subseteq[n], |T|\geq2}
\widehat F_{\mathbf r}(T)v_{\mathbf r}(T),
$
orthonormality gives
$$
\begin{aligned}
\|R\|_{2,\mu_{\mathbf r}}^2
&=
\sum_{T,U\subseteq[n], |T|,|U|\geq2}
\widehat F_{\mathbf r}(T)
\widehat F_{\mathbf r}(U)
\left\langle
v_{\mathbf r}(T),v_{\mathbf r}(U)
\right\rangle_{\mu_{\mathbf r}}\\
&=
\sum_{T\subseteq[n], |T|\geq2}
\widehat F_{\mathbf r}(T)^2=
4\sum_{T\subseteq[n], |T|\geq2}
\widehat f_{\mathbf r}(T)^2\\
&=4W_{>1}^{\mu_{\mathbf r}}(f)
=4w.
\end{aligned}
$$

Since $v_{\mathbf r}(\emptyset)=1$, we have
$
\widehat f_{\mathbf r}(\emptyset)
=
\mathbb E_{\mu_{\mathbf r}}[f].
$
Then
$
f-\mathbb E_{\mu_{\mathbf r}}[f]
=
\sum_{T\subseteq[n], T\neq\emptyset}
\widehat f_{\mathbf r}(T)v_{\mathbf r}(T).
$
Using orthonormality, we obtain
$$
\begin{aligned}
s^2
&=\operatorname{Var}_{\mu_{\mathbf r}}(f)
=
\left\|
f-\mathbb E_{\mu_{\mathbf r}}[f]
\right\|_{2,\mu_{\mathbf r}}^2=
\left\langle
\sum_{T\neq\emptyset}
\widehat f_{\mathbf r}(T)v_{\mathbf r}(T),
\sum_{U\neq\emptyset}
\widehat f_{\mathbf r}(U)v_{\mathbf r}(U)
\right\rangle_{\mu_{\mathbf r}}\\
&=
\sum_{T\neq\emptyset}
\widehat f_{\mathbf r}(T)^2=
\sum_{i=1}^n
\widehat f_{\mathbf r}(\{i\})^2
+
W_{>1}^{\mu_{\mathbf r}}(f)=
\sum_{i=1}^n
\widehat f_{\mathbf r}(\{i\})^2+w.
\end{aligned}
$$
Since  $X_1,\ldots,X_n$ are independent and have mean
zero, we have
\begin{align}
\operatorname{Var}_{\mu_{\mathbf r}}
\Bigl(\sum_{i=1}^nX_i\Bigr)
&=
\sum_{i=1}^n\|X_i\|_{2,\mu_{\mathbf r}}^2=\sum_{i=1}^n\left\langle X_i,X_i\right\rangle_{\mu_{\mathbf r}}\notag\\
&=
\sum_{i=1}^n
\widehat F_{\mathbf r}(\{i\})^2=
4\sum_{i=1}^n
\widehat f_{\mathbf r}(\{i\})^2\notag\\
&=4(s^2-w).\label{af41}
\end{align}

Suppose first that $w>s^2/2$. Define
$$
\mathcal H=
\begin{cases}
\emptyset,
&\mu_{\mathbf r}(\mathcal A)\leq1/2,\\
2^{[n]},
&\mu_{\mathbf r}(\mathcal A)>1/2.
\end{cases}
$$
If $\mathcal H=\emptyset$, then
$
\mu_{\mathbf r}(\mathcal A\triangle\mathcal H)
=
\mu_{\mathbf r}(\mathcal A),
$
whereas if $\mathcal H=2^{[n]}$, then
$
\mu_{\mathbf r}(\mathcal A\triangle\mathcal H)
=
1-\mu_{\mathbf r}(\mathcal A).
$
Consequently,
$$
\mu_{\mathbf r}(\mathcal A\triangle\mathcal H)
=
\min\left\{
\mu_{\mathbf r}(\mathcal A),
1-\mu_{\mathbf r}(\mathcal A)
\right\}
\leq\frac12.
$$
Since $w>s^2/2$, we have
$$
\frac12<\frac{w}{s^2}
\leq\frac{3201w}{s^2}.
$$
Therefore,
$$
\mu_{\mathbf r}(\mathcal A\triangle\mathcal H)
\leq\frac{3201w}{s^2},
$$
as required.

Next assume that
$
w\leq\frac{s^2}{2}.
$
Define
$
L=m+\sum_{i=1}^nX_i.
$
For every $A\subseteq[n]$, we have $F(A)\in\{-1,1\}$. Moreover,
$$
\begin{aligned}
\operatorname{dist}\bigl(L(A),\{-1,1\}\bigr)
=
\min\{|L(A)-1|,|L(A)+1|\}=
\bigl||L(A)|-1\bigr|.
\end{aligned}
$$
Since $F(A)$ is one of the two points $-1$ and $1$, we have
$$
\bigl||L(A)|-1\bigr|
\leq|L(A)-F(A)|.
$$
Thus, 
$
\bigl||L|-1\bigr|^2\leq(L-F)^2.
$
Taking expectation with respect to $\mu_{\mathbf r}$ gives
$$
\mathbb E_{\mu_{\mathbf r}}\bigl[(|L|-1)^2\bigr]
\leq
\mathbb E_{\mu_{\mathbf r}}\bigl[(L-F)^2\bigr].
$$
Since
$
F=L+R,
$
we have $L-F=-R$, and hence
$$
\begin{aligned}
\mathbb E_{\mu_{\mathbf r}}\bigl[(|L|-1)^2\bigr]
\leq
\mathbb E_{\mu_{\mathbf r}}[R^2]=\|R\|_{2,\mu_{\mathbf r}}^2=4w.
\end{aligned}
$$
Since $w\leq s^2/2$ and $s^2\leq1/4$, it follows that
$
4w\leq2s^2\leq\frac12.
$
By \eqref{af41}, we have
$$
\begin{aligned}
\operatorname{Var}_{\mu_{\mathbf r}}(L)
=
\operatorname{Var}_{\mu_{\mathbf r}}
\Bigl(\sum_{i=1}^nX_i\Bigr)=4(s^2-w)\geq2s^2>0.
\end{aligned}
$$
We now apply Lemma~\ref{lem:one-large-summand} to the independent
random variables
$$
Y_1=m+X_1,
\qquad
Y_i=X_i\quad\text{for }i=2,\ldots,n.
$$
Their sum is $L$.  Moreover, these random variables have finite second moments. Indeed,
by the Cauchy--Schwarz inequality,
$
\left|\widehat F_{\mathbf r}(\{i\})\right|
\leq
\|F\|_{2,\mu_{\mathbf r}}
\|v_{\mathbf r}(\{i\})\|_{2,\mu_{\mathbf r}}
=1,
$
and hence
$
\mathbb E_{\mu_{\mathbf r}}[X_i^2]
=
\widehat F_{\mathbf r}(\{i\})^2
<\infty.
$
Since $m$ is finite and $\mathbb E_{\mu_{\mathbf r}}[X_1]=0$, we have
$$
\begin{aligned}
\mathbb E_{\mu_{\mathbf r}}[(m+X_1)^2]
&=
m^2
+2m\mathbb E_{\mu_{\mathbf r}}[X_1]
+\mathbb E_{\mu_{\mathbf r}}[X_1^2]\\
&=
m^2+\mathbb E_{\mu_{\mathbf r}}[X_1^2]<\infty.
\end{aligned}
$$
Taking $4w$ as the parameter $\eta$ in that Lemma~\ref{lem:one-large-summand},
we obtain an index $k\in[n]$ such that
$$
\operatorname{Var}_{\mu_{\mathbf r}}
\Bigl(\sum_{i\neq k}Y_i\Bigr)
\leq
\frac{1600\cdot4w}
{4(s^2-w)}.
$$
The random variable $\sum_{i\neq k}Y_i$ differs from
$\sum_{i\neq k}X_i$ by a constant, either $0$ or $m$. Hence, their
variances are equal. Since $s^2-w\geq s^2/2$, it follows that
$$
\begin{aligned}
\operatorname{Var}_{\mu_{\mathbf r}}
\Bigl(\sum_{i\neq k}X_i\Bigr)
\leq
\frac{1600\cdot4w}
{4(s^2-w)}\leq
\frac{3200w}{s^2}.
\end{aligned}
$$

Let
$
G=m+X_k.
$
Then
$
F-G=\sum_{i\neq k}X_i+R.
$
Put
$
U=\sum_{i\neq k}X_i.
$
Observe that $U$ is a sum of degree-one Fourier components, whereas
$R$ is a sum of Fourier components of degree at least two. Hence,
by orthogonality of the Fourier basis,
$
\langle U,R\rangle_{\mu_{\mathbf r}}=0.
$
It follows that
$$
\begin{aligned}
\|F-G\|_{2,\mu_{\mathbf r}}^2
&=\|U+R\|_{2,\mu_{\mathbf r}}^2=\|U\|_{2,\mu_{\mathbf r}}^2
+2\langle U,R\rangle_{\mu_{\mathbf r}}
+\|R\|_{2,\mu_{\mathbf r}}^2\\
&=\|U\|_{2,\mu_{\mathbf r}}^2
+\|R\|_{2,\mu_{\mathbf r}}^2.
\end{aligned}
$$
Since each $X_i$ has mean zero, we have
$
\mathbb E_{\mu_{\mathbf r}}[U]=0.
$
Therefore,
$$
\|U\|_{2,\mu_{\mathbf r}}^2
=
\operatorname{Var}_{\mu_{\mathbf r}}(U).
$$
Using the preceding estimates, we obtain
$$
\begin{aligned}
\|F-G\|_{2,\mu_{\mathbf r}}^2
&=
\operatorname{Var}_{\mu_{\mathbf r}}
\Bigl(\sum_{i\neq k}X_i\Bigr)
+\|R\|_{2,\mu_{\mathbf r}}^2\leq
\frac{3200w}{s^2}+4w.
\end{aligned}
$$
Since $s^2\leq1/4$, we have
$
4w\leq\frac{w}{s^2}.
$
Thus,
$$
\|F-G\|_{2,\mu_{\mathbf r}}^2
\leq
\frac{3201w}{s^2}.
$$

Recall that $
G=m+X_k.$
Define a $\{-1,1\}$-valued function $\widetilde F$ by
$$
\widetilde F(A)=
\begin{cases}
1,&G(A)\geq0,\\
-1,&G(A)<0.
\end{cases}
$$
If $\widetilde F(A)\neq F(A)$, then one of the following two cases
occurs. If $F(A)=1$, then $\widetilde F(A)=-1$, and the definition of
$\widetilde F$ implies that $G(A)<0$. Hence,
$
|F(A)-G(A)|
=
|1-G(A)|
=
1-G(A)
\geq1.
$
If $F(A)=-1$, then $\widetilde F(A)=1$, so $G(A)\geq0$. Therefore,
$
|F(A)-G(A)|
=
|-1-G(A)|
=
1+G(A)
\geq1.
$
Thus, in either case,
\begin{align}
|F(A)-G(A)|\geq1.\label{f34}
\end{align}
Since
$
X_k
=
\widehat F_{\mathbf r}(\{k\})
v_{\mathbf r}(\{k\}),
$
the value of $X_k(B)$ depends only on whether $k\in B$. Hence,
$G(B)=m+X_k(B)$ and  $\widetilde F(B)$ depend only
on whether $k\in B$.
Thus, $\widetilde F$ is constant on each of the two classes
$$
\{B\subseteq[n]:k\in B\}
\qquad\text{and}\qquad
\{B\subseteq[n]:k\notin B\}.
$$
Define
$$
\mathcal H
=
\{B\subseteq[n]:\widetilde F(B)=1\}.
$$
Since $\widetilde F$ is constant on each of the two classes
$
\{B\subseteq[n]:k\in B\}
$
and
$
\{B\subseteq[n]:k\notin B\},
$
there are only four possibilities. If $\widetilde F=-1$ on both
classes, then $\mathcal H=\emptyset$. If $\widetilde F=1$ on both
classes, then $\mathcal H=2^{[n]}$. If $\widetilde F=1$ exactly when
$k\in B$, then $\mathcal H=\mathcal S_k$. Finally, if
$\widetilde F=1$ exactly when $k\notin B$, then
$\mathcal H=\overline{\mathcal S}_k$. Therefore,
$$
\mathcal H\in
\{\emptyset,2^{[n]},
\mathcal S_k,\overline{\mathcal S}_k\}.
$$

For every $B\subseteq[n]$, the definition $F=2\,1_{\mathcal A}-1$
shows that $F(B)=1$ when $B\in\mathcal A$ and $F(B)=-1$ when
$B\notin\mathcal A$. Similarly, by the definition of $\mathcal H$,
the function $\widetilde F$ equals $1$ on $\mathcal H$ and $-1$
outside $\mathcal H$. Therefore, $F(B)$ and $\widetilde F(B)$ are
different precisely when $B$ belongs to the symmetric difference
$\mathcal A\triangle\mathcal H$.
If $B\in\mathcal A\triangle\mathcal H$, then
$F(B)\neq\widetilde F(B)$. By \eqref{f34}, we have
$$
|F(B)-G(B)|\geq1.
$$
Consequently, for every $B\subseteq[n]$,
$$
1_{\mathcal A\triangle\mathcal H}(B)
\leq
(F(B)-G(B))^2.
$$
Taking expectation with respect to $\mu_{\mathbf r}$ gives
$$
\begin{aligned}
\mu_{\mathbf r}(\mathcal A\triangle\mathcal H)
&=
\mathbb E_{\mu_{\mathbf r}}
[1_{\mathcal A\triangle\mathcal H}]\leq
\mathbb E_{\mu_{\mathbf r}}\bigl[(F-G)^2\bigr]\\
&=
\|F-G\|_{2,\mu_{\mathbf r}}^2\leq
\frac{3201w}{s^2}\\
&=
\frac{3201}
{\operatorname{Var}_{\mu_{\mathbf r}}(f)}
W_{>1}^{\mu_{\mathbf r}}(f).
\end{aligned}
$$
This proves the lemma.
\end{proof}

Recall that
$$
 p=p_1,\qquad q=q_1,\qquad
 \mu_1=\mu_{\mathbf p},\qquad
 \mu_2=\mu_{\mathbf q}.
$$
and
$$
 0<q\leq p<\frac12,\qquad
 p_\ell\leq p,\qquad q_\ell\leq q
 \quad\text{for every }\ell\in[n].
$$
Near extremality also keeps the measures of the two increasing
families away from zero and one. This provides the variance lower
bound needed when Lemma~\ref{lem:one-coordinate-approximation} is
applied.

\begin{lemma}\label{lem:variance-lower-bound}
Suppose that $\mathcal A,\mathcal B\subseteq 2^{[n]}$ are increasing
and cross-intersecting, $0\leq\varepsilon\leq 1/2$, and
$$
\mu_1(\mathcal A)\mu_2(\mathcal B)
\geq
(1-\varepsilon)pq.
$$
Set
$
\rho_0=\frac{pq}{2}.
$
Then
$$
\rho_0\leq\mu_1(\mathcal A)\leq1-\rho_0,
\qquad
\rho_0\leq\mu_2(\mathcal B)\leq1-\rho_0.
$$
Moreover,
$$
\operatorname{Var}_{\mu_1}(1_{\mathcal A})
\geq
\rho_0(1-\rho_0),
\qquad
\operatorname{Var}_{\mu_2}(1_{\mathcal B})
\geq
\rho_0(1-\rho_0).
$$
\end{lemma}

\begin{proof}
Since each family has measure at most one, we have
$$
\begin{aligned}
\mu_1(\mathcal A)
\geq
\mu_1(\mathcal A)\mu_2(\mathcal B)
\geq
(1-\varepsilon)pq
\geq\rho_0,\qquad
\mu_2(\mathcal B)
\geq
\mu_1(\mathcal A)\mu_2(\mathcal B)
\geq
(1-\varepsilon)pq
\geq\rho_0.
\end{aligned}
$$

We next prove the upper bounds. Recall that the transversal of
$\mathcal A$ is
$$
\mathcal A^*
=
\{B\subseteq[n]:
B\cap A\neq\emptyset
\text{ for every }A\in\mathcal A\}.
$$
Since $\mathcal A$ and $\mathcal B$ are cross-intersecting, we have
$\mathcal B\subseteq\mathcal A^*$. Because $\mathcal A$ is
increasing, Lemma~\ref{lem:transversal} gives
$$
\mu_2(\mathcal A^*)
=
1-\mu_{\mathbf 1-\mathbf q}(\mathcal A).
$$
Moreover, since $p_i,q_i<1/2$, we have
$p_i\leq1-q_i$ for every $i\in[n]$. As mentioned in Subsection \ref{sec:endpoint}, the measure of an increasing
family is nondecreasing in each coordinate probability.
Hence,
$$
\mu_{\mathbf p}(\mathcal A)
\leq
\mu_{\mathbf 1-\mathbf q}(\mathcal A).
$$
Consequently,
$$
\begin{aligned}
\mu_2(\mathcal B)
&\leq
\mu_2(\mathcal A^*)=
1-\mu_{\mathbf 1-\mathbf q}(\mathcal A)\leq
1-\mu_1(\mathcal A).
\end{aligned}
$$
Since $\mu_2(\mathcal B)\geq\rho_0$, it follows that
$\mu_1(\mathcal A)\leq1-\rho_0$. Applying the same argument to
$\mathcal B$, and using $q_i\leq1-p_i$ for every $i$, gives
$\mu_2(\mathcal B)\leq1-\rho_0$.

Finally, if $t=\mu_1(\mathcal A)$, then
$
\operatorname{Var}_{\mu_1}(1_{\mathcal A})
=
\mathbb E_{\mu_1}[1_{\mathcal A}^2]
-
\mathbb E_{\mu_1}[1_{\mathcal A}]^2=
t-t^2=
t(1-t).
$
Since $\rho_0\leq t\leq1-\rho_0$, we have
$$
t(1-t)\geq\rho_0(1-\rho_0).
$$
The same argument applies to $1_{\mathcal B}$ under $\mu_2$.
\end{proof}

The next observation rules out the decreasing one-coordinate family
$
\overline{\mathcal S}_j
=
\{A\subseteq[n]:j\notin A\}
$
when the family being approximated is increasing.

\begin{lemma}\label{lem:exclude-antistar}
Let $\mathbf r=(r_1,\ldots,r_n)\in(0,1/2)^n$, let
$\mathcal A\subseteq2^{[n]}$ be increasing, and let $j\in[n]$. If
$
\mu_{\mathbf r}
\bigl(\mathcal A\triangle\overline{\mathcal S}_j\bigr)
\leq d,
$
then
$$
\mu_{\mathbf r}
\bigl(\mathcal A\triangle2^{[n]}\bigr)
\leq2d.
$$
\end{lemma}

\begin{proof}
Let $\mathbf r_{-j}$ be obtained from $\mathbf r$ by deleting its
$j$th coordinate. Recall that 
$$
\begin{aligned}
\mathcal A(\bar j)
=
\{C\subseteq[n]\setminus\{j\}:C\in\mathcal A\},\qquad \mathcal A(j)
=
\{C\subseteq[n]\setminus\{j\}:
C\cup\{j\}\in\mathcal A\}.
\end{aligned}
$$
Since $\mathcal A$ is increasing, we have
$
\mathcal A(\bar j)\subseteq\mathcal A(j).
$
Put
$$
r=r_j,
\qquad
x=\mu_{\mathbf r_{-j}}\bigl(\mathcal A(\bar j)\bigr),
\qquad
y=\mu_{\mathbf r_{-j}}\bigl(\mathcal A(j)\bigr).
$$
Then
$
0\leq x\leq y\leq1.
$
Let $X$ be a random subset of $[n]$ distributed according to
$\mu_{\mathbf r}$. 
If $j\notin X$, then $X\in\overline{\mathcal S}_j$. Hence,
$X\in\mathcal A\triangle\overline{\mathcal S}_j$ exactly when
$X\notin\mathcal A$. Since the conditional distribution of
$X\setminus\{j\}$ is $\mu_{\mathbf r_{-j}}$, the conditional
probability of this event is $1-x$.
If $j\in X$, then $X\notin\overline{\mathcal S}_j$. Hence,
$X\in\mathcal A\triangle\overline{\mathcal S}_j$ exactly when
$X\in\mathcal A$, whose conditional probability is $y$.
Since $\mathbb P(j\notin X)=1-r$ and $\mathbb P(j\in X)=r$, the law
of total probability gives
$$
\mu_{\mathbf r}
\bigl(\mathcal A\triangle\overline{\mathcal S}_j\bigr)
=
(1-r)(1-x)+ry.
$$
Using $x\leq y$, $r<1/2$, and $y\leq1$, we obtain
$$
\begin{aligned}
(1-r)(1-x)+ry
\geq
(1-r)(1-y)+ry=
1-r-(1-2r)y\geq r.
\end{aligned}
$$
Combining this with
$
\mu_{\mathbf r}
\bigl(\mathcal A\triangle\overline{\mathcal S}_j\bigr)
\leq d,
$
we obtain $d\geq r$. 
Since $\mathcal A\subseteq2^{[n]}$, we have
$
\mathcal A\triangle2^{[n]}=2^{[n]}\setminus\mathcal A.
$
It follows that
$$
\mu_{\mathbf r}
\bigl(\mathcal A\triangle2^{[n]}\bigr)
=1-\mu_{\mathbf r}(\mathcal A)=1-(1-r)x-ry=
(1-r)(1-x)+r(1-y).
$$
Observe that
$$
(1-r)(1-x)
\leq
(1-r)(1-x)+ry
=
\mu_{\mathbf r}
\bigl(\mathcal A\triangle\overline{\mathcal S}_j\bigr)
\leq d.
$$
Therefore,
$$
\begin{aligned}
\mu_{\mathbf r}
\bigl(\mathcal A\triangle2^{[n]}\bigr)
=
(1-r)(1-x)+r(1-y)\leq
d+r\leq 2d,
\end{aligned}
$$
where the last inequality follows from $r\leq d$.
\end{proof}

The original families need not be increasing. The following
quantitative monotone reduction shows that replacing them by their
upsets changes their measures by at most $\varepsilon$.

\begin{lemma}\label{lem:upset-cost}
Let $\varepsilon\geq0$. Suppose that
$\mathcal A,\mathcal B\subseteq2^{[n]}$ are cross-intersecting and
$$
\mu_1(\mathcal A)\mu_2(\mathcal B)
\geq
(1-\varepsilon)pq.
$$
Then
$$
\mu_1(\mathcal A^\uparrow\setminus\mathcal A)
\leq
\varepsilon,
\qquad
\mu_2(\mathcal B^\uparrow\setminus\mathcal B)
\leq
\varepsilon.
$$
\end{lemma}

\begin{proof}
If $\varepsilon\geq1$, both conclusions are immediate because a
probability is at most one. We may  assume that
$
0\leq\varepsilon<1.
$
By Lemma \ref{lem:upclosure}, the upsets $\mathcal A^\uparrow$ and
$\mathcal B^\uparrow$ are cross-intersecting. Since
$p_i\leq p$ and $q_i\leq q$ for every $i\in[n]$, while
$p_1=p$ and $q_1=q$, we have
$
\max_{i\in[n]}p_iq_i=pq.
$
Applying Theorem~ \ref{thm:main-product} to
$\mathcal A^\uparrow,\mathcal B^\uparrow$ gives
$$
\mu_1(\mathcal A^\uparrow)
\mu_2(\mathcal B^\uparrow)
\leq pq.
$$
Combining this with the assumed lower bound gives
$$
\frac{\mu_1(\mathcal A)}
     {\mu_1(\mathcal A^\uparrow)}
\,
\frac{\mu_2(\mathcal B)}
     {\mu_2(\mathcal B^\uparrow)}
\geq
1-\varepsilon.
$$
The denominators are nonzero because
$(1-\varepsilon)pq>0$ and
$
\mathcal A\subseteq\mathcal A^\uparrow,
\mathcal B\subseteq\mathcal B^\uparrow.
$
Both ratios lie in $[0,1]$. Since their product is at least
$1-\varepsilon$, each ratio is itself at least $1-\varepsilon$.
Consequently,
$$
\begin{aligned}
\mu_1(\mathcal A^\uparrow\setminus\mathcal A)
&=
\mu_1(\mathcal A^\uparrow)-\mu_1(\mathcal A)=
\mu_1(\mathcal A^\uparrow)
\left(
1-
\frac{\mu_1(\mathcal A)}
     {\mu_1(\mathcal A^\uparrow)}
\right)\\
&\leq
\varepsilon\mu_1(\mathcal A^\uparrow)\leq
\varepsilon.
\end{aligned}
$$
The proof of the corresponding bound for $\mathcal B$ is identical.
\end{proof}

It remains to show that the two stars supplied by the
one-coordinate approximation have the same centre.

\begin{lemma}\label{lem:star-alignment}
There exists a constant $\delta_0=\delta_0(p,q)>0$ with the
following property. Suppose that
$\mathcal A,\mathcal B\subseteq2^{[n]}$ are cross-intersecting and
that, for some $a,b\in[n]$ and $0<\delta<\delta_0$,
$$
\mu_1(\mathcal A\triangle\mathcal S_a)\leq\delta,
\qquad
\mu_2(\mathcal B\triangle\mathcal S_b)\leq\delta.
$$
If
$
\mu_1(\mathcal A)\mu_2(\mathcal B)
\geq
(1-\delta)pq,
$
then $a=b$.
\end{lemma}

\begin{proof}
Set
$$
L=\frac{3}{pq(1-p)}
$$
and choose
$$
\delta_0
=
\min\left\{
\frac{2pq}{3(1+pq)},
\frac{1-\sqrt{pq}}{2L}
\right\}.
$$
Since $0<pq<1$, this number is positive and depends only on
$p$ and $q$.

Suppose, for a contradiction, that $a\neq b$. Since
$
\mu_1(\mathcal S_a)=p_a
$ 
and
$
\mu_2(\mathcal S_b)=q_b,
$
we have
$$
\begin{aligned}
\bigl|\mu_1(\mathcal A)-p_a\bigr|
=
\bigl|
\mu_1(\mathcal A)-\mu_1(\mathcal S_a)
\bigr|\leq
\mu_1(\mathcal A\triangle\mathcal S_a)
\leq\delta,
\end{aligned}
$$
and similarly,
$$
\bigl|\mu_2(\mathcal B)-q_b\bigr|
\leq
\delta.
$$
Moreover, since each family has measure at most one, the product
lower bound implies
$$
\mu_1(\mathcal A),\mu_2(\mathcal B)
\geq
(1-\delta)pq.
$$
It follows that
$$
\begin{aligned}
p_a
\geq
\mu_1(\mathcal A)-\delta\geq
(1-\delta)pq-\delta=
pq-(1+pq)\delta\geq
\frac{pq}{3},
\end{aligned}
$$
where we used
$
\delta<\delta_0
\leq
\frac{2pq}{3(1+pq)}.
$
The same calculation gives
$$
q_b\geq\frac{pq}{3}.
$$

Let
$
N=[n]\setminus\{a,b\},
$
and define two families on $N$ by
$$
\begin{aligned}
\mathcal C
=
\{R\subseteq N:R\cup\{a\}\in\mathcal A\},\qquad
\mathcal D
=
\{T\subseteq N:T\cup\{b\}\in\mathcal B\}.
\end{aligned}
$$
Since $\mathcal A,\mathcal B$ are cross-intersecting, we infer that
 $\mathcal C,\mathcal D$ are cross-intersecting on $N$. 
Let
$$
\mathbf p_N=(p_i)_{i\in N},
\qquad
\mathbf q_N=(q_i)_{i\in N}.
$$
Thus, $\mu_{\mathbf p_N}$ and $\mu_{\mathbf q_N}$ are the
restrictions of $\mu_1$ and $\mu_2$, respectively, to $2^N$.

Let $X$ be distributed according to $\mu_1$. Consider the event
$$
a\in X,
\qquad
b\notin X,
\qquad
X\cap N\notin\mathcal C.
$$
On this event,
$
X=(X\cap N)\cup\{a\}.
$
By the definition of $\mathcal C$, this implies
$X\notin\mathcal A$. Hence, this event is contained in
$\mathcal S_a\setminus\mathcal A$. Independence of the coordinates
therefore gives
$$
\begin{aligned}
\delta
\geq
\mu_1(\mathcal S_a\setminus\mathcal A)\geq
p_a(1-p_b)
\bigl(1-\mu_{\mathbf p_N}(\mathcal C)\bigr).
\end{aligned}
$$
Since $p_a\geq pq/3$ and $p_b\leq p$, we obtain
$$
\begin{aligned}
\mu_{\mathbf p_N}(\mathcal C)
\geq
1-\frac{\delta}{p_a(1-p_b)}\geq
1-\frac{3\delta}{pq(1-p)}=
1-L\delta.
\end{aligned}
$$

Similarly, let $X$ be distributed according to $\mu_2$. The event
$$
b\in X,
\qquad
a\notin X,
\qquad
X\cap N\notin\mathcal D
$$
is contained in $\mathcal S_b\setminus\mathcal B$. Since
$q_b\geq pq/3$ and $q_a\leq q\leq p$, the same calculation gives
$$
\mu_{\mathbf q_N}(\mathcal D)
\geq
1-L\delta.
$$

Our choice of $\delta_0$ ensures that
$$
1-L\delta
>
1-L\delta_0
\geq
\frac{1+\sqrt{pq}}{2}
>
\sqrt{pq}.
$$
Since
$
\mu_{\mathbf p_N}(\mathcal C),
\mu_{\mathbf q_N}(\mathcal D)
\geq
1-L\delta
>
0,
$
both $\mathcal C$ and $\mathcal D$ are nonempty. Choose
$R\in\mathcal C$ and $T\in\mathcal D$. Since $\mathcal C$ and
$\mathcal D$ are cross-intersecting,
$$
\varnothing\neq R\cap T\subseteq N.
$$
Therefore, $N\neq\varnothing$.
The restricted coordinate probabilities still lie in $(0,1/2)$,
and
$
\max_{i\in N}p_iq_i\leq pq.
$
Applying
Theorem \ref{thm:main-product}  to the cross-intersecting families
$\mathcal C,\mathcal D\subseteq2^N$ gives
$$
\mu_{\mathbf p_N}(\mathcal C)
\mu_{\mathbf q_N}(\mathcal D)
\leq
\max_{i\in N}p_iq_i
\leq
pq.
$$
On the other hand,
$$
\mu_{\mathbf p_N}(\mathcal C)
\mu_{\mathbf q_N}(\mathcal D)
\geq
(1-L\delta)^2
>
pq,
$$
which is a contradiction. Therefore $a=b$.
\end{proof}

\subsection{Completion of the proof}

\begin{proof}[Proof of Theorem \ref{thm:main-stability}]
Since $p_i\leq p$ and $q_i\leq q$ for every $i\in[n]$, with equality
in both inequalities when $i=1$, we have
$
\max_{i\in[n]}p_iq_i=pq.
$
The hypothesis of the theorem and Theorem~\ref{thm:main-product}  give
$$
(1-\varepsilon)pq
\leq
\mu_1(\mathcal A)\mu_2(\mathcal B)
\leq
pq.
$$
Thus, the hypotheses cannot hold when $\varepsilon<0$.

Suppose first that $\varepsilon=0$. Then equality holds in
Theorem~\ref{thm:main-product}. Since $p,q<1/2$, we have $pq<1/4$, and the equality
characterization in Theorem~\ref{thm:main-product} gives
$
\mathcal A=\mathcal B=\mathcal S_j
$
for some $j\in[n]$.  Consequently,
$$
\mu_1(\mathcal A\triangle\mathcal S_j)
=
\mu_2(\mathcal B\triangle\mathcal S_j)
=
0.
$$
Thus, the required estimate holds when $\varepsilon=0$.

Next assume that $\varepsilon>0$.
Let
$
\widetilde{\mathcal A}=\mathcal A^\uparrow
$
and
$
\widetilde{\mathcal B}=\mathcal B^\uparrow.
$
By Lemma \ref{lem:upclosure}, these families are increasing and cross-intersecting.
Moreover,
$$
\mu_1(\widetilde{\mathcal A})
\mu_2(\widetilde{\mathcal B})
\geq
\mu_1(\mathcal A)\mu_2(\mathcal B)
\geq
(1-\varepsilon)pq.
$$
Lemma~\ref{lem:upset-cost} gives
\begin{align}
\mu_1(\widetilde{\mathcal A}\setminus\mathcal A)
\leq
\varepsilon,
\qquad
\mu_2(\widetilde{\mathcal B}\setminus\mathcal B)
\leq
\varepsilon.\label{f35}
\end{align}

Put
$$
\rho_0=\frac{pq}{2},
\qquad
\sigma_*=\rho_0(1-\rho_0)>0,
$$
and let $\delta_0=\delta_0(p,q)>0$ be the constant in
Lemma~\ref{lem:star-alignment}. Let
$C_{\mathrm{sd}}=C_{\mathrm{sd}}(p,q)>0$ be the constant in
Proposition~\ref{prop:higher-degree-estimate}, and define
$$
C_*=
\max\left\{
1,\frac{3201C_{\mathrm{sd}}}{\sigma_*}
\right\}.
$$
Finally, set
$$
\varepsilon_0
=
\min\left\{
\frac12,
\frac{\rho_0}{2C_*},
\frac{\delta_0}{C_*}
\right\}.
$$
All these constants depend only on $p$ and $q$.

We first consider the case
$
0<\varepsilon<\varepsilon_0.
$
Since $\varepsilon_0\leq1/2$,
applying Proposition~\ref{prop:higher-degree-estimate} to
$\widetilde{\mathcal A}$ and $\widetilde{\mathcal B}$, gives
$$
W_{>1}^{\mu_1}(1_{\widetilde{\mathcal A}})
+
W_{>1}^{\mu_2}(1_{\widetilde{\mathcal B}})
\leq
C_{\mathrm{sd}}\varepsilon.
$$
Moreover,
Lemma~\ref{lem:variance-lower-bound} gives
$$
\operatorname{Var}_{\mu_1}(1_{\widetilde{\mathcal A}})
\geq
\sigma_*,
\qquad
\operatorname{Var}_{\mu_2}(1_{\widetilde{\mathcal B}})
\geq
\sigma_*.
$$
We may apply
Lemma~\ref{lem:one-coordinate-approximation} separately under
$\mu_1$ and $\mu_2$. It gives families
$$
\mathcal H_i
\in
\{\varnothing,2^{[n]}\}
\cup
\{\mathcal S_j,\overline{\mathcal S}_j:j\in[n]\},
\qquad i=1,2,
$$
such that
$$
\begin{aligned}
\mu_1(\widetilde{\mathcal A}\triangle\mathcal H_1)
+
\mu_2(\widetilde{\mathcal B}\triangle\mathcal H_2)
&\leq
\frac{3201}
{\operatorname{Var}_{\mu_1}(1_{\widetilde{\mathcal A}})}
W_{>1}^{\mu_1}(1_{\widetilde{\mathcal A}})
+
\frac{3201}
{\operatorname{Var}_{\mu_2}(1_{\widetilde{\mathcal B}})}
W_{>1}^{\mu_2}(1_{\widetilde{\mathcal B}})\\
&\leq
\frac{3201C_{\mathrm{sd}}}{\sigma_*}\varepsilon\leq
C_*\varepsilon.
\end{aligned}
$$
Set
$
\delta=C_*\varepsilon<C_*\varepsilon_0.
$
The definition of $\varepsilon_0$ gives
$$
0<\delta<\delta_0,
\qquad
2\delta<\rho_0.
$$
Since both terms in the preceding sum are nonnegative, we obtain
$$
\mu_1(\widetilde{\mathcal A}\triangle\mathcal H_1)
\leq\delta,
\qquad
\mu_2(\widetilde{\mathcal B}\triangle\mathcal H_2)
\leq\delta.
$$
By Lemma~\ref{lem:variance-lower-bound}, we have
\begin{align}
\mu_1(\widetilde{\mathcal A}\triangle\varnothing)
=
\mu_1(\widetilde{\mathcal A})
\geq
\rho_0,\qquad
\mu_1(\widetilde{\mathcal A}\triangle2^{[n]})
=
1-\mu_1(\widetilde{\mathcal A})
\geq
\rho_0.\label{f36}
\end{align}
The corresponding inequalities also hold for
$\widetilde{\mathcal B}$ under $\mu_2$. 

We now exclude the two constant possibilities. If
$\mathcal H_1=\varnothing$, then
$$
\mu_1(\widetilde{\mathcal A}\triangle\mathcal H_1)
=
\mu_1(\widetilde{\mathcal A})
\geq
\rho_0,
$$
contradicting
$
\mu_1(\widetilde{\mathcal A}\triangle\mathcal H_1)
\leq
\delta
<
\rho_0.
$
If $\mathcal H_1=2^{[n]}$, then
$$
\mu_1(\widetilde{\mathcal A}\triangle\mathcal H_1)
=
1-\mu_1(\widetilde{\mathcal A})
\geq
\rho_0,
$$
which gives the same contradiction. Therefore,
$
\mathcal H_1\notin\{\varnothing,2^{[n]}\}.
$
Applying the same argument to $\widetilde{\mathcal B}$ under
$\mu_2$, we also obtain
$
\mathcal H_2\notin\{\varnothing,2^{[n]}\}.
$

Suppose that
$
\mathcal H_1=\overline{\mathcal S}_a
$
for some $a\in[n]$. Since $\widetilde{\mathcal A}$ is increasing and $\mu_1\bigl(\widetilde{\mathcal A}\triangle
\overline{\mathcal S}_a\bigr)\leq \delta$,
Lemma~\ref{lem:exclude-antistar} gives
$$
\begin{aligned}
\mu_1(\widetilde{\mathcal A}\triangle2^{[n]})
\leq
2\delta<
\rho_0,
\end{aligned}
$$
a contradiction with \eqref{f36}.
Thus, $\mathcal H_1$ cannot be of the form
$\overline{\mathcal S}_a$. The same argument excludes this
alternative for $\mathcal H_2$. Consequently, there exist
$a,b\in[n]$ such that
$$
\mathcal H_1=\mathcal S_a,
\qquad
\mathcal H_2=\mathcal S_b.
$$
It follows that
$$
\mu_1(\widetilde{\mathcal A}\triangle\mathcal S_a)
\leq
\delta,
\qquad
\mu_2(\widetilde{\mathcal B}\triangle\mathcal S_b)
\leq
\delta.
$$

Since $C_*\geq1$, we have $\delta=C_*\varepsilon\geq\varepsilon$. Hence,
$$
\mu_1(\widetilde{\mathcal A})
\mu_2(\widetilde{\mathcal B})
\geq
(1-\varepsilon)pq
\geq
(1-\delta)pq.
$$
Applying Lemma~\ref{lem:star-alignment} to
$\widetilde{\mathcal A}$ and $\widetilde{\mathcal B}$  gives
$a=b$. Denote their common value by $j$.
Since $\mathcal A\subseteq\widetilde{\mathcal A}$, we have
$
\mathcal A\setminus\mathcal S_j
\subseteq
\widetilde{\mathcal A}\setminus\mathcal S_j
\subseteq
\widetilde{\mathcal A}\triangle\mathcal S_j.
$
Moreover, if
$X\in\mathcal S_j\setminus\mathcal A$, then either
$X\in\widetilde{\mathcal A}\setminus\mathcal A$, or
$X\in\mathcal S_j\setminus\widetilde{\mathcal A}$. In the latter
case,
$
X\in
\widetilde{\mathcal A}\triangle\mathcal S_j.
$
Consequently,
$$
\mathcal A\triangle\mathcal S_j
\subseteq
(\widetilde{\mathcal A}\triangle\mathcal S_j)
\cup
(\widetilde{\mathcal A}\setminus\mathcal A).
$$
Combining this with \eqref{f35}, we obtain
$$
\begin{aligned}
\mu_1(\mathcal A\triangle\mathcal S_j)
\leq
\mu_1(\widetilde{\mathcal A}\triangle\mathcal S_j)
+
\mu_1(\widetilde{\mathcal A}\setminus\mathcal A)\leq
\delta+\varepsilon=
(C_*+1)\varepsilon.
\end{aligned}
$$
Applying the same argument to
$\mathcal B\subseteq\widetilde{\mathcal B}$, we obtain
$$
\mu_2(\mathcal B\triangle\mathcal S_j)
\leq
\mu_2(\widetilde{\mathcal B}\triangle\mathcal S_j)
+
\mu_2(\widetilde{\mathcal B}\setminus\mathcal B)
\leq
(C_*+1)\varepsilon.
$$

Finally, define
$$
c(p,q)
=
\max\left\{
C_*+1,\frac{1}{\varepsilon_0}
\right\}.
$$
Since $p=p_1$ and $q=q_1$, this constant depends only on
$p_1$ and $q_1$. The preceding argument proves the desired
conclusion when $0<\varepsilon<\varepsilon_0$.

If $\varepsilon\geq\varepsilon_0$, choose $j=1$. Since every
symmetric-difference measure is at most one,
$$
\max\left\{
\mu_1(\mathcal A\triangle\mathcal S_1),
\mu_2(\mathcal B\triangle\mathcal S_1)
\right\}
\leq
1
\leq
c(p,q)\varepsilon_0\leq c(p,q)\varepsilon.
$$
This completes the proof.
\end{proof}

\section{Concluding remarks}\label{se4}

Theorem~\ref{thm:main-product} determines the exact maximum of
$
\mu_{\mathbf p}(\mathcal A)
\mu_{\mathbf q}(\mathcal B)
$
for cross-intersecting families when
$\mathbf p,\mathbf q\in(0,1/2]^n$, and it characterizes all pairs
attaining this maximum. More precisely, let
$
M=\max_{i\in[n]}p_iq_i.
$
If $M<1/4$, then equality holds precisely when
$
\mathcal A=\mathcal B=\mathcal S_j
$
for some coordinate $j$ satisfying $p_jq_j=M$. If $M=1/4$, then
the theorem describes all additional equality cases in terms of
increasing families of size $2^{|I|-1}$ on
$
I=\{i\in[n]:p_i=q_i=1/2\}.
$
Thus, throughout the range
$\mathbf p,\mathbf q\in(0,1/2]^n$, both the sharp product bound and
all extremal configurations are completely determined.

Theorem~\ref{thm:main-stability} complements this exact result with a
quantitative stability statement in the strict range
$\mathbf p,\mathbf q\in(0,1/2)^n$. Suppose that 
$
p_1=\max_{i\in[n]}p_i
$
and 
$
q_1=\max_{i\in[n]}q_i.
$
If
$
\mu_{\mathbf p}(\mathcal A)
\mu_{\mathbf q}(\mathcal B)
\geq
(1-\varepsilon)p_1q_1,
$
then there is a coordinate $j\in[n]$ such that
$$
\max\left\{
\mu_{\mathbf p}(\mathcal A\triangle\mathcal S_j),
\mu_{\mathbf q}(\mathcal B\triangle\mathcal S_j)
\right\}
\leq
c(p_1,q_1)\varepsilon.
$$
Hence, every near-extremal pair is quantitatively close, under the two
respective measures, to the same star.

The preceding results leave open what happens when coordinate
probabilities are allowed to exceed $1/2$. The restriction in
Theorem~\ref{thm:main-product} cannot simply be removed. Indeed,
Remark~\ref{rem1} gives probability vectors with a coordinate larger
than $1/2$ for which a non-star cross-intersecting pair satisfies
$
\mu_{\mathbf p}(\mathcal A)
\mu_{\mathbf q}(\mathcal B)
>
\max_{i\in[n]}p_iq_i.
$
On the other hand, for arbitrary probability vectors
$\mathbf p,\mathbf q\in(0,1)^n$, a pair of common stars always satisfies
$
\mu_{\mathbf p}(\mathcal S_j)
\mu_{\mathbf q}(\mathcal S_j)
=
p_jq_j.
$
Choosing a coordinate $j$ that maximizes $p_jq_j$ therefore gives a
cross-intersecting pair with product
$
\max_{i\in[n]}p_iq_i.
$
The question is to determine exactly when no other cross-intersecting
pair can have a larger product.

\begin{problem}
Let $n\geq2$. Characterize all pairs of probability vectors
$
\mathbf p,\mathbf q\in(0,1)^n
$
for which every pair of cross-intersecting families
$\mathcal A,\mathcal B\subseteq2^{[n]}$ satisfies
$$
\mu_{\mathbf p}(\mathcal A)
\mu_{\mathbf q}(\mathcal B)
\leq
\max_{i\in[n]}p_iq_i.
$$
\end{problem}

\section*{Declaration of competing interest}
The authors declare that they have no competing interests.

\section*{Data availability}
No data was used for the research described in the article.

\section*{Acknowledgments}
The authors used AI tools only for early-stage exploration. They independently derived and confirmed all proofs and conclusions.

\end{document}